\documentclass[12pt,a4paper,reqno]{amsart}
\usepackage{graphicx}
\usepackage{amssymb}
\usepackage{cite}
\usepackage{amsmath}
\usepackage{latexsym}
\usepackage{amscd}
\usepackage{tikz}
\usepackage{amsthm}
\usepackage{mathrsfs}
\usepackage{url}
\usepackage[utf8]{inputenc}
\usepackage[english]{babel}
\usepackage{amsfonts}
\usepackage{mathtools}
\usepackage[margin=1.2in]{geometry}
\usepackage{natbib}
\usepackage{algorithm}
\usepackage{algpseudocode}
\usepackage{listings}
\usepackage{xcolor}
\usepackage{caption}
\usepackage{subcaption}
\usepackage{tabularx}
\emergencystretch=\maxdimen
\usepackage{listings}
\usepackage{xcolor}

\definecolor{codegreen}{rgb}{0,0.6,0}
\definecolor{codegray}{rgb}{0.5,0.5,0.5}
\definecolor{codepurple}{rgb}{0.58,0,0.82}
\definecolor{backcolour}{rgb}{0.95,0.95,0.92}

\lstdefinestyle{mystyle}{
	backgroundcolor=\color{backcolour},   
	commentstyle=\color{codegreen},
	keywordstyle=\color{magenta},
	numberstyle=\tiny\color{codegray},
	stringstyle=\color{codepurple},
	basicstyle=\ttfamily\footnotesize,
	breakatwhitespace=false,         
	breaklines=true,                 
	captionpos=b,                    
	keepspaces=true,                 
	numbers=left,                    
	numbersep=5pt,                  
	showspaces=false,                
	showstringspaces=false,
	showtabs=false,                  
	tabsize=2
}

\usepackage{hyperref}
\hypersetup{
	colorlinks=true,
	linkcolor=blue,
	filecolor=magenta,      
	urlcolor=cyan,
}

\definecolor{shadecolor}{rgb}{1,0.8,0}

\usepackage{setspace} % Added by Ethan for code

\numberwithin{equation}{section}
\DeclarePairedDelimiter\floor{\lfloor}{\rfloor}

\newtheorem{theorem}{Theorem}[section]
\newtheorem{lemma}[theorem]{Lemma}
\newtheorem{proposition}[theorem]{Proposition}

\newtheorem{corollary}[theorem]{Corollary}
\newtheorem{conjecture}[theorem]{Conjecture}

\newtheorem*{Main Result}{Main Result}

\theoremstyle{remark}

\theoremstyle{definition}
\newtheorem{definition}[theorem]{Definition}

\theoremstyle{theorem}
\newtheorem{Assumption}[theorem]{Assumption}
\theoremstyle{theorem }

\theoremstyle{remark}

\newcommand*{\emaill}[1]{\texttt{#1}}

\numberwithin{equation}{section}

\title{Online Regenerative Learning}

\begin{document}
	
	\author[Owen Shen]{
		Owen Shen
		\\
		\vspace{0.1cm}
		\\
		Department of Mathematics
		\vspace{0.1cm}
		\\
		\MakeLowercase{
			\emaill{owenshen@stanford.edu}}\\
		\vspace{0.2cm}
		Stanford University}
	
	\begin{abstract}
		We study a type of Online Linear Programming (OLP) problem that maximizes the objective function with stochastic inputs. The performance of various algorithms that analyze this type of OLP is well studied when the stochastic inputs follow some i.i.d distribution. The two central questions to ask are: (i) can the algorithms achieve the same efficiency if the stochastic inputs are not i.i.d but still stationary, and (ii) how can we modify our algorithms if we know the stochastic inputs are trendy, hence not stationary. We answer the first question by analyzing a regenerative type of input and show the regrets of two popular algorithms are bounded by the same orders as their i.i.d counterparts. We discuss the second question in the context of linearly growing inputs and propose a trend-adaptive algorithm. We provide numerical simulations to illustrate the performance of our algorithms under both regenerative and trendy inputs. 
		
	\end{abstract}
	\maketitle
	
	\tableofcontents

	\section{Introduction}
	Online Linear Programming belongs to an essential type of sequential decision making process. The formulation of Online Linear Programming can be understood to optimize the profit of selling a set of products to different customers each of whom appears sequentially with the amount of products intended for purchase and a bid price. The seller must make an irrevocable decision at the time each customer appears. Mathematically, when we have $m$ different products with storage of $b_i$ for $i$th product, we hope to maximize $$
	\begin{array}{c}
		\underset{x}{\operatorname{maximize}} \sum_{j=1}^{n} r_{j} x_{j} \\
		\text { subject to } \sum_{j=1}^{n} a_{i j} x_{j} \leq b_{i}, \forall i=1,2, \cdots, m \\
		0 \leq x \leq 1, \forall j=1,2, \cdots, n
	\end{array}
	$$
	where $a_{ij}$ is $jth$ customer's wanted amount for $j$th products, $r_j$ is her bid price, $x_j$ is the decision the seller makes whether to fulfill (either completely or partially) her order, and $n$ is the total selling period. 
	
	Such formulation is widely applied in the fields of revenue management (\cite{12}), advertisement deliveries (\cite{11}, \cite{10}), and resource allocation (\cite{8} , \cite{16}). The performance of various algorithms to solve this type of OLP is well studied when the stochastic inputs are i.i.d (\cite{5}, \cite{2}). There are also considerable progresses made to analyze non-i.i.d inputs: \cite{21} studies the adversarial stochastic input model, and  \cite{22} , \cite{3}, and \cite{23} study the permutation model. In this paper, we focus on analyzing the performance of the algorithms proposed in \cite{5} with the regenerative model. 
	
	In \cite{5}, three algorithms \ref{alg:1}, \ref{alg:2}, and \ref{alg:3} are analyzed in the i.i.d model. Their regrets are proved to be bounded by $O(\sqrt{n})$, $O(\sqrt{n}\log n)$, and $O(\log n \log \log n)$ respectively. Recent developments based on \cite{5} include \cite{17}, \cite{15}, and \cite{18} that improve revenue management by adopting the dual-policy based algorithms; \cite{16} that discusses the performance when the resource capacity does not scale up linearly with $n$; and \cite{13} that improves the matching problem in the discrete form, widely applied in kidney exchange platforms and carpooling platforms. Hence, suppose we can further generalize the results for those three algorithms, we may find a handful of promising applications. 
	
	The first central goal of this paper is to analyze algorithm \ref{alg:1} and algorithm \ref{alg:2} using regenerative data so defined in the next section. Intuitively, the regenerative process can be thought of as a process that can be decomposed into i.i.d cycles with randomized length. Hence, such a feature can well model certain local dependencies and periodic behaviors of data. Some well-known regenerative models include certain types of Markov Chain, which is a popular model for financial modeling (\cite{24}). Another popular example of regenerative data is the inventory problem (\cite{25}). To achieve our goal, we have the following steps. 
	
	 In Section Two we analyze the regenerative process and establish a concentration result, the first main result of this paper:
	 	\begin{theorem} (Exponential Bound for Regenerative Processes with Bounded Time) Suppose $|f(X)|$ is almost surely bounded by $M$, and $T_{0}, \tau_{i}$ are almost surely bounded by $T$, then we have the following concentration bound: suppose $t>T M K / \epsilon$ for some large $K$, then
	 	$$
	 	\mathbb{P}\left(\frac{1}{t}\left|\int_{0}^{t}(f(X(s))-\alpha) d s\right|>\epsilon\right) \leq 2 \exp \left(-\frac{2 \epsilon^{2}(K-2)^{2}}{K^{2} \lambda M^{2} T^{2}} t\right)+\epsilon(\delta, K)
	 	$$
	 	where
	 	$$
	 	\epsilon(\delta, K)=2 \exp \left(-\frac{\delta^{2} t}{(\lambda-\delta)^{2} \lambda^{2} T^{2}}\right)+\exp \left(\left(2 \delta M-\frac{K-2}{K} \epsilon\right) t\right)
	 	$$
	 	
	 	$$
	 	\alpha=\frac{E \int_{T_{0}}^{T_{1}} f(X(s)) d s}{E \tau_{1}}
	 	$$
	 	and $\delta, K $ are free parameters.
	 	
	 \end{theorem}
	 	
	 This Hoeffding-style inequality is critical in our algorithms analysis, because the regret defined in Section Four is essential a minimax problem on distributional optimization, and a Hoeffding-style inequality only requires a certain upper bound on the data. 
	 
	 In Section Three, we review the OLP models proposed by \cite{5} and extend the results to regenerative models. Specifically, we use the above concentration result to derive a Regenerative Dual Convergence essential for the regret analysis as the second main result:
 	\begin{theorem}
 	(Dual Convergence Theorem for Regenerative Processes ) 
 	For a regenerative price process $r_i$, and under regularity conditions \ref{ass 1*},\ref{ass 2*},\ref{ass 3*}, there exists a constant $C$ such that
 	$$
 	\mathbb{E}\left[\left\|\boldsymbol{p}_{n}^{*}-\boldsymbol{p}^{*}\right\|_{2}^{2}\right] \leq \frac{C m \log m \log \log n}{n}
 	$$
 	holds for all $n \geq \max \{m, 3\}, m \geq 2$, and distribution $\mathcal{P}$ that satisfies those assumptions. Additionally,
 	$$
 	\mathbb{E}\left[\left\|\boldsymbol{p}_{n}^{*}-\boldsymbol{p}^{*}\right\|_{2}\right] \leq C \sqrt{\frac{m \log m \log \log n}{n}}
 	$$ \label{First Main Result}
 \end{theorem}
 Since the algorithms we are interested in analyzing belong to the dual-policy algorithms, the convergence in the dual paves the way to regret analysis for dual-policy algorithms. 
 
 In Section Four, we discuss the efficiency of algorithm \ref{alg:1} and \ref{alg:2}, and present them as our third and fourth main results:
 
\begin{theorem}(Regenerative Regret for Algorithm \ref{alg:1})
	With the online policy $\boldsymbol{\pi}_{1}$ specified by Algorithm 1 with regenerative data,
	$$
	\Delta_{n}\left(\boldsymbol{\pi}_{1}\right) \leq O(\sqrt{n})
	$$
\end{theorem}
\begin{theorem}(Regenerative Regret for Algorithm \ref{alg:2})
	With the online policy $\boldsymbol{\pi}_{2}$ specified by Algorithm 2 with regenerative data,
	$$
	\Delta_{n}\left(\pi_{2}\right) \leq O(\sqrt{n} \log n)
	$$ 
\end{theorem}

 In Section Five, we provide some numerical simulations, discuss the source of regrets, and use the numerical results to analyze two small modifications that can potentially improve the algorithms. As a result, we answer our first question that "can the algorithms achieve the same efficiency if the stochastic inputs are not i.i.d but still stationary" by extending the theorems in the context of the regenerative model. 
 
 In Section Six, we address the second question that "how can we modify our algorithms if we know the stochastic inputs are trendy, hence not stationary". We provide some candidates algorithms and demonstrate their efficiency through numerical simulations. Hence, we leave the second question open and discuss the future directions.

	\section{Regenerative Processes and Convergence Rate}
	A stochastic process $\mathbf{X}=\{X(t): t \geq 0\}$ is called a regenerative process, first defined by \cite{26}, if there exists a sequence of stopping time $0 \leq T_{0}<T_{1}<T_{2}<\ldots$ such that each post- $T_{k}$ processes $\left\{X\left(T_{k}+t\right): t \geq 0\right\}$ form an i.i.d sequence of processes. The interval $T_j-T_{j-1}=\tau_{j-1}$ is a sequence of i.i.d random time. Intuitively, the process on time interval $[0,t]$ is split into i.i.d cycles except $[0,T_1]$, the interval before the first regeneration, and $[T_n,t]$, the interval between the final regeneration and the termination of the process.
	
	Since a regenerative process resembles an i.i.d sequence of random variables but exhibits many desirable traits such as periodic behaviors, it is natural to study the Law of Large Number of such process \cite{26}: 
	\begin{proposition}(Law of Large Number for Regenerative Processes)
		Suppose that 
		$$\int_{T_i}^{T_{i+1}}|f(X(s))|ds $$
		is integrable, then 
		$$
		\lim _{t \rightarrow \infty} \frac{1}{t} \int_{0}^{t} f(X(s)) d s=\frac{\mathbb{E}[R]}{\mathbb{E}[\tau]}
		$$
		where $\tau$ is the length of the first cycle and $R=\int_{T_0}^{T_0\tau} f(X(s)) d s$ is the value over the first full cycle.
	\end{proposition}
	For convenience, we normally denote $\mathbb{E}(\tau_1)$ as $\lambda$, the regenerative rate. Intuitively, the higher the regenerative rate, the more the process behaves like a standard i.i.d process. Similarly, there is a Central Limit Theorem for such process \cite{26}: 
	\begin{proposition}(Central Limit Theorem for Regenerative Processes)
		Suppose that 
		$$\int_{T_i}^{T_{i+1}}(f(X(s)))^2ds $$
		is integrable, and $T_0$ and
		$$\int_{0}^{T_0}|f(X(s))|ds $$
		are finite almost surely, then 
		$$
		\lim _{t \rightarrow \infty} \frac{1}{t^{1/2}}\left(\frac{1}{t} \int_{0}^{t} f(X(s)) -\mathbb{E}[R]\right) d s \Rightarrow \sigma N(0,1)
		$$
		where $N(0,1)$ is the Standard Normal Distribution and $\sigma$ is the normalized variance:
		$$\sigma^2=\frac{1}{\mathbb{E}(\tau_1)}Var\left(\int_{T_i}^{T_{i+1}}f(X(s))ds\right). $$
	\end{proposition}
	Those two propositions would be sufficient to analyze the limiting behaviors and approximation for the regenerative processes, provided that $t$ is large. However, those propositions say very little about the rate of convergence, a crucial element in the application. It is therefore the goal of this section to fulfill the missing piece by introducing the Regenerative version of one of the most commonly used propositions on the i.i.d model that bounds the convergence rate: Hoeffding's inequality. 
	\begin{proposition}
		(Hoeffding's inequality for Bounded Variables). Let $Z_{1}, \ldots, Z_{n}$ be independent bounded random variables with $Z_{i} \in[a, b]$ for all $i$, where $-\infty<a \leq b<\infty$. Then
		$$
		\mathbb{P}\left(\frac{1}{n} \sum_{i=1}^{n}\left(Z_{i}-\mathbb{E}\left[Z_{i}\right]\right) \geq t\right) \leq \exp \left(-\frac{2 n t^{2}}{(b-a)^{2}}\right)
		$$
		and
		$$
		\mathbb{P}\left(\frac{1}{n} \sum_{i=1}^{n}\left(Z_{i}-\mathbb{E}\left[Z_{i}\right]\right) \leq-t\right) \leq \exp \left(-\frac{2 n t^{2}}{(b-a)^{2}}\right)
		$$
		for all $t \geq 0$.
	\end{proposition}
	
	One of the main reasons for the popularity of Hoeffding's inequality is that, under the i.i.d assumption, Hoeffding's inequality would give an exponentially decay upper bound on the convergence rate. We will show a similar result can be established for the Regenerative Processes as our first main result:
	
	\begin{theorem} (Exponential Bound for Regenerative Processes with Bounded Time)
		Suppose $|f(X)|$ is almost surely bounded by $M$, and $T_0, \tau_i$ are almost surely bounded by $T$, then we have the following concentration bound: suppose $t>TMK/\epsilon$ for some large $K$, then
		$$
		\mathbb{P}\left(\frac{1}{t}\left| \int_{0}^{t}(f(X(s))-\alpha)ds \right|>{\epsilon}\right)\leq
		2\exp\left({-\frac{2\epsilon^2 (K-2)^2}{K^2\lambda M^2T^2}t}\right)+\epsilon(\delta,K)
		$$
		where 
		$$\epsilon(\delta, K)=2 \exp \left(-\frac{\delta^{2} t}{(\lambda-\delta)^{2} \lambda^{2} T^{2}}\right)+\exp \left(\left(2 \delta M-\frac{K-2}{K}\epsilon\right) t\right) $$
		and  $$
		\alpha=\frac{E \int_{T_{0}}^{T_{1}} f(X(s)) d s}{E \tau_{1}},
		$$
		and $\delta, K$ are free parameters. 
	\end{theorem}
	Let discuss what this $\delta$ stands for. The upper bound is partitioned into a form that is almost identical to Hoeffding's inequality except normalized by the regenerative rate $\lambda$. The error probability is partitioned into two parts; the former stands for the probability of the sample average epsilon away from the true mean, conditioned on the event that the true number of regeneration differs from the expected number $\lambda t$ less than $\delta t$; the latter part is the probability that the true number of regeneration differs from the expected number $\lambda t$ more than $\delta t$. We have checked that the sum of those two parts forms a convex function in $\delta$, so one can easily numerically approximate the optimal $\delta$ given reasonable belief about the bounds $M$ and $T$, the regeneration rate $\lambda$, and the error tolerance $\epsilon$. 

	The proof is inspired by the central limit theorem proof in \cite{26}. This concentration result assumes maximum regenerative time, which may not be realistic in practice. In a uniformly ergodic Markov model, for example, regeneration can happen in geometric time. Related works include \cite{19} that establishes a Heoffding inequality without assuming bounded regenerative time but on a finite state space Markov chain; and \cite{20} that also establishes a Hoeffding inequality of a different form. 
	
	One may result in a Corollary if we have further information on the interval from which $|f(X(s))|$ lies:
	\begin{corollary} (Exponential Bound for Regenerative Processes with Bounded Time)
		Suppose $f(X)\in (a,b)$ and $T_0, \tau_i$ are almost surely bounded by $T$, then we have the following concentration bound: suppose $t>TMK/\epsilon$ for some large $K$, then, let $M=\max\{|a|,|b|\}$:
		$$
		\mathbb{P}\left(\frac{1}{t} \int_{0}^{t}(f(X(s))-\alpha)ds < {\epsilon}\right)\leq 
		\exp\left({-\frac{2\epsilon^2 (K-2)^2}{K^2\lambda (b-a)^2T^2}t}\right)+\epsilon(\delta,K)
		$$
		and 
		$$
		\mathbb{P}\left(\frac{1}{t} \int_{0}^{t}(f(X(s))-\alpha)ds > {\epsilon}\right)\leq 
		\exp\left({-\frac{2\epsilon^2 (K-2)^2}{K^2\lambda (b-a)^2T^2}t}\right)+\epsilon(\delta,K)
		$$
		where 
		$$\epsilon(\delta, K)=2 \exp \left(-\frac{\delta^{2} t}{(\lambda-\delta)^{2} \lambda^{2} T^{2}}\right)+\exp \left(\left(2 \delta M-\frac{K-2}{K}\epsilon\right) t\right) $$
		and  $$
		\alpha=\frac{E \int_{T_{0}}^{T_{1}} f(X(s)) d s}{E \tau_{1}},
		$$
		and $\delta, K$ are free parameters. 
	\end{corollary}

	\section{ Online Linear Programming}
	\subsection{Backgrounds}
	Online Linear Programming belongs to the sequential decision making problem: In mathematics, Online Linear Programming is concerned with solving the following linear programming in the presence of incomplete information:
	\begin{equation}
		\tag{3.1}\label{eq:3.1}
		\underset{x}{\operatorname{maximize}}  \sum_{j=1}^{n} r_{j} x_{j} 
	\end{equation}
	$$
	\begin{aligned}
		\text { subject to } & \sum_{j=1}^{n} a_{i j} x_{j} \leq b_{i}, \forall i=1,2, \cdots, m \\
		& 0 \leq x \leq 1, \forall j=1,2, \cdots, n
	\end{aligned}
	$$
	where $r=\left(r_{1}, r_{2}, \cdots, r_{n}\right)^{T}$ can be interpreted as the price vector such that the goal is to find the allocation of decision vector $x_i$ such that the total profit is maximized. In this setting, $a_{ij}$ is the required $i$th resource to fulfill the $j$th decision while $b=\left(b_{1}, b_{2}, \cdots, b_{m}\right)^{T}$ is the resource capacity constraint. In this section, we will assume $(a_i,r_i)$ follows some i.i.d distribution. Such assumption is commonly used when analyzing OLP--\cite{1}, \cite{2}, and \cite{3}.
	 The theoretical foundation for the i.i.d case is first established in \cite{5}. Therefore, we are interested in extending the main result of \cite{5}, which shows the dual multiplier, or the shadow price, of the online problem converges to that of the off-line.  
	
	To analyze this problem, we consider the Dual of this system:
	$$
	\begin{aligned}
		\min & \sum_{i=1}^{m} b_{i} p_{i}+\sum_{j=1}^{n} y_{j} \\
		\text { s.t. } & \sum_{i=1}^{m} a_{i j} p_{i}+y_{j} \geq r_{j}, \quad j=1, \ldots, n \\
		& p_{i}, y_{j} \geq 0 \text { for all } i, j.
	\end{aligned}
	$$
	Here the decision variables are $\boldsymbol{p}=\left(p_{1}, \ldots, p_{m}\right)^{\top}$ and $\boldsymbol{y}=\left(y_{1}, \ldots, y_{n}\right)^{\top}$. 
	Let $\left(\boldsymbol{p}_{n}^{*}, \boldsymbol{y}_{n}^{*}\right)$ be an optimal solution for the dual LP. From the complementary slackness condition, we know the primal optimal solution satisfies
	$$
	x_{j}^{*}=\left\{\begin{array}{ll}
		1, & r_{j}>\boldsymbol{a}_{j}^{\top} \boldsymbol{p}_{n}^{*} \\
		0, & r_{j}<\boldsymbol{a}_{j}^{\top} \boldsymbol{p}_{n}^{*}
	\end{array}\right.
	$$
	Therefore, if we are able to solve the Dual system, we know what the decision vector should be. In fact, this complementary slackness condition would give us discrete solutions if the bidding price is distinct from the $\boldsymbol{p}_{n}^{*}$, which is interpreted as the Shadow Price. If $r_{j}=\boldsymbol{a}_{j}^{\top} \boldsymbol{p}_{n}^{*}$, the optimal solution $x_{j}^{*}$ may take on non-integer values. In the case when only integer solution is allowed, we can view the action to be probabilistic, whereas integer values represent the deterministic action. Or we may accept or reject the order, depending on how conservative we want to be about the resource. Since we know $y_i\geq 0$, an equivalent way to write this system is 
	$$
	\begin{array}{l}
		\min \sum_{i=1}^{m} b_{i} p_{i}+\sum_{j=1}^{n}\left(r_{j}-\sum_{i=1}^{m} a_{i j} p_{i}\right)^{+} \\
		\text {s.t. } p_{i} \geq 0, \quad i=1, \ldots, m.
	\end{array}
	$$
	As a result, this optimization problem resembles a stochastic problem: 
	\begin{equation}
		\min f_{n}(\boldsymbol{p}):=\sum_{i=1}^{m} d_{i} p_{i}+\frac{1}{n} \sum_{j=1}^{n}\left(r_{j}-\sum_{i=1}^{m} a_{i j} p_{i}\right)^{+} \\
		\text {s.t. } p_{i} \geq 0, \quad i=1, \ldots, m.
		\label{eq:1}
	\end{equation}
	where $d_i=b_i/n$. This is similar to take the expectation with respect to $r_j$ and $a_{ij}$: 
	\begin{equation} 
		\min f(\boldsymbol{p}):=\boldsymbol{d}^{\top} \boldsymbol{p}+\mathbb{E}\left[\left(r-\boldsymbol{a}^{\top} \boldsymbol{p}\right)^{+}\right] \\
		\text {s.t. } \quad \boldsymbol{p} \geq \mathbf{0},
		\label{eq:2}
	\end{equation}
	
	such that $$
	\mathbb{E} f_{n}(\boldsymbol{p})=f(\boldsymbol{p}).
	$$
	Therefore, given the distribution of $(r,a)$, we can find the expected minimum of $f_n(p)$ by evaluating the function $f(p)$. The convergence problem is to show the optimal solution to system \eqref{eq:1}, denoted as $p^*_n$ will converge to the optimal solution to \eqref{eq:2}, denoted as $p^*$. This convergence can be viewed as an extension of the Law of Large Numbers in the dual space. 
	
	To have a reasonable convergence result for the stochastic optimization, we first need some assumptions on the distribution of $(r,a)$:
	\begin{Assumption}[Boundedness and Linear Growth Capacity].\\
		\noindent (a) $\left\{\left(r_{j}, \boldsymbol{a}_{j}\right)\right\}_{j=1}^{n}$ are generated i.i.d. from distribution $\mathcal{P}$.\\
		\noindent (b) There exist constants $\bar{r}, \bar{a}>0$ such that $\left|r_{j}\right| \leq \bar{r}$ and $\left\|\boldsymbol{a}_{j}\right\|_{2} \leq \bar{a}$ almost surely.\\
		\noindent (c) $d_{i}=b_{i} / n \in(\underline{d}, \bar{d})$ for $\underline{d}, \bar{d}>0, i=1, \ldots$, m. Denote $\Omega_{d}=\bigotimes_{i=1}^{m}(\underline{d}, \bar{d})$\\
		\noindent (d) $n>m$.
		\label{ass:1}
	\end{Assumption}
	Roughly speaking, this assumption asserts that the incoming orders and their prices are i.i.d and bounded almost surely. Moreover, the resource constraints grow linearly so that the service level remains relatively stable. Two consequences are the almost surely bounded optimal solution and the convexity of $f_n(p),f(p)$ as discussed in Proposition 1 of \cite{5}; so it makes sense to define 
	$$
	\Omega_{p}:=\left\{\boldsymbol{p} \in \mathbb{R}^{m}: \boldsymbol{p} \geq \mathbf{0}, \boldsymbol{e}^{\top} \boldsymbol{p} \leq \frac{\bar{r}}{\underline{d}}\right\}
	$$
	where $\boldsymbol{e} \in \mathbb{R}^{m}$ is an all-one vector. We know that $\Omega_{p}$ covers all possible optimal solutions. Now we state the second assumption on the distribution of $(r,a)$:
	
	\begin{Assumption}[Non-degeneracy].\\ 
		\noindent (a) The second-order moment matrix $\boldsymbol{M}:=\mathbb{E}_{(r, \boldsymbol{a}) \sim \mathcal{P}}\left[\boldsymbol{a a}^{\top}\right]$ is positive-definite. Denote its minimum eigenvalue with $\lambda_{\min }$.\\
		\noindent (b) There exist constants $\lambda$ and $\mu$ such that if $(r, \boldsymbol{a}) \sim \mathcal{P}$,
		$$
		\left.\lambda\left|\boldsymbol{a}^{\top} \boldsymbol{p}-\boldsymbol{a}^{\top} \boldsymbol{p}^{*}\right| \leq\left|\mathbb{P}\left(r>\boldsymbol{a}^{\top} \boldsymbol{p} \mid \boldsymbol{a}\right)-\mathbb{P}\left(r>\boldsymbol{a}^{\top} \boldsymbol{p}^{*} \mid \boldsymbol{a}\right)\right| \leq \mu\left|\boldsymbol{a}^{\top} \boldsymbol{p}-\boldsymbol{a}^{\top} \boldsymbol{p}^{*}\right|\right)
		$$
		holds for any $\boldsymbol{p} \in \Omega_{p}$.\\
		\noindent (c) The optimal solution $\boldsymbol{p}^{*}$ to the stochastic optimization problem (7) satisfies $p_{i}^{*}=0$ if and only if $d_{i}-\mathbb{E}_{(r, \boldsymbol{a}) \sim \mathcal{P}}\left[a_{i} I\left(r>\boldsymbol{a}^{\top} \boldsymbol{p}^{*}\right)\right]>0$
		\label{ass:2}
	\end{Assumption}
	The second group of assumptions is called Non-degeneracy, for the first condition essentially requires the constraints matrix to be full rank; the second condition imposes a linear growth on the conditional probability so that the biding prices are reasonable; and third condition states the strict complementarity for the stochastic program. When those two assumptions are satisfied, we have the following theorem from \cite{5}: 
	
	\begin{theorem}[Dual Convergence Theorem] \label{thm:conv}  Under Assumption 3.1 and 3.2, there exists a constant $C$ such that
		$$
		\mathbb{E}\left[\left\|\boldsymbol{p}_{n}^{*}-\boldsymbol{p}^{*}\right\|_{2}^{2}\right] \leq \frac{C m \log m \log \log n}{n}
		$$
		holds for all $n \geq \max \{m, 3\}, m \geq 2$, and distribution $\mathcal{P} \in \Xi .$ Additionally,
		$$
		\mathbb{E}\left[\left\|\boldsymbol{p}_{n}^{*}-\boldsymbol{p}^{*}\right\|_{2}\right] \leq C \sqrt{\frac{m \log m \log \log n}{n}}.
		$$ 
	\end{theorem}
	This Dual Convergence Theorem is the theoretical foundation for the Online Learning Algorithms, for it provides the provable basis for the convergence efficiency. Therefore, if we can derive a similar dual convergence theorem for the regenerative processes, we provide the theoretical foundation to extend Online Learning Algorithms beyond the barrier of the i.i.d restriction.

	\subsection{Regenerative Online Linear Programming}
  In this section, we will prove the regenerative dual convergence theorem in the case where $\{a\}$ follows the i.i.d assumption, yet the proposed prices $\{r\}$ follow a regenerative process.
	
	First, let us recall the dual optimization problem we are interested in solving:
	\begin{equation}
		\min f_{n}(\boldsymbol{p}):=\sum_{i=1}^{m} d_{i} p_{i}+\frac{1}{n} \sum_{j=1}^{n}\left(r_{j}-\sum_{i=1}^{m} a_{i j} p_{i}\right)^{+} \text {s.t. } p_{i} \geq 0, \quad i=1, \ldots, m. 
	\end{equation}
	Then, we know by the law of large number of the regenerative processes, this converges to 
	\begin{equation}\label{3.4}
		\min f(\boldsymbol{p}):=\boldsymbol{d}^{\top} \boldsymbol{p}+\frac{1}{\mathbb{E}\tau_1}\mathbb{E}\sum_{i=\tau_0}^{\tau_1}\left[\left(r_i-\boldsymbol{a}^{\top} \boldsymbol{p}\right)^{+}\right] \text {s.t. } \quad \boldsymbol{p} \geq \mathbf{0} 
	\end{equation}
	Let note observe that suppose $r$ is a non-delay regenerative process, where the process regenerates itself at the initial point, and suppose further $r$ terminates exactly before the next regeneration, we would have 
	$$
	\mathbb{E} f_{n}(\boldsymbol{p})=f(\boldsymbol{p})
	$$
	Even though those two quantities do not agree in general, the difference decays exponentially. For the remainder of the section, let us assume the equivalence of Assumption 3.1 for our regenerative process:
	\begin{Assumption}[Regenerative Boundedness and Linear Growth Capacity 1*].\\
		\noindent (a) $\left\{\left( \boldsymbol{a}_{j}\right)\right\}_{j=1}^{n}$ is generated i.i.d. and $\left\{\left( \boldsymbol{r}_{j}\right)\right\}_{j=1}^{n}$ is generated as a regenerative process from distribution $\mathcal{P}_n$.\\
		\noindent (b) There exist constants $\bar{r}, \bar{a}>0$ such that $\left|r_{j}\right| \leq \bar{r}$ and $\left\|\boldsymbol{a}_{j}\right\|_{2} \leq \bar{a}$ almost surely.\\
		\noindent (c) $d_{i}=b_{i} / n \in(\underline{d}, \bar{d})$ for $\underline{d}, \bar{d}>0, i=1, \ldots$, m. Denote $\Omega_{d}=\bigotimes_{i=1}^{m}(\underline{d}, \bar{d})$\\
		\noindent (d) $n>m$.
		\label{ass 1*}
	\end{Assumption}
	Similarly, we have the boundedness on the optimal dual solution in the space  $\Omega_{p}$ and the convexity of $f_n(p)$ and $f(p)$. Then, it makes sense to assume,
	\begin{Assumption}[Regenerative Non-degeneracy 2*].\\ 
		\noindent (a) The second-order moment matrix $\boldsymbol{M}:=\mathbb{E}_{(r, \boldsymbol{a}) \sim \mathcal{P}_n}\left[\boldsymbol{a a}^{\top}\right]$ is positive-definite for all $n$. Denote its minimum eigenvalue with $\lambda_{\min }$.\\
		\noindent (b) There exist constants $\lambda$ and $\mu$ such that if $(r, \boldsymbol{a}) \sim \mathcal{P}_n$,
		$$
		\left.\lambda\left|\boldsymbol{a}^{\top} \boldsymbol{p}-\boldsymbol{a}^{\top} \boldsymbol{p}^{*}\right| \leq\left|\mathbb{P}\left(r>\boldsymbol{a}^{\top} \boldsymbol{p} \mid \boldsymbol{a}\right)-\mathbb{P}\left(r>\boldsymbol{a}^{\top} \boldsymbol{p}^{*} \mid \boldsymbol{a}\right)\right| \leq \mu\left|\boldsymbol{a}^{\top} \boldsymbol{p}-\boldsymbol{a}^{\top} \boldsymbol{p}^{*}\right|\right)
		$$
		holds for any $\boldsymbol{p} \in \Omega_{p}$, where $\Omega_{p}$ is as defined in Assumption \ref{ass:2}.\\
		\noindent (c) The optimal solution $\boldsymbol{p}^{*}$ to the stochastic optimization problem \ref{3.4} satisfies $p_{i}^{*}=0$ if and only if $d_{i}-\mathbb{E}_{(r, \boldsymbol{a}) \sim \mathcal{P}_n}\left[a_{i} I\left(r>\boldsymbol{a}^{\top} \boldsymbol{p}^{*}\right)\right]>0 $ for all $n$. In the case where $p_i^*>0$, we call the $i$th resource binding. 
		\label{ass 2*}
	\end{Assumption}
	For simplicity, let us denote them as Assumption 1* and 2*. In addition, we assume 
	\begin{Assumption} (Bounded and Independent Regenerative Times 3*)  \\
		\noindent (a) The $\{r\}_i$ is a non-delay regenerated process with i.i.d stopping time $\tau_i$. \\
		\noindent (b) The stopping time $\tau_i$ sequence is independent of the process $\{a\}_i$ and bounded by $T$. 
		\label{ass 3*}
	\end{Assumption}
	We denote this assumption as 3*. With those three assumptions, we are able to establish the following:
	\begin{theorem}
		(Dual Convergence Theorem for Regenerative Processes ) 
		For a regenerative price process $r_i$, and under certain regularity conditions \ref{ass 1*},\ref{ass 2*},\ref{ass 3*}, there exists a constant $C$ such that
		$$
		\mathbb{E}\left[\left\|\boldsymbol{p}_{n}^{*}-\boldsymbol{p}^{*}\right\|_{2}^{2}\right] \leq \frac{C m \log m \log \log n}{n}
		$$
		holds for all $n \geq \max \{m, 3\}, m \geq 2$, and distribution $\mathcal{P} \in \Xi .$ Additionally,
		$$
		\mathbb{E}\left[\left\|\boldsymbol{p}_{n}^{*}-\boldsymbol{p}^{*}\right\|_{2}\right] \leq C \sqrt{\frac{m \log m \log \log n}{n}}
		$$
	\end{theorem}
	Therefore, the Dual Convergence Theorem in \cite{5} is rendered as a sub-class of this more general expression. Since the dual convergence is the theoretical foundation for the OLP, this result would be the foundation that extends the power of OLP far beyond its original i.i.d constraint. Moreover, it is worth pointing out that the strategy of the proofs does not depend explicitly on the dual objective function $f(\cdot)$. Therefore, the results can be easily extended to other learning program that uses law of large number approximation as the foundation.
	
	Now we present the key steps in proving the Dual Convergence Theorem for Regenerative Processes. The main structure of this proof is the following: we first decompose the difference between the optimal ($f(p^*)$) and optimal values approximated through the sample ($f(p_n)$) into first and second order approxiations (Proposition \ref{Q1}), second we show the relation of the dual convergence rate and the convergence rate of the first and second order approximations, (Proposition \ref{Q2}), third we provide convergence rate for first and second order approximations (Proposition \ref{Q3}, Proposition \ref{Q4}), and fourth we use their convergence rates to show the convergence rate of the dual (Proposition \ref{Q5}). We use $N(t)$ to denote the number of complete regenerations before time $t$, and $T(n)$  to denote the time when $n$th cycle is completed.
	
	To derive a tractable decomposition, we borrow the strategies from \cite{5} to define a function $h: \mathbb{R}^{m} \times \mathbb{R}^{m+1} \rightarrow \mathbb{R}$
	$$
	h(\boldsymbol{p}, \boldsymbol{u}):=\sum_{i=1}^{m} d_{i} p_{i}+\left(u_{0}-\sum_{i=1}^{m} u_{i} p_{i}\right)^{+}
	$$
	and function $\phi: \mathbb{R}^{m} \times \mathbb{R}^{m+1} \rightarrow \mathbb{R}^{m}$
	$$
	\phi(\boldsymbol{p}, \boldsymbol{u}):=\frac{\partial h(\boldsymbol{p}, \boldsymbol{u})}{\partial p}=\left(d_{1}, \ldots, d_{m}\right)^{\top}-\left(u_{1}, \ldots, u_{m}\right)^{\top} \cdot I\left(u_{0}>\sum_{i=1}^{m} u_{i} p_{i}\right)
	$$
	where $\boldsymbol{u}=\left(u_{0}, u_{1}, \ldots, u_{m}\right)^{\top}$ and $\boldsymbol{p}=\left(p_{1}, \ldots, p_{m}\right)^{\top}$. From \cite{5} we know the function $\phi$ is the partial sub-gradient of the function $h$ with respect to $\boldsymbol{p}$; in particular, $\phi(\boldsymbol{p}, \boldsymbol{u})=\boldsymbol{d}$ when $u_{0}=\sum_{i=1}^{m} u_{i} p_{i}$. Then, we denote 

	 $$f(p):=\mathbb{E} \sum_{i=1}^{\tau_1}\frac{1}{\mathbb{E}\tau_1}[h(p,u_i)], u_n\sim\mathcal{P}_n$$ 
	$$\hat{f}_n(p):=\mathbb{E} [h(p,u)], u_n\sim\mathcal{P}_n$$ 
	and 
	$$\nabla  f(p):=\mathbb{E}\sum_{i=1}^{\tau_1}\frac{1}{\mathbb{E}\tau_1}[\phi(p,u_i)], u_n\sim\mathcal{P}_n.$$
	$$\nabla  \hat{f}_n(p):=\mathbb{E}[\phi(p,u)], u_n\sim\mathcal{P}_n.$$
	The key difference is that we make the differentiation between $f,\hat{f}$, for the expected value of a regenerative process does not agree with its limiting sample average in general. Such sub-gradient allows us to analyze the functions in first and second orders, an idea encapsulated in the following proposition: 
	\begin{proposition} \label{Q1}
		For any $\boldsymbol{p} \geq \mathbf{0}$ and $\lambda=\mathbb{E}\tau_1$, we have the following identity,
		$$
		f(\boldsymbol{p})-f\left(\boldsymbol{p}^{*}\right)={\nabla f\left(\boldsymbol{p}^{*}\right)\left(\boldsymbol{p}-\boldsymbol{p}^{*}\right)}+{ \mathbb{E}\lambda\sum_{i=0}^{\tau_1}\left[\int_{\boldsymbol{a}^{\top} \boldsymbol{p}}^{\boldsymbol{a}^{\top} \boldsymbol{p}^{*}}\left(I(r_i>v)-I\left(r_i>\boldsymbol{a}^{\top} \boldsymbol{p}^{*}\right)\right) d v\right]} .
		$$
		and 
		$$
		\hat{f}_n(\boldsymbol{p})-\hat{f}_n\left(\boldsymbol{p}^{*}\right)={\nabla \hat{f}_n\left(\boldsymbol{p}^{*}\right)\left(\boldsymbol{p}-\boldsymbol{p}^{*}\right)}+{ \mathbb{E}\left[\int_{\boldsymbol{a}_i^{\top} \boldsymbol{p}}^{\boldsymbol{a}^{\top} \boldsymbol{p}^{*}}\left(I(r_n>v)-I\left(r_n>\boldsymbol{a}_i^{\top} \boldsymbol{p}^{*}\right)\right) d v\right]} .
		$$
	\end{proposition}
	The second step is to use this proposition to show the lipschitz continuity of $f(\cdot)$ and the uniqueness of the optimality solution $p^*$:
	\begin{proposition} \label{Q2} Under Assumption 1*, 2*,and 3*, for $\boldsymbol{p} \in \Omega_{p}$,
		$$
		\frac{\lambda \lambda_{\min }}{2}\left\|\boldsymbol{p}-\boldsymbol{p}^{*}\right\|_{2}^{2} \leq f(\boldsymbol{p})-f\left(\boldsymbol{p}^{*}\right)-\nabla f\left(\boldsymbol{p}^{*}\right)\left(\boldsymbol{p}-\boldsymbol{p}^{*}\right) \leq \frac{\mu \bar{a}^{2}}{2}\left\|\boldsymbol{p}-\boldsymbol{p}^{*}\right\|_{2}^{2}
		$$
		Moreover, the optimal solution $\boldsymbol{p}^{*}$ to the stochastic program (3.2) is unique.\label{4.5}
	\end{proposition}
The above proposition shows our Assumption \ref{ass 2*} imposes a strong local convexity and smoothness around $\boldsymbol{p}^*$ as in the i.i.d counterpart. It is not surprising to observe that under the same regularity conditions the dual objective functions should exhibit the same characteristics. With such a relationship between the dual objective functions and the dual optimal, to bound the convergence rate of the dual optimal it suffices to bound the convergence rate of the dual objective functions. With such a goal, we proceed to consider the concentration for the first and second order approximations:

\begin{proposition}
	We have \label{Q3}
	$$
	\mathbb{P}\left(\left\|\frac{1}{n} \sum_{i=1}^{n} \phi\left(p^{*},u_i\right)-\nabla f\left(\boldsymbol{p}^{*}\right)\right\|_{2} \geq \epsilon\right)\leq 2m \exp \left(-\frac{ \epsilon^{2}(K-2)^{2}}{K^{2} 2\lambda \bar{a}^2 T^{2}m} t\right)+m\tilde{\epsilon}(\delta, K) $$
	where $K=n\epsilon/T(d_i+\bar{a})$ and 
	$$
	\epsilon(\delta, K)=2 \exp \left(-\frac{\delta^{2} t}{(\lambda-\delta)^{2} \lambda^{2} T^{2}}\right)+\exp \left(\left(2 \delta (d_i+\bar{a})-\frac{K-2}{K\sqrt{m}} \epsilon \right) t\right),
	$$  where $u_i\sim \mathcal{P}_i$.\label{4.7}
	
\end{proposition}
The above result establishes the concentration of the first order approximation of the dual objective. The error term is a result of the error generated from the incomplete cycle of the regenerative process. Since we know for binding resource i the sub-gradient $\nabla f(\boldsymbol{p}^{*})_i=0$, we know the average sample gradient concentrates around zero up to small resources accumulated for the incomplete cycle. Let us show further that the second order term is uniformly bounded below with a high probability: 
\begin{proposition} \label{Q4}
	We have
	$$
	\begin{aligned}
		\mathbb{P}\left(\frac{1}{n} \sum_{j=1}^{n} \int_{\boldsymbol{a}_{j}^{\top} \boldsymbol{p}}^{\boldsymbol{a}_{j}^{\top} \boldsymbol{p}^{*}}\left(I\left(r_{j}>v\right)-I\left(r_{j}>\boldsymbol{a}_{j}^{\top} \boldsymbol{p}^{*}\right)\right) d v\right.& \geq-\epsilon^{2}-2 \epsilon \bar{a}\left\|\boldsymbol{p}^{*}-\boldsymbol{p}\right\|_{2}\\+\frac{\lambda \lambda_{\min }}{32}\left\|\boldsymbol{p}^{*}-\boldsymbol{p}\right\|_{2}^{2} 
		\text { for all } \left.\boldsymbol{p} \in \Omega_{p}\right)
	\end{aligned}
	$$
	$$\geq 1-m \exp \left(-\frac{n \lambda_{\min }^{2}}{4 \bar{a}^{2}}\right)+2 \exp \left(-\frac{n \epsilon^{2}}{2\lambda T^2} +\right) \cdot(2 N)^{m}$$$$-(2N)^m\cdot 2 \exp \left(-\frac{\delta^{2} t}{(\lambda-\delta)^{2} \lambda^{2} T^{2}}\right)+\exp \left(\left(2 \delta -\frac{K-2}{K} \epsilon\right) t\right)$$
	
	holds for any $\epsilon>0, n>m$ and $\mathcal{P}$ satisfies assumption 1*,2* and 3*. Here
	$$
	N=\left\lfloor\log _{q}\left(\frac{\underline{d} \epsilon^{2}}{\bar{a} \bar{r} \sqrt{m}}\right)\right\rfloor+1, \quad q=\max \left\{\frac{1}{1+\frac{1}{\sqrt{m}}}, \frac{1}{1+\frac{1}{\sqrt{m}}\left(\frac{\lambda \lambda_{\min }}{8 \mu \bar{a}^{2}}\right)^{\frac{1}{3}}}\right\}
	$$
	where $\lfloor\cdot\rfloor$ is the floor function. \label{4.8}
\end{proposition}
The above proposition establishes that the second order term is uniformly bounded below with high probability. Above two propositions on the concentration of first and second order impose a concentration on a quadratic function of  $\boldsymbol{p}$; namely the following has a high probability:$$
\begin{aligned}
	f_{n}(\boldsymbol{p})-f_{n}\left(\boldsymbol{p}^{*}\right) & \geq \nabla f\left(\boldsymbol{p}^{*}\right)^{\top}\left(\boldsymbol{p}-\boldsymbol{p}^{*}\right)-\epsilon\left\|\boldsymbol{p}^{*}-\boldsymbol{p}\right\|_{2}-\epsilon^{2}-2 \epsilon \bar{a}\left\|\boldsymbol{p}^{*}-\boldsymbol{p}\right\|_{2}+\frac{\lambda \lambda_{\min }}{32}\left\|\boldsymbol{p}^{*}-\boldsymbol{p}\right\|_{2}^{2} \\
	& \geq-\epsilon^{2}-(2 \bar{a}+1) \epsilon\left\|\boldsymbol{p}^{*}-\boldsymbol{p}\right\|_{2}+\frac{\lambda \lambda_{\min }}{32}\left\|\boldsymbol{p}^{*}-\boldsymbol{p}\right\|_{2}^{2} \text { uniformly for all } \boldsymbol{p} \in \Omega_{p}.
\end{aligned}
$$
This form is identical the equation 13 in \cite{5}, for the proofs that derive those propositions depend largely on the regularity conditions in our assumptions, and neither the i.i.d nor the regenerative structure plays a significant role. It is expected that other stationary price process may also exhibit a similar characteristic. Hence, one natural extension of the dual convergence theorem is to ask whether this quadratic bound also exist for other types of price data. Since this bound is the key to proving the dual convergence theorem, which is almost sufficient to prove the following regrets for the algorithms, any successful extension of this quadratic bound to other price processes would make the regret analysis for dual algorithms on such price process possible. 

The proof technique is similar to \cite{5} in the sense that since the proposition requires a uniform bound on uncountably many elements, we first partition the space into different sets; pick a representative element to which we apply the Regenerative Heoffding; and finally check show uniformly any element is close to one of the representative to conclude the proof. The details can be found in the appendix. 

Now we are ready to prove the Dual Convergence for the regenerative process:
\begin{theorem} \label{Q5}
	Under Assumption 1*, 2*, and 3*, there exists a constant $C$ such that
	$$
	\mathbb{E}\left[\left\|\boldsymbol{p}_{n}^{*}-\boldsymbol{p}^{*}\right\|_{2}^{2}\right] \leq \frac{C m \log m \log \log n}{n}
	$$
	holds for all $n \geq \max \{m, 3\}, m \geq 2$, and distribution $\mathcal{P}$ that satisfies Assumption 1*, 2*, and 3*. Additionally,
	$$
	\mathbb{E}\left[\left\|\boldsymbol{p}_{n}^{*}-\boldsymbol{p}^{*}\right\|_{2}\right] \leq C \sqrt{\frac{m \log m \log \log n}{n}}
	$$\label{4.11}
\end{theorem}
The significance of this theorem, as we will demonstrate later in the regret analysis, is that it provides an error bound for the dual-policy algorithm, for if our sample dual converges to the actual off-line dual fast enough, our accumulated error should be small. Such an idea is illustrated in the regret decomposition proposition. The extension of this theorem compared to its original version in \cite{5} is that it shows the regenerative process has the same order of convergence; hence we can expect the same order of regret for the algorithms. As we discussed above, very likely other stationary price processes may also have such dual convergence theorem. Hence the investigation of such a possibility remains an interesting open problem. 

	\section{Regret Analysis for Algorithms}
\subsection{Regret Analysis}	In this section, we shift our focus to analyzing the regret of the regenerative online linear programming that uses the dual optimization as its policy's basis. We will shortly define formally the regret and the dual-based policy in this section after a short introduction. Recall that the procedure of our dual algorithms depends on the following comparison: 
	\begin{equation}\tag{6.1}\label{eq:6.1}
		x_{j}^{*}=\left\{\begin{array}{ll}
			1, & r_{j}>\boldsymbol{a}_{j}^{\top} \boldsymbol{p}_{n}^{*} \\
			0, & r_{j}\leq \boldsymbol{a}_{j}^{\top} \boldsymbol{p}_{n}^{*}
		\end{array}\right.
	\end{equation}
	where $x_{j}^{*}$ is the optimal policy. This inequality holds true when the complementarity condition is in force. Hence, our optimization problem can be reformulated in the following way if we use such dual policy procedure:
	$$
	\begin{array}{rl}
		\max _{\boldsymbol{p} \geq \mathbf{0}} & \mathbb{E}\left[r I\left(r>\boldsymbol{a}^{\top} \boldsymbol{p}\right)\right] \\
		\text { s.t. } & \mathbb{E}\left[\boldsymbol{a} I\left(r>\boldsymbol{a}^{\top} \boldsymbol{p}\right)\right] \leq \boldsymbol{d}
	\end{array}
	$$
	However, in practice, we do not need to spare the energy to compute the exact form of $\boldsymbol{p}_{n}^{*}$ each time. Nor do we know such optimal dual before the completion of the program. Hence, suppose a decent approximation of $\boldsymbol{p}_{n}^{*}$ is possible for each $n$, then if we use the same procedure as \ref{eq:6.1} except using the approximated dual optimal, we will get a small regret, that is the difference between the true optimal revenue and our actual revenue should be small. This reasoning is exactly why we need to compute the convergence of the dual optimals, for it provides a theoretical basis for the regret analysis. Let us define the regrets formally now: Suppose $\mathbf{a}_i$ is generated i.i.d while $r_i$ follows a regenerative process. We denote the offline optimal value of the objective as $\boldsymbol{x}^{*}=\left(x_{1}^{*}, \ldots, x_{n}^{*}\right)^{\top}$, and the offline (online) objective value as $R_{n}^{*}\left(R_{n}\right)$. Specifically,
	$$
	\begin{aligned}
		R_{n}^{*} &:=\sum_{j=1}^{n} r_{j} x_{j}^{*} \\
		R_{n}(\pi) &:=\sum_{j=1}^{n} r_{j} x_{j} .
	\end{aligned}
	$$
	A quick observation would tell us since $R_{n}^{*}$ is the revenue generated by the policy which assumes a full knowledge of the realization, it is the upper bound of any other policies. Therefore, the regret is the comparison of those two objects:
	\begin{definition}
		We define the regret as 
		$$
		\Delta_{n}^{\mathcal{P}}(\boldsymbol{\pi}):=\mathbb{E}_{\mathcal{P}}\left[R_{n}^{*}-R_{n}(\boldsymbol{\pi})\right]
		$$
		and the worst-case regret as 
		$$
		\Delta_{n}(\boldsymbol{\pi}):=\sup _{\mathcal{P} \in \Xi} \Delta_{n}^{\mathcal{P}}(\boldsymbol{\pi})=\sup _{\mathcal{P} \in \Xi} \mathbb{E}_{\mathcal{P}}\left[R_{n}^{*}-R_{n}(\boldsymbol{\pi})\right].
		$$ \label{6.1}
	\end{definition}
	When the distribution is known, the regret of the first kind is sufficient for our analysis. However, in the case where we only know certain regularity conditions of our distribution, we will encounter a distributional optimization problem as illustrated in the worst-case regret. Now, we will also formally define our dual-based policy. 
	
	\begin{definition}
		A dual-based policy $\{x_i\}$ is a policy constructed by the following procedure: first we compute some vector, interpreted as the approximation of the dual optimal, $\boldsymbol{p}_{t}=h_{t}\left(\mathcal{H}_{t-1}\right)$, where $\mathcal{H}_{t-1}=\left\{r_{j}, \boldsymbol{a}_{j}, x_{j}\right\}_{j=1}^{t-1}$. Then, we set the candidate policy as 
		$$
		\tilde{x}_{t}=\left\{\begin{array}{ll}
			1, & \text { if } r_{t}>\boldsymbol{a}_{t}^{\top} \boldsymbol{p}_{t} \\
			0, & \text { if } r_{t} \leq \boldsymbol{a}_{t}^{\top} \boldsymbol{p}_{t}
		\end{array}\right.
		$$
		To set our policy as the candidate policy, we need to check whether adopting such candidate policy would not violate the resource constraint:
		$$
		x_{t}=\left\{\begin{array}{ll}
			\tilde{x}_{t}, & \text { if } \sum_{j=1}^{t-1} a_{i j} x_{j}+a_{i t} \tilde{x}_{t} \leq b_{i}, \quad \text { for } i=1, \ldots, m \\
			0, & \text { otherwise. }
		\end{array}\right.
		$$
		Such policy based on this rule is called the dual-based policy.
	\end{definition}
	Since our procedure, in the one-sided situation, terminates when the resources are depleted, it makes sense to define the stopping time for resource depletion as 
	$$
	\tau_{s}:=\min \{n\} \cup\left\{t \geq 1: \min _{i} b_{i t}<s\right\}
	$$
	where $\boldsymbol{b}_{0}=\boldsymbol{b}=n \boldsymbol{d}$, 
	$\boldsymbol{b}_{t}=\boldsymbol{b}_{t-1}-\boldsymbol{a}_{t} x_{t}$ represents the left-over resource after time $t$. This stopping time stops when some resource $i$ at time $t$ is less than a threshold amount of $s$. In practice, future orders may still be fulfilled when some resource falls below the threshold moment. Moreover, in the double-sided situation where orders represent both buyers and sellers, such stopping time would not cause an issue to the programming. However, assuming the orders are time-homogenous, the resource depletion rate should be linear in time and any early resource depletion represents a certain amount of misuses of the resources. We will see how such early depletion would cause an increase to the regret. To study the regret, let us consider the following Optimization problem: 
	\begin{equation}\tag{6.2}\label{eq:6.2}
		\max _{\boldsymbol{p} \geq \mathbf{0}} \frac{1}{\mathbb{E}\tau_1}\mathbb{E}\sum_{i=0}^{\tau_1}\left[r_i I\left(r_i>\boldsymbol{a}_i^{\top} \boldsymbol{p}\right)\right] \text { s.t. } \frac{1}{\mathbb{E}\tau_1} \mathbb{E}\left[\sum_{i=1}^{\tau_1}\boldsymbol{a}_i I\left(r_i>\boldsymbol{a}^{\top} \boldsymbol{p}\right)\right] \leq \boldsymbol{d}
	\end{equation}
	This optimization can be seen as the deterministic relaxation of the stochastic program of  \ref{eq:3.1}, and it differs mainly form \cite{5} in the sense that we need to take the average over an entire period of the regeneration. The reason for such a formulation is that it provides a clean and tractable form for the upper bound, for let us recall that when $n$ is large, the average reward we collect form each other in \ref{eq:3.1} is approximated the same as in \ref{eq:6.2}. Let us consider the Lagrangian of the deterministic formulation as $$
	g(\boldsymbol{p}):=\frac{1}{\mathbb{E}\tau_1}\mathbb{E}\sum_{i=1}^{\tau_1}\left[r_i I\left(r_i>\boldsymbol{a}_i^{\top} \boldsymbol{p}\right)+\left(\boldsymbol{d}-\boldsymbol{a}_i I\left(r_i>\boldsymbol{a}_i^{\top} \boldsymbol{p}\right)\right)^{\top} \boldsymbol{p}^{*}\right]
	$$
	where $\boldsymbol{p}^{*}$ is the optimal solution to \ref{eq:2}. Since our price parameter is not i.i.d, it makes sense to depend the Lagrangian on time as 
	$$
	g_{i}(\boldsymbol{p}):=\left[r_{i} I\left(r_{i}>\boldsymbol{a}_{i}^{\top} \boldsymbol{p}\right)+\left(\boldsymbol{d}-\boldsymbol{a}_{i} I\left(r_{i}>\boldsymbol{a}_{i}^{\top} \boldsymbol{p}\right)\right)^{\top} \boldsymbol{p}^{*}\right].
	$$
	To formalize our idea that the expected revenue $R_{n}^{*}$ is bounded by our tractable form, let us prove the following proposition:
	
\begin{proposition}\label{R1}
	Under Assumptions \ref{ass 1*}, \ref{ass 2*}, and \ref{ass 3*} we have
	$$
	\mathbb{E} R_{n}^{*} \leq \sum_{i=1}^{n}g_i\left(\boldsymbol{p}^{*}\right) $$$$
	g_i\left(\boldsymbol{p}^{*}\right) \geq g_i(\boldsymbol{p})
	$$
	for any $\boldsymbol{p} \geq 0$. Additionally,
	$$
	g_i\left(\boldsymbol{p}^{*}\right)-g_i(\boldsymbol{p}) \leq \mu \bar{a}^{2}\left\|\boldsymbol{p}^{*}-\boldsymbol{p}\right\|_{2}^{2}
	$$
	holds for all $\boldsymbol{p} \in \Omega_{p}$ and all the distribution $\mathcal{P}$ that satisfies those three assumptions. 
\end{proposition}

	With this result, let us move to analyze the worst-case regret as defined by \ref{6.1}. In particular, there are three different sources of regret in such programming. The first is that approximated regret, resulted from using non-optimal dual in the policy making procedure. This regret is linear in the operation time. A second source of regret is the temporary regret, resulted from the situation when the programming terminates too early such that profitable orders in the end are left unfulfilled. This regret corresponds to the case like a tail risk, where the highly profitable orders can accumulate in the end. The third source of regret is the resource regret, resulted from not utilizing all the resources, especially the binding resources that, from the complementarity perspective, constitute the bottleneck for optimizing the objective function. Let us formalize those ideas in the following theorem:
	\begin{theorem} \label{R2}
		Under Assumption \ref{ass 1*},\ref{ass 2*} and \ref{ass 3*}, there exists a constant $K$ such that the worst-case regret under policy $\pi$,
		$$
		\Delta_{n}(\pi) \leq K \cdot \mathbb{E}\left[\sum_{t=1}^{\tau_{\bar{a}}}\left\|\boldsymbol{p}_{t}-\boldsymbol{p}^{*}\right\|_{2}^{2}+\left(n-\tau_{\bar{a}}\right)+\sum_{i \in I_{B}} b_{i n}\right]
		$$
		holds for all $n>0$. Here $I_{B}$ is the set of binding constraints, $\boldsymbol{p}_{t}$ is specified by the policy $\boldsymbol{\pi}$, and $\boldsymbol{p}^{*}$ is the optimal.\label{6.4}
	\end{theorem}
	Therefore, as we discussed above, a nice policy should have the following features: first, the average error between the approximated dual optimal and the true dual optimal shouldn't be large. Second, the consumption rate should be smooth. And third, all the binding resources should be utilized with no waste. It is in this regret theorem where we see exactly why wee need to construct the dual convergence theorem of \ref{First Main Result}. 
	One Corollary to this theorem is 
	\begin{corollary}
		Using the same notation as above, we have any given $\boldsymbol{b}_{t}$-adapted stopping time $\tau$, if $\mathbb{P}\left(\tau \leq \tau_{\bar{a}}\right)=1$,
		$$
		\Delta_{n}(\boldsymbol{\pi}) \leq K \cdot \mathbb{E}\left[\sum_{t=1}^{\tau-1}\left\|\boldsymbol{p}_{t}-\boldsymbol{p}^{*}\right\|_{2}^{2}+(n-\tau)+\sum_{i \in I_{B}} b_{i n}\right].
		$$
	\end{corollary}
Above Theorem \ref{R2} establishes that the best possible upper bound for the efficiency of our algorithms is of the same order of the Dual Convergence Theorem. Hence, for any geometrically updating algorithms, $\log n \log \log n$ is the best upper bound given the  Dual Convergence Theorem. We will discuss in more details later.

	\subsection{When the Distribution is Known}
	In this section, we will discuss the regret for the algorithm when the distribution is known discussed in the \cite{5} using the regret analysis derived from the previous section. 
	\begin{algorithm}[h]
		\caption{Known Distribution}\label{alg:1}
		\begin{algorithmic}[1]
			\State Input: $n, d_{1}, \ldots, d_{m}$, Distribution $\mathcal{P}$
			\State Compute the optimal solution of the stochastic programming problem
			$$
			\begin{array}{l}
				\boldsymbol{p}^{*}=\arg \min \boldsymbol{d}^{\top} \boldsymbol{p}+\mathbb{E}_{(r, \boldsymbol{a}) \sim \mathcal{P}}\frac{1}{\mathbb{E}(\tau_1)}\sum_{i=1}^{\tau_1}\left[\left(r_i-\boldsymbol{a}_i^{\top} \boldsymbol{p}\right)^{+}\right] \\
				\text {s.t. } \boldsymbol{p} \geq 0
			\end{array}
			$$
			\For{ $t=1, \ldots, n$} 
			\State$\quad$ If constraints are not violated, choose
			$$
			x_{t}=\{\begin{array}{ll}
				1, & \text { if } r_{t}>\boldsymbol{a}_{t}^{\top} \boldsymbol{p}^{*} \\
				0, & \text { if } r_{t} \leq \boldsymbol{a}_{t}^{\top} \boldsymbol{p}^{*}.
			\end{array}
			$$
			\EndFor
		\end{algorithmic}
	\end{algorithm}
	The regret bound for this algorithm is given by
	\begin{theorem}\label{Al1}
		With the online policy $\boldsymbol{\pi}_{1}$ specified by Algorithm \ref{alg:1},
		$$
		\Delta_{n}\left(\boldsymbol{\pi}_{1}\right) \leq O(\sqrt{n})
		$$
	\end{theorem}
	Essentially, knowing the distribution for the data is powerful enough to achieve sub-linear regret. Hence, to optimize our objective value with sub-linear regret, we do not need to consider every data in the sequence, and the optimizing problem can be transformed into a statistical problem of distributional approximation. 

	\subsection{ Dynamic Learning Algorithm}
	The above algorithm assumes the knowledge of the distribution, which is usually not true in the application. Therefore, we want to approximate the dual optimal as more information becomes available. The question becomes, how frequently should we update the dual price, since there will be a computational cost associated with this update. Since as more information becomes available, our dual price becomes a better approximation of the actual dual optimal so that the update should be less frequent. To illustrate such an idea, The algorithm below incorporates a geometric update rule:
	
	\begin{algorithm}[H]
		\caption{Dynamic Learning Algorithm}\label{alg:2}
		\begin{algorithmic}[1]
			\State Input: $d_{1}, \ldots, d_{m}$ where $d_{i}=b_{i} / n$
			\State Initialize: Find $\delta \in(1,2]$ and $L>0$ s.t. $\left\lfloor\delta^{L}\right\rfloor=n$.
			\State Let $t_{k}=\left\lfloor\delta^{k}\right\rfloor, k=1,2, \ldots, L-1$ and $t_{L}=n+1$
			\State Set $x_{1}=\ldots=x_{t_{1}}=0$
			\For{ $k=1,2, \ldots, L-1$ }
			\State Specify an optimization problem
			$$
			\begin{aligned}
				\max & \sum_{j=1}^{t_{k}} r_{j} x_{j} \\
				\text { s.t. } & \sum_{j=1}^{t_{k}} a_{i j} x_{j} \leq t_{k} d_{i}, \quad i=1, \ldots, m \\
				& 0 \leq x_{j} \leq 1, \quad j=1, \ldots, t_{k}
			\end{aligned}
			$$
			\State  Solve its dual problem and obtain the optimal dual variable $\boldsymbol{p}_{k}^{*}$
			$$
			\begin{array}{c}
				\boldsymbol{p}_{k}^{*}=\underset{p}{\arg \min } \sum_{i=1}^{m} d_{i} p_{i}+\frac{1}{t_{k}} \sum_{j=1}^{t_{k}}\left(r_{j}-\sum_{i=1}^{m} a_{i j} p_{i}\right)^{+} \\
				\text {s.t. } p_{i} \geq 0, \quad i=1, \ldots, m
			\end{array}
			$$
			\For{ $t=t_{k}+1, \ldots, t_{k+1}$} 
			\State If constraints permit, set
			$$
			x_{t}=\left\{\begin{array}{ll}
				1, & \text { if } r_{t}>\boldsymbol{a}_{t}^{\top} \boldsymbol{p}_{k}^{*} \\
				0, & \text { if } r_{t} \leq \boldsymbol{a}_{t}^{\top} \boldsymbol{p}_{k}^{*}
			\end{array}\right.
			$$
			\State  $\quad$ Otherwise, set $x_{t}=0$
			\State  $\quad$ If $t=n$, stop the whole procedure.
			\EndFor 
			\EndFor
		\end{algorithmic}
	\end{algorithm}
	Let us analyze the regret of this algorithm by proving the following theorem
	\begin{theorem}
		With the online policy $\pi_{2}$ specified by above Algorithm, where the distribution of $(\mathbf{a},r)$ satisfies \ref{ass 1*},\ref{ass 2*},and \ref{ass 3*}, then 
		$$
		\Delta_{n}\left(\pi_{2}\right) \leq O(\sqrt{n} \log n)
		$$ \label{7.2}
	\end{theorem}
Essentially, we see the main contribution to the regret for this algorithm is, as seen from the proof, is the wasted time. The accumulated errors generated from the sample dual is $O(\log n\log\log n)$ while the regret generated from the wasted resources is $\sqrt{n}\sqrt{\log \log n}$, whereas the regret generated from early exit time is $O(\sqrt{n}\log  n)$. Hence, early exit time is considered to be most harmful to this type of algorithm, for it forgoes potentially large orders in the end, compared to wasted resources whose cost is at most the shadow price per unit. It is no surprise that some similar algorithm like \cite{2} include a small shrinkage term in the constraint to be slightly more conservative, in order to ensure minimum early exit time at the relatively low cost of wasted resources.

\section{Numerical Simulations}
In this section, we provide numerical simulations to test the Dynamical Learning Algorithm. We test two kinds of models, the model where the price depends on the quantity of purchase and the model where the price is independent. We can also observe that though the data violates some regularity constraints for our three assumptions \ref{ass 1*}, \ref{ass 2*}, \ref{ass 3*}, the performance is better than what the regret theorem \ref{7.2} predicts.

Let us denote a bounded random walk model 
$$	
 R_{t+1} =\left\{\begin{array}{ll}
	X_t+\epsilon_t, & \text { if }\underline{r}\leq  |R_t+\epsilon_t|\leq  \bar{r} \\
\bar{r}, & \text { if }	R_t+\epsilon_t,> \bar{r}\\
\underline{r}, & \text { if }	R_t+\epsilon_t,<  \underline{r}
\end{array}\right.
$$
where $\epsilon_t$ is the i.i.d increment. In below's example, we use $ \underline{r}=1, \bar{r}=5, m=5, \epsilon_i\sim 2B(0.5)-1$ for Bernoulli $B(0.5)$. In Random Input I, we chose m independent bounded random walks, starting from $\underline{r}$, as the hidden market price for each resource, and the bid price is the sum of the quantity multiplied by the market price. Therefore, Random Input I reflects a type of efficient market where the fair prices are known to the buyer while the seller is to learn those prices. In this case the seller does not receive any surplus. Random Input II has a single regenerative price with no hidden item price. Therefore,  Random Input II describes a situation where the price and the quantity are independent, so there is a chance for the seller to exploit consumer surplus, for consumers may pay more than the fair prices. Both inputs follow a certain regenerative random walk structure. That financial data is well modeled by random walk is not new to us. The bounded random walk can be used to model the return of combinations of options, for example a protective collar option strategy. 

\begin{tabularx}{1\textwidth} { 
		| >{\raggedright\arraybackslash}X 
		| >{\centering\arraybackslash}X 
		| >{\raggedleft\arraybackslash}X | }
	\hline
	Random Input I (Quantity Dependent Price) & $a_{ij} \sim |\text{Normal}(0.5, 1)| $ & $r_i=\sum_{j=1}^{m} a_{ij}r_{ij}$, $r_ij\sim R_{ij}$\\
	\hline
	Random Input II (Quantity Independent Price) & $a_{ij} \sim |\text{Normal}(0.5, 1)| $ &  $r_i\sim R_i$\\
	\hline
	Random Input III (I.I.D Price) & $a_{ij} \sim |\text{Normal}(0.5, 1)| $ &  $r_i\sim \text{Uniform}(1,5)$\\
	\hline
\end{tabularx}
\\

The realization of the sample paths of bounded random walk are given below in figure \ref{figure 0}.
\begin{figure}[!ht]
 		\caption{Bounded Fair (Left) and Weighted Down (Right) Random Walks }
  \includegraphics[scale=0.3]{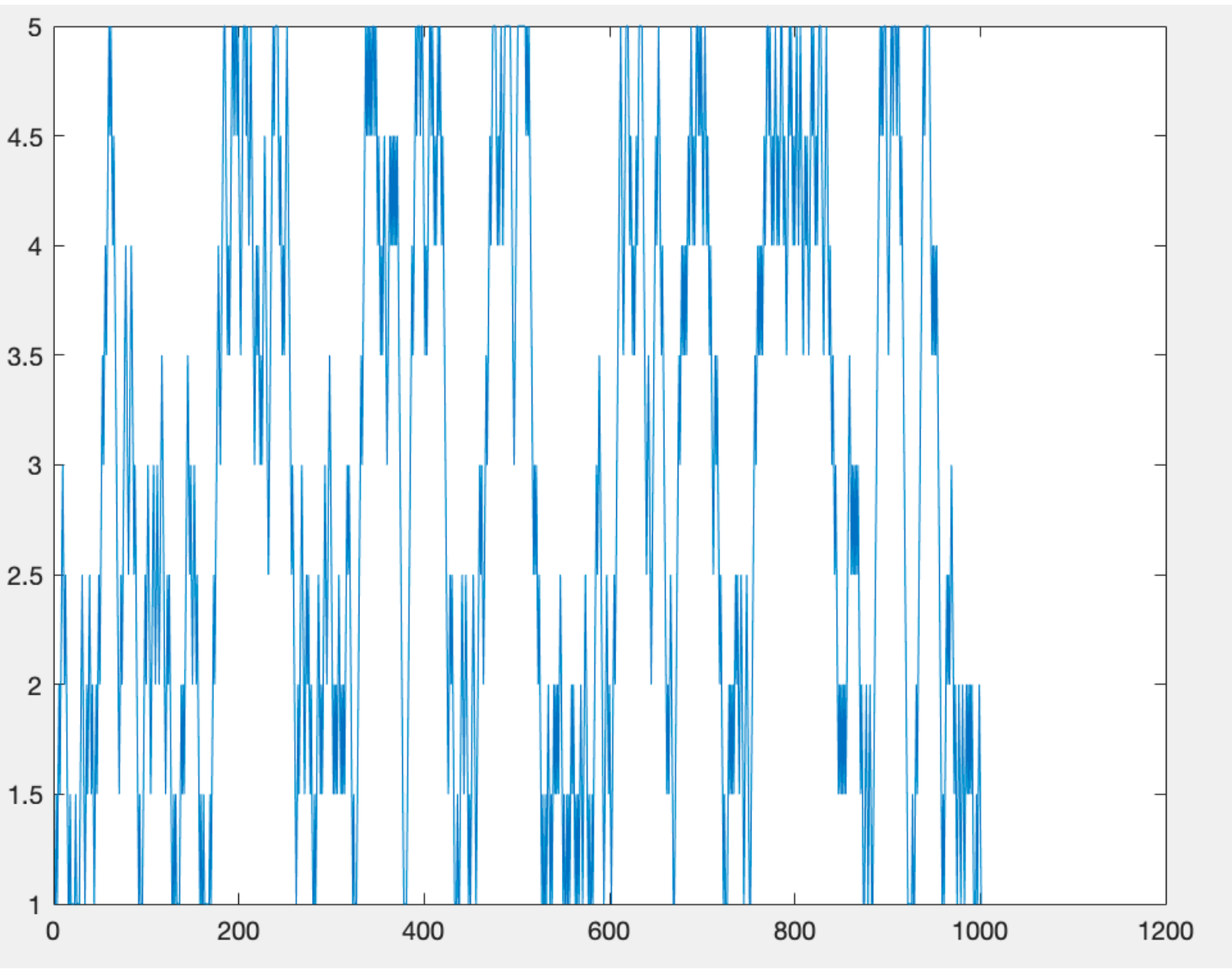}
   \includegraphics[scale=0.3]{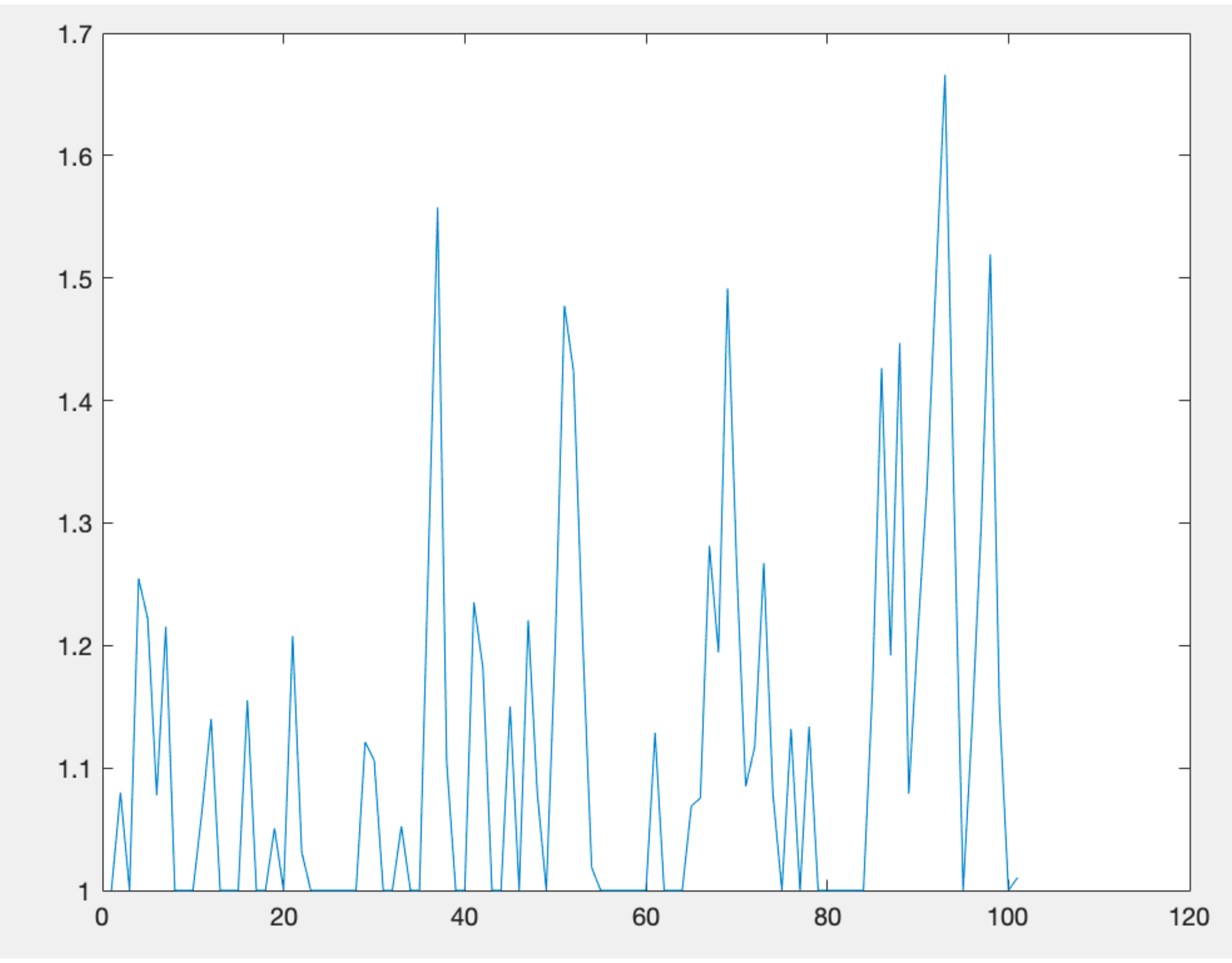}
\centering
\label{figure 0}
\end{figure}

The regret and the consumption rate is shown below. In figure \ref{figure 1}, figure \ref{figure 3}, and figure \ref{figure 5} , we observe that the regrets are below the upper bound of $O(\sqrt{n}\log n)$ as $O(\sqrt{n})$. Meanwhile, they imply that on a larger scale regenerative price and i.i.d price data give the similar performance for our algorithms. 

\begin{figure}[!ht]
 		\caption{Regret for Algorithm \ref{alg:2} (Red) bounded by $25\sqrt{n}$ (Blue)  with Input I}
  \includegraphics[scale=0.3]{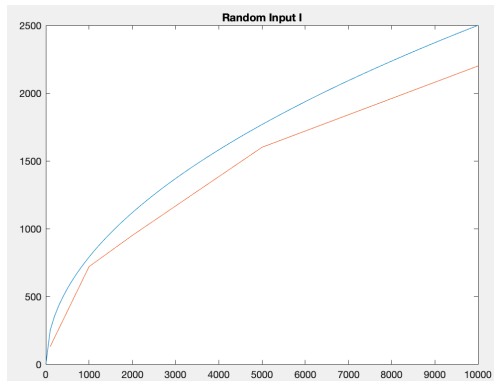}
\centering
\label{figure 1}
\end{figure}

\begin{figure}[!ht]

 		\caption{Resource Consumption Rate of Algorithm \ref{alg:2} (Left) and Zoomed in View (Right) with Input I}
  \includegraphics[scale=0.3]{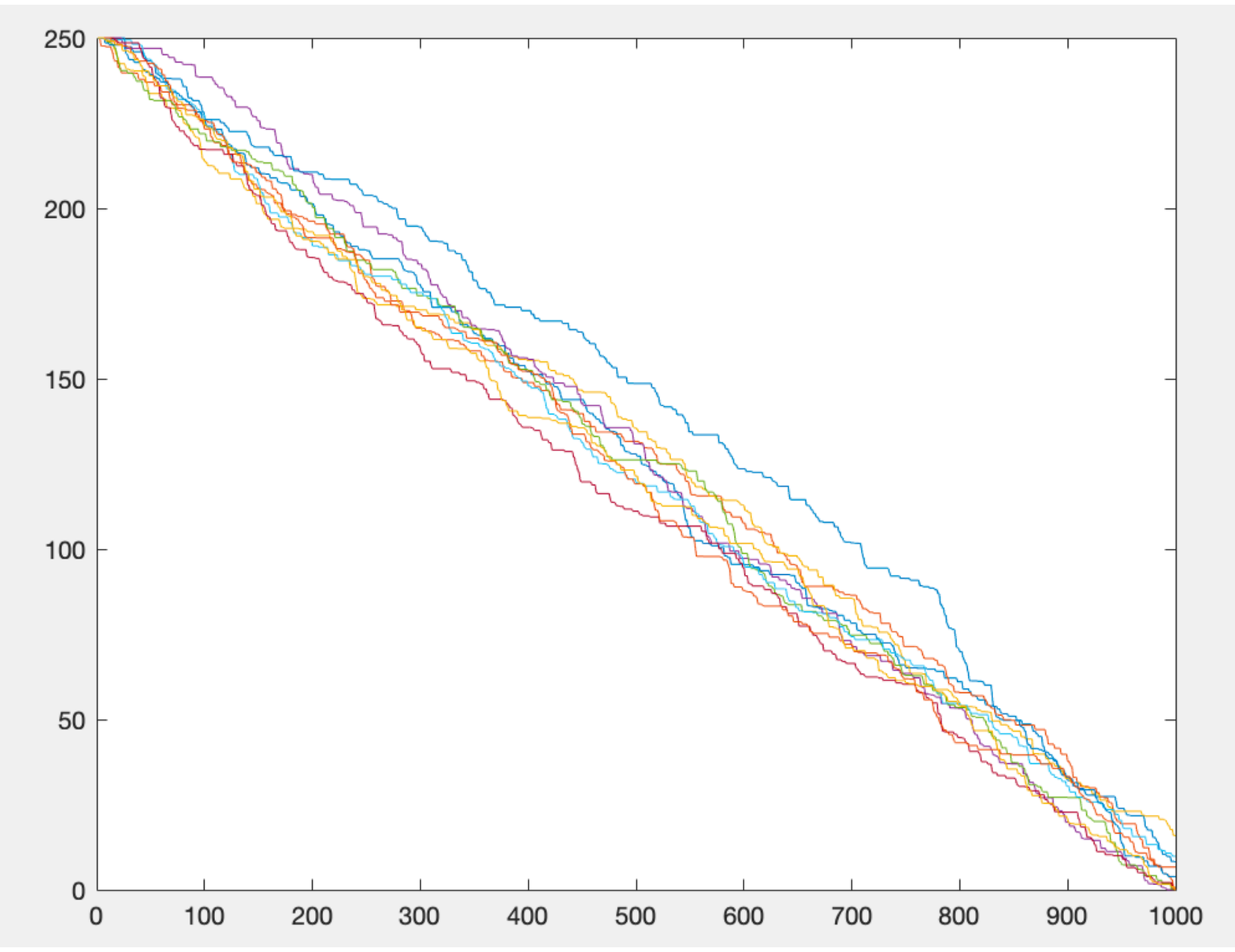}
    \includegraphics[scale=0.3]{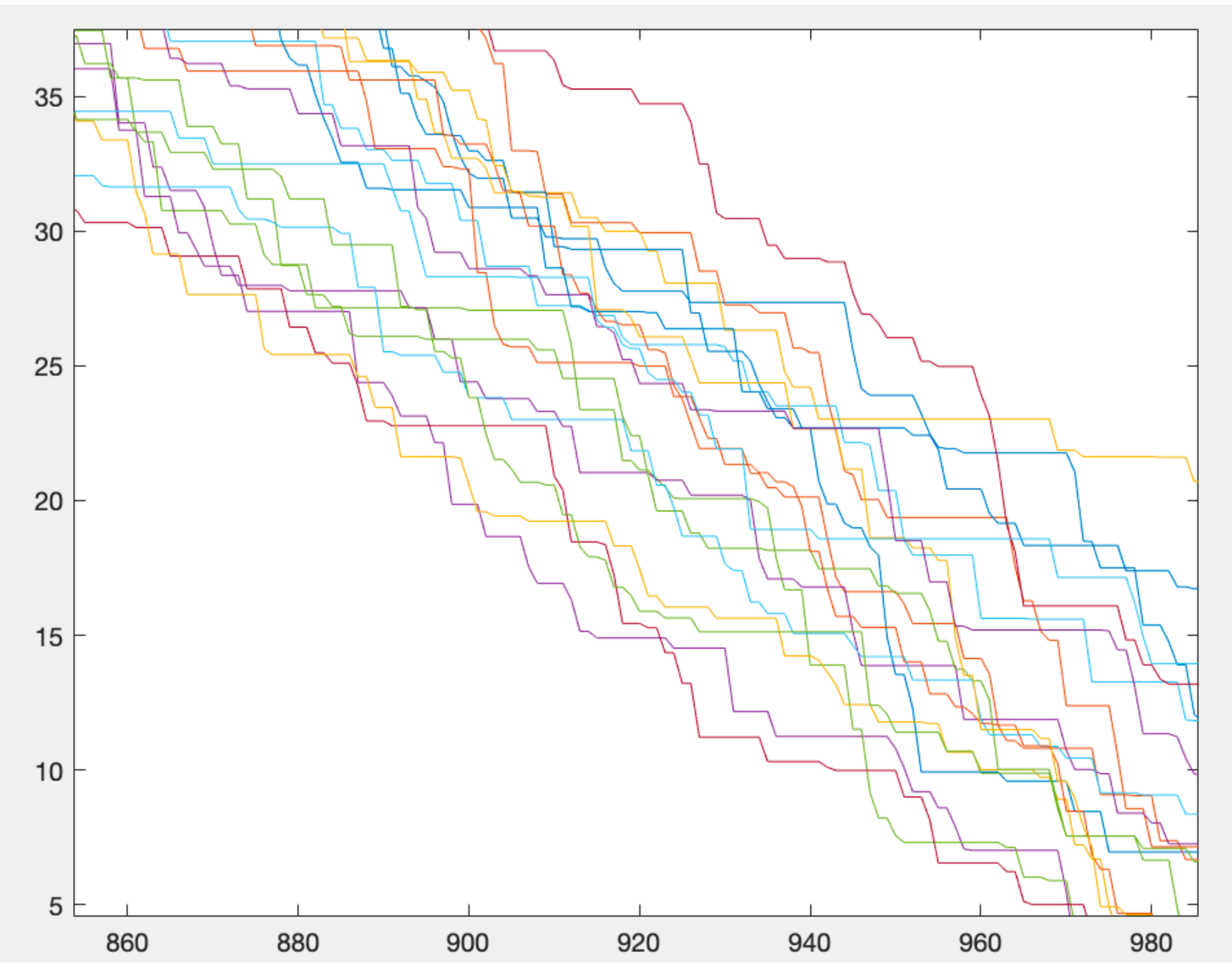}
\centering
\label{figure 2}
\end{figure}

\begin{figure}[!ht]

 		\caption{Regret for Algorithm \ref{alg:2} (Red) bounded by $4\sqrt{n}$ (Blue) with Input II}
  \includegraphics[scale=0.3]{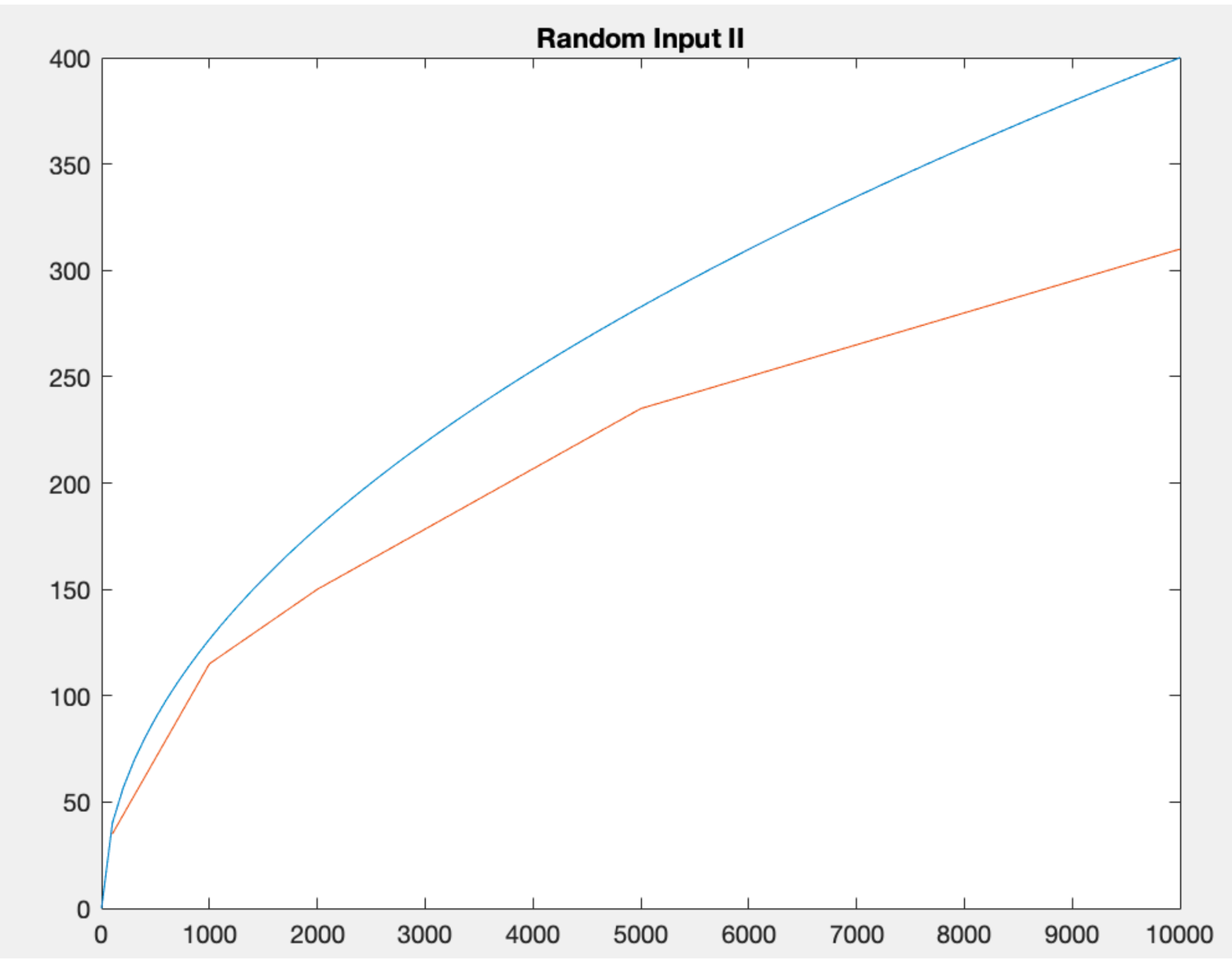}
\centering
\label{figure 3}
\end{figure}

\begin{figure}[!ht]

 		\caption{Resource Consumption Rate of Algorithm \ref{alg:2} (Left) and Zoomed in View (Right) with Input II}
  \includegraphics[scale=0.3]{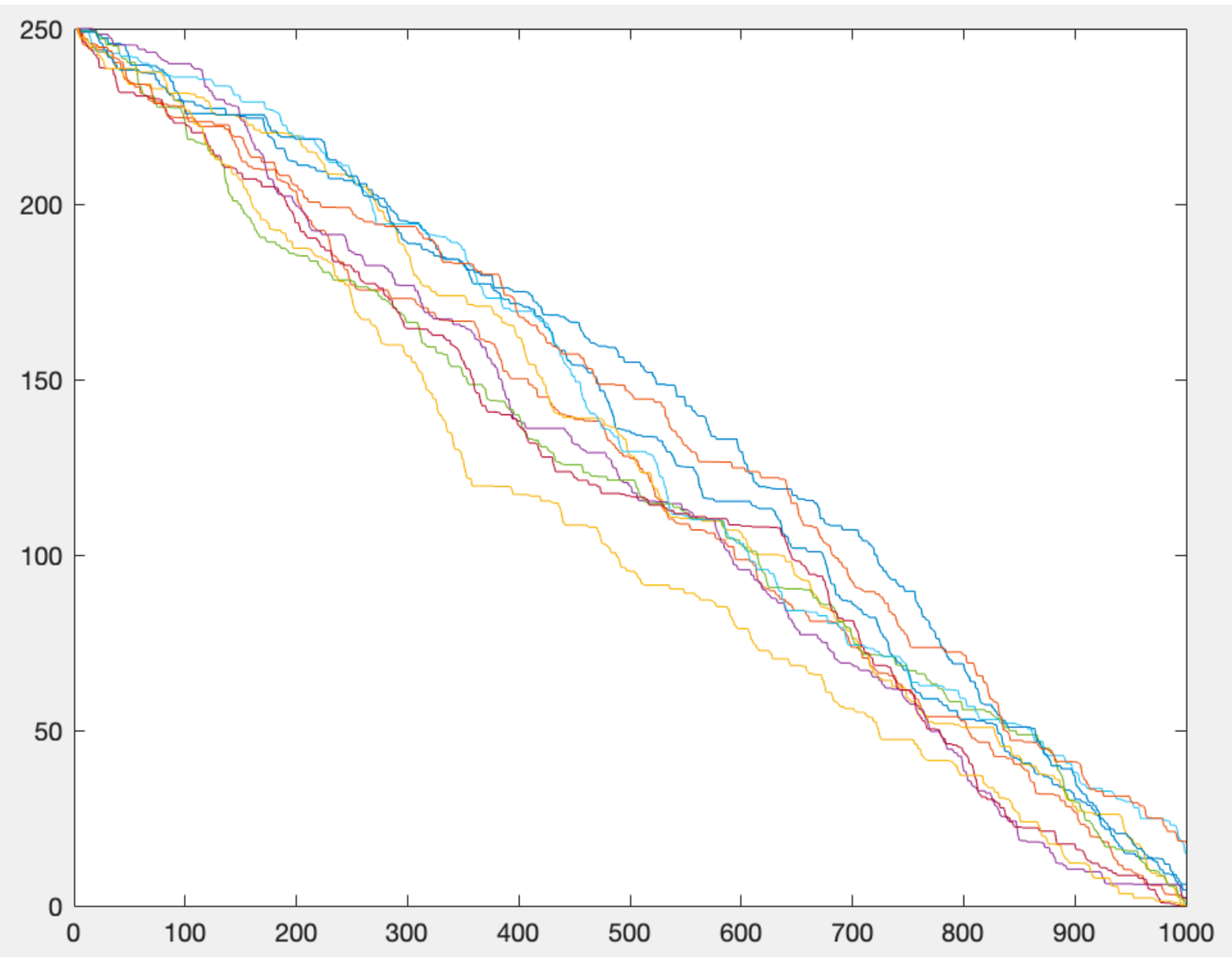}
    \includegraphics[scale=0.3]{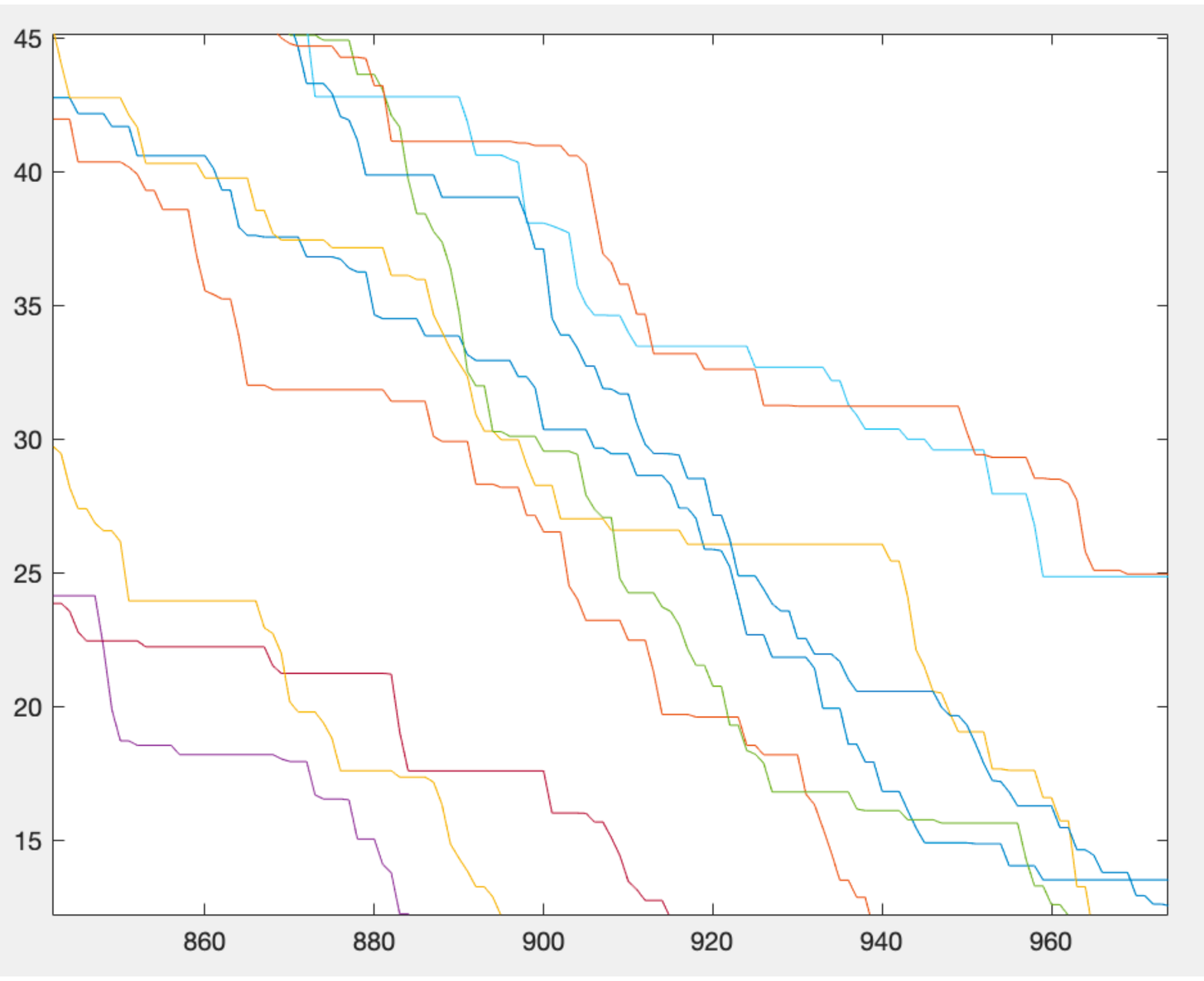}
\centering
\label{figure 4}
\end{figure}

\begin{figure}[!ht]

 		\caption{Regret for Algorithm \ref{alg:2} (Red) bounded by $4\sqrt{n}$ (Blue) with Input III}
  \includegraphics[scale=0.3]{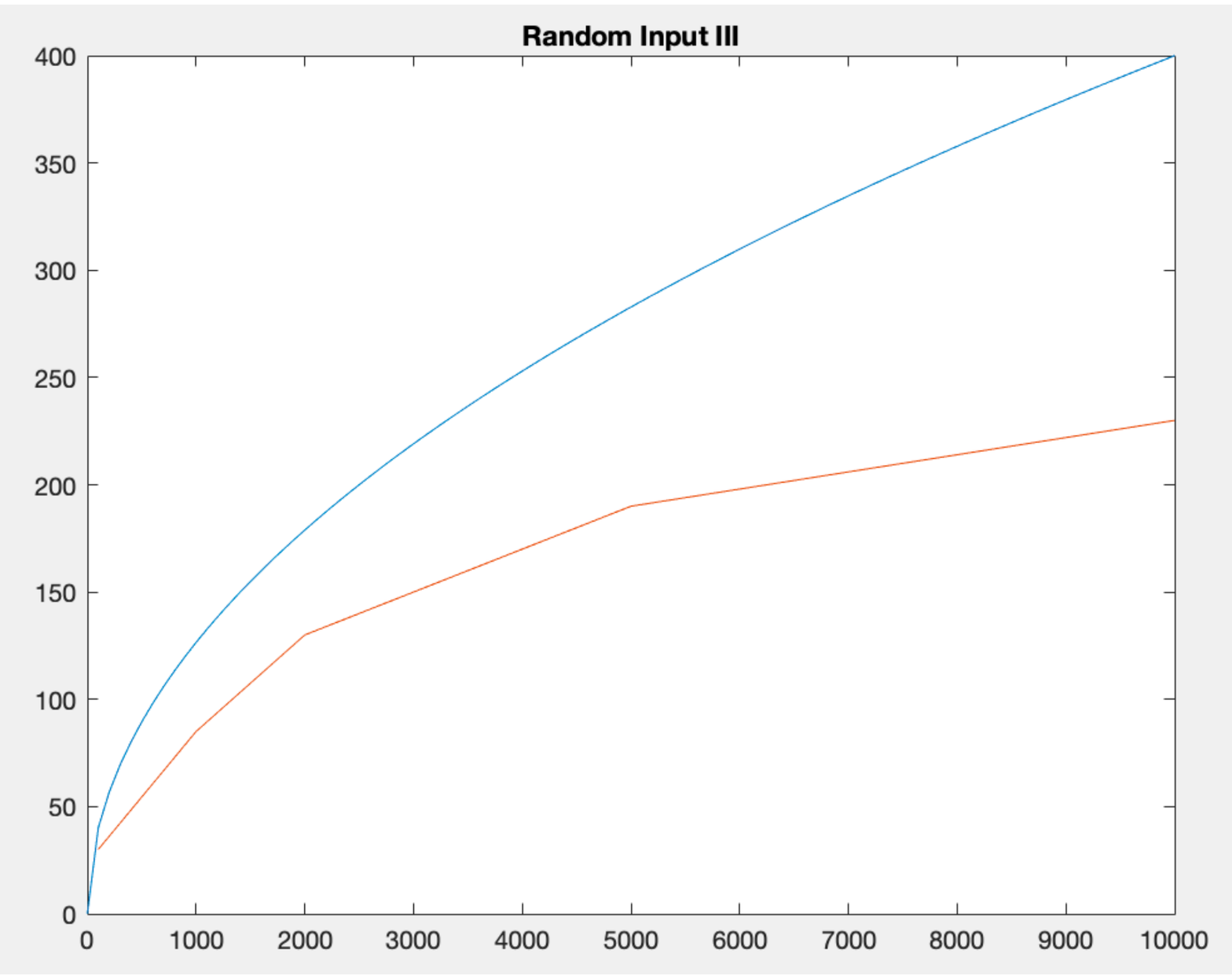}
\centering
\label{figure 5}
\end{figure}

\begin{figure}[!ht]
	\caption{Resource Consumption Rate of Algorithm \ref{alg:2} (Left) and Zoomed in View (Right) with Input III}
 		
  \includegraphics[scale=0.3]{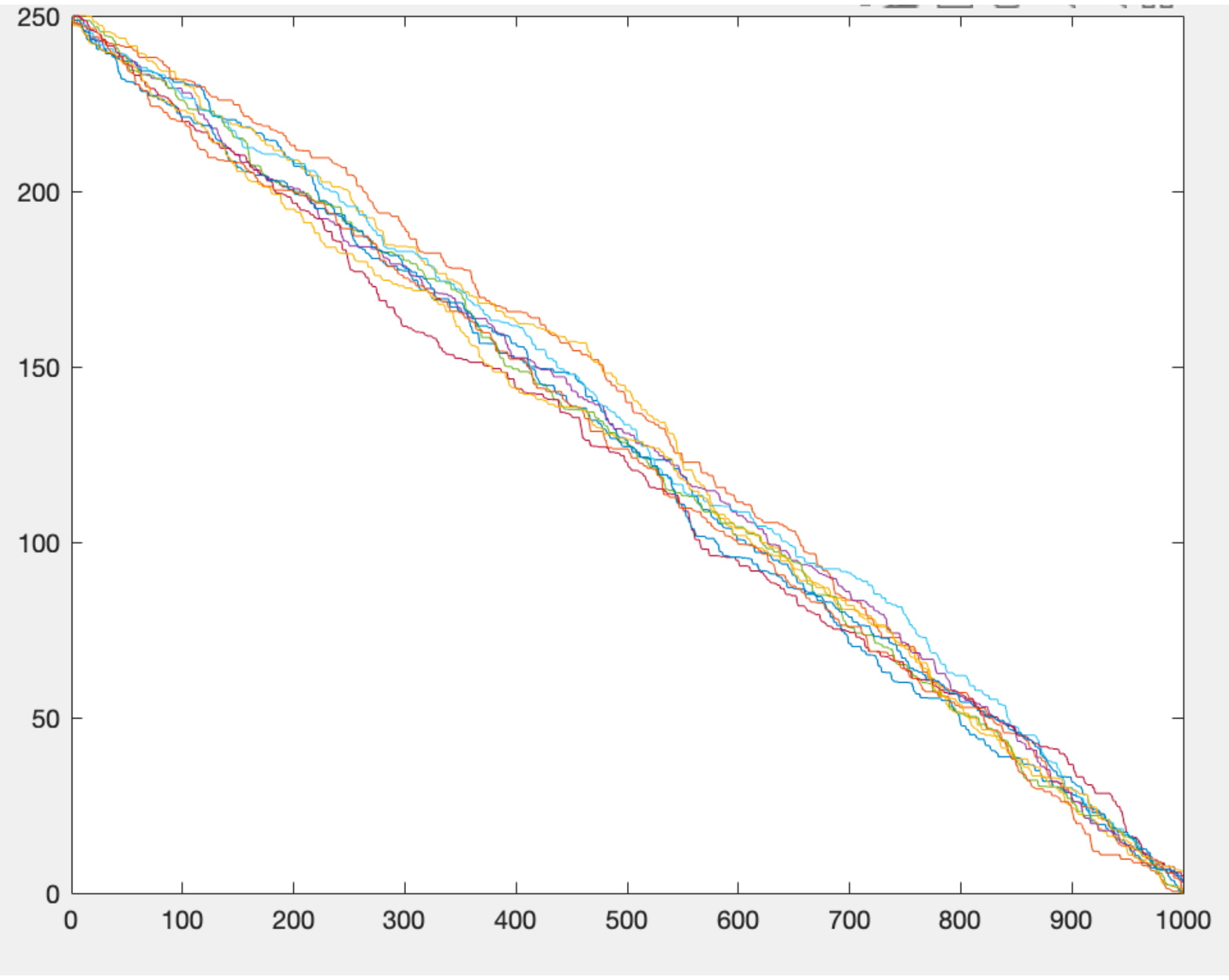} 
    \includegraphics[scale=0.3]{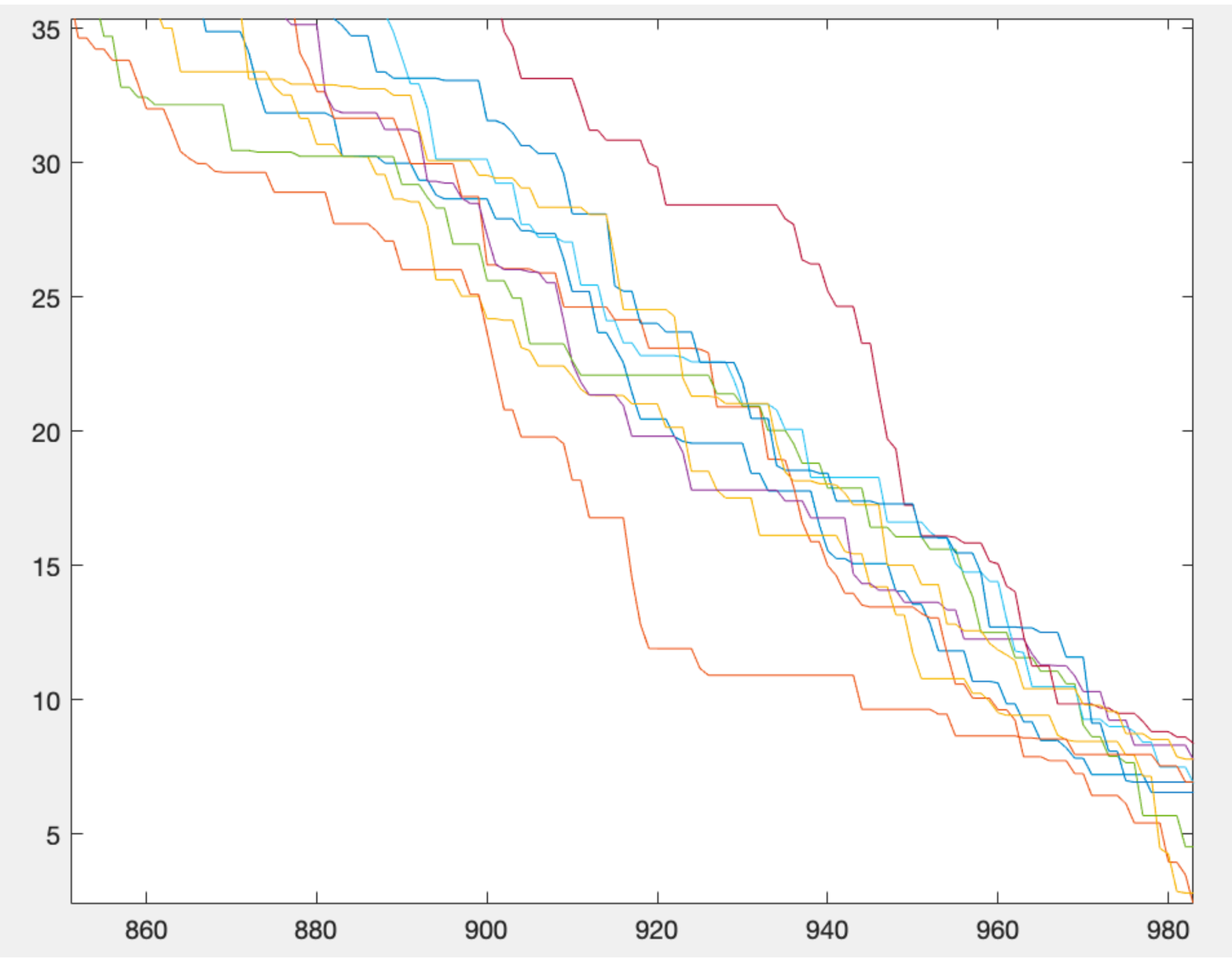}

\centering
\label{figure 6}
\end{figure}
\newpage
There are few important observations to be made from the consumption rates in figure \ref{figure 2}, \ref{figure 4},\ref{figure 6}. Figure \ref{figure 6} with the i.i.d price data has the smoothest consumption rate with the least wasted resource and time. This optimal performance is due to the fact that we assume a linear consumption rate when solving for the dual optimal in the algorithm. For i.i.d data, this assumption is realistic at all scales, both macro and local, and therefore the real consumption rate based on this approach is smooth.  Regenerative data that is independent of the quantity of purchase in Figure \ref{figure 4}, however, may not suit this assumption well at the micro level, for even if the consumption rate is linear at the macro level, since each period is i.i.d, the consumption at local level is not linear. Hence, using this assumption may cause some small deviations of the true dual and result in the consumption rate becoming rough in the zoomed in picture on the right of figure \ref{figure 4}. However, when the scale is large, such a small deviation is insignificant. Indeed, figure \ref{figure 3} and \ref{figure 5} show when the price is independent, no matter whether it is i.i.d or regenerative, they have similar regrets. Figure \ref{figure 2} shows there exists some true fluctuating market price for each item and the price is the market price for the entire bundle. Since there exists a hidden unobserved market price, there is relatively no noise in the system, compared to random input II where the market has only noises (since the price is independent of the quantity). Such data is therefore easier to learn and causes a more stable consumption rate. To summarize, the consumption rate is most linear when the price data is i.i.d  with less noise, and less linear when the price data is regenerative with noises. Such difference is caused by the linear consumption rate assumed by the algorithm.

As we have discussed earlier that the main contribution of the regret comes from the early exit time. To prevent the early exit time, the two solutions are either to be more conservative about the resource and introduce a shrinkage term as in \cite{2}, or to take into account the rate leftover resource such that our algorithm is no longer consuming resources linearly. We demonstrate both algorithms here. 
	\begin{algorithm}[H]
	\caption{Conservative Dynamic Learning Algorithm}\label{alg:3}
	\begin{algorithmic}[1]
		\State Input: $d_{1}, \ldots, d_{m}$ where $d_{i}=b_{i} / n$
		\State Initialize: Find $\delta \in(1,2]$ and $L>0$ s.t. $\left\lfloor\delta^{L}\right\rfloor=n$.
		\State Let $t_{k}=\left\lfloor\delta^{k}\right\rfloor, k=1,2, \ldots, L-1$ and $t_{L}=n+1$
		\State Set $x_{1}=\ldots=x_{t_{1}}=0$
		\For{ $k=1,2, \ldots, L-1$ }
		\State Specify an optimization problem
		$$
		\begin{aligned}
			\max & \sum_{j=1}^{t_{k}} r_{j} x_{j} \\
			\text { s.t. } & \sum_{j=1}^{t_{k}} a_{i j} x_{j} \leq \left(1-\epsilon \sqrt{\frac{n}{t_{k}}}\right)t_{k} d_{i}, \quad i=1, \ldots, m \\
			& 0 \leq x_{j} \leq 1, \quad j=1, \ldots, t_{k}
		\end{aligned}
		$$
		\State  Solve its dual problem and obtain the optimal dual variable $\boldsymbol{p}_{k}^{*}$
		$$
		\begin{array}{c}
			\boldsymbol{p}_{k}^{*}=\underset{p}{\arg \min } \sum_{i=1}^{m} d_{i} p_{i}+\frac{1}{t_{k}} \sum_{j=1}^{t_{k}}\left(r_{j}-\sum_{i=1}^{m} a_{i j} p_{i}\right)^{+} \\
			\text {s.t. } p_{i} \geq 0, \quad i=1, \ldots, m
		\end{array}
		$$
		\For{ $t=t_{k}+1, \ldots, t_{k+1}$} 
		\State If constraints permit, set
		$$
		x_{t}=\left\{\begin{array}{ll}
			1, & \text { if } r_{t}>\boldsymbol{a}_{t}^{\top} \boldsymbol{p}_{k}^{*} \\
			0, & \text { if } r_{t} \leq \boldsymbol{a}_{t}^{\top} \boldsymbol{p}_{k}^{*}
		\end{array}\right.
		$$
		\State  $\quad$ Otherwise, set $x_{t}=0$
		\State  $\quad$ If $t=n$, stop the whole procedure.
		\EndFor 
		\EndFor
	\end{algorithmic}
\end{algorithm}
Above algorithm modifies line 6 to include a shrinkage term $\left(1-\epsilon \sqrt{\frac{n}{t_{k}}}\right)$. The idea is that since the cost of early exit ($O(\sqrt{n}\log n)$) is higher than the cost of wasted resource ($O(\sqrt{n}\sqrt{\log \log n})$), an algorithm slightly more conservative with the resource may be better off. However, this imposes a tradeoff because to compensate for operation time we need to pay for wasted resources and potential errors in computing the samples optimal duals.

The regrets and the consumptions rate are given below:

\begin{figure}[!ht]
 		\caption{Regret for Algorithm \ref{alg:3} (Dotted) compared with  Algorithm \ref{alg:2} (Red) with Input I}
  \includegraphics[scale=0.3]{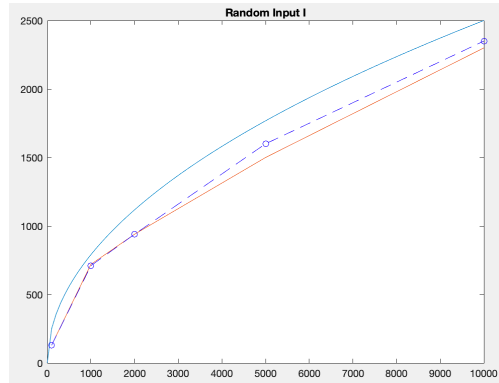}
\centering
\label{figure 7}
\end{figure}

\begin{figure}[!ht]

 		\caption{Resource Consumption Rate of Algorithm \ref{alg:3} (Left) and Zoomed in View (Right) with Input I}
  \includegraphics[scale=0.3]{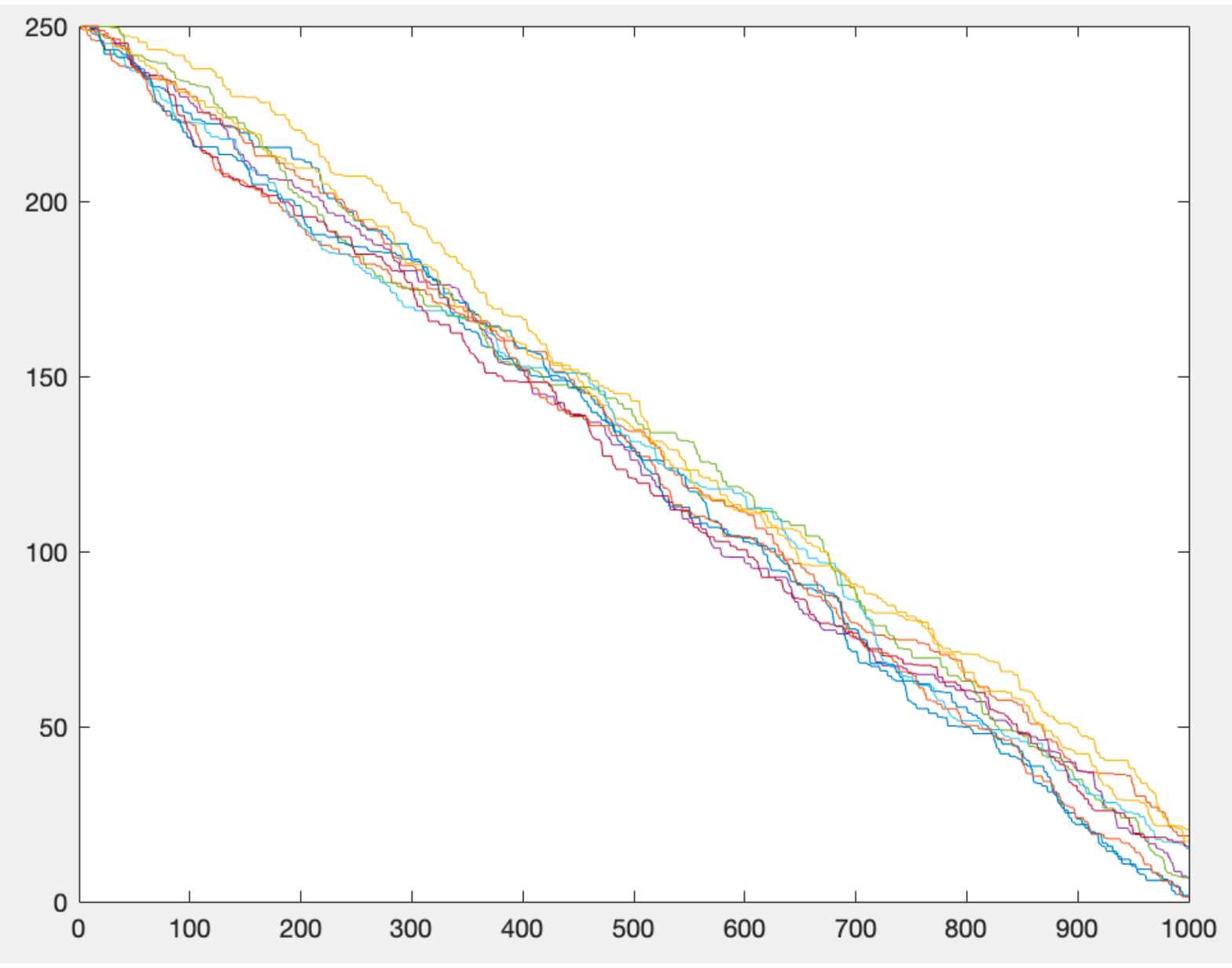}
    \includegraphics[scale=0.3]{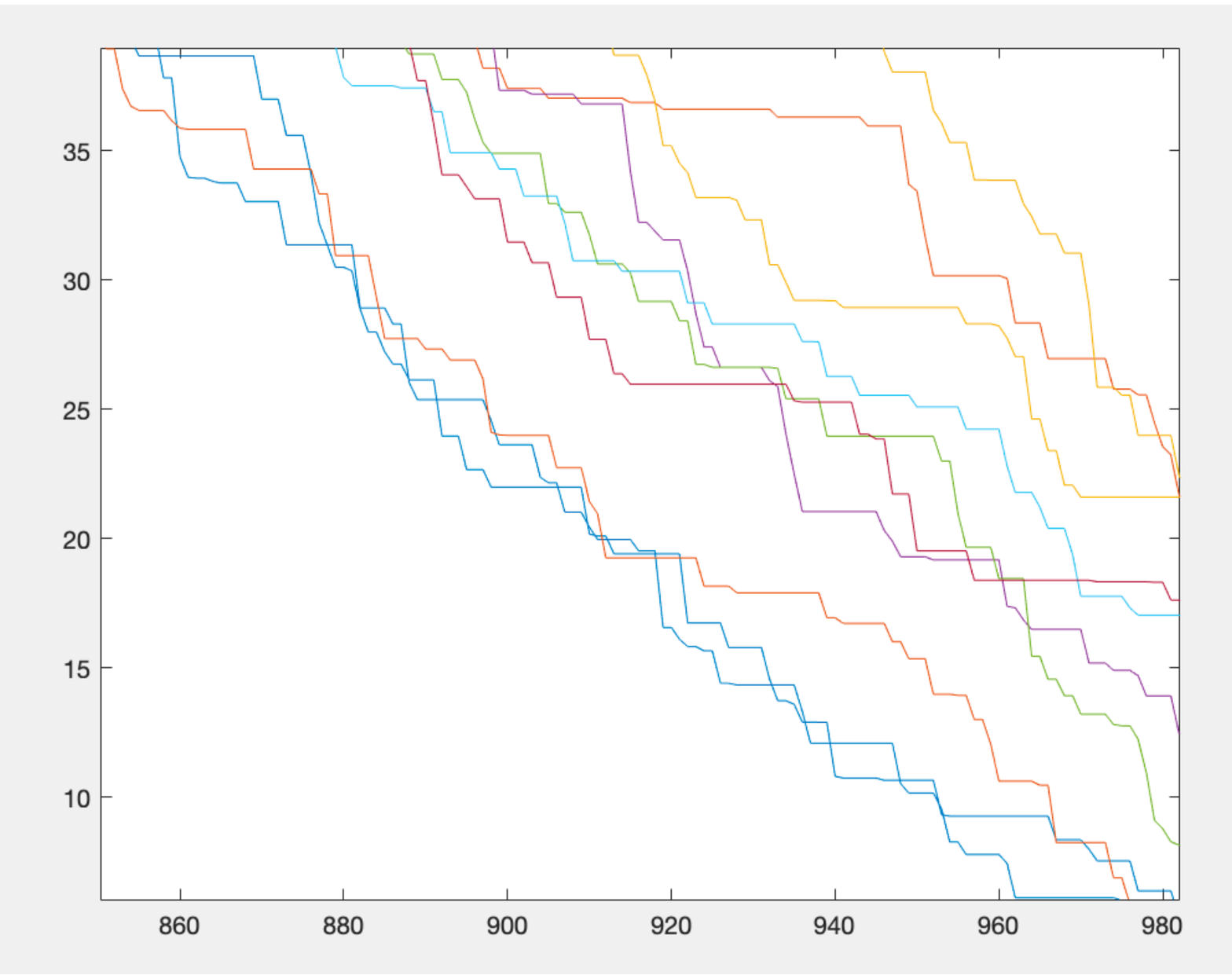} 
\centering
\label{figure 8}
\end{figure}

\begin{figure}[!ht]

 	\caption{Regret for Algorithm \ref{alg:3} (Dotted) compared with  Algorithm \ref{alg:2} (Red) with Input II}
  \includegraphics[scale=0.3]{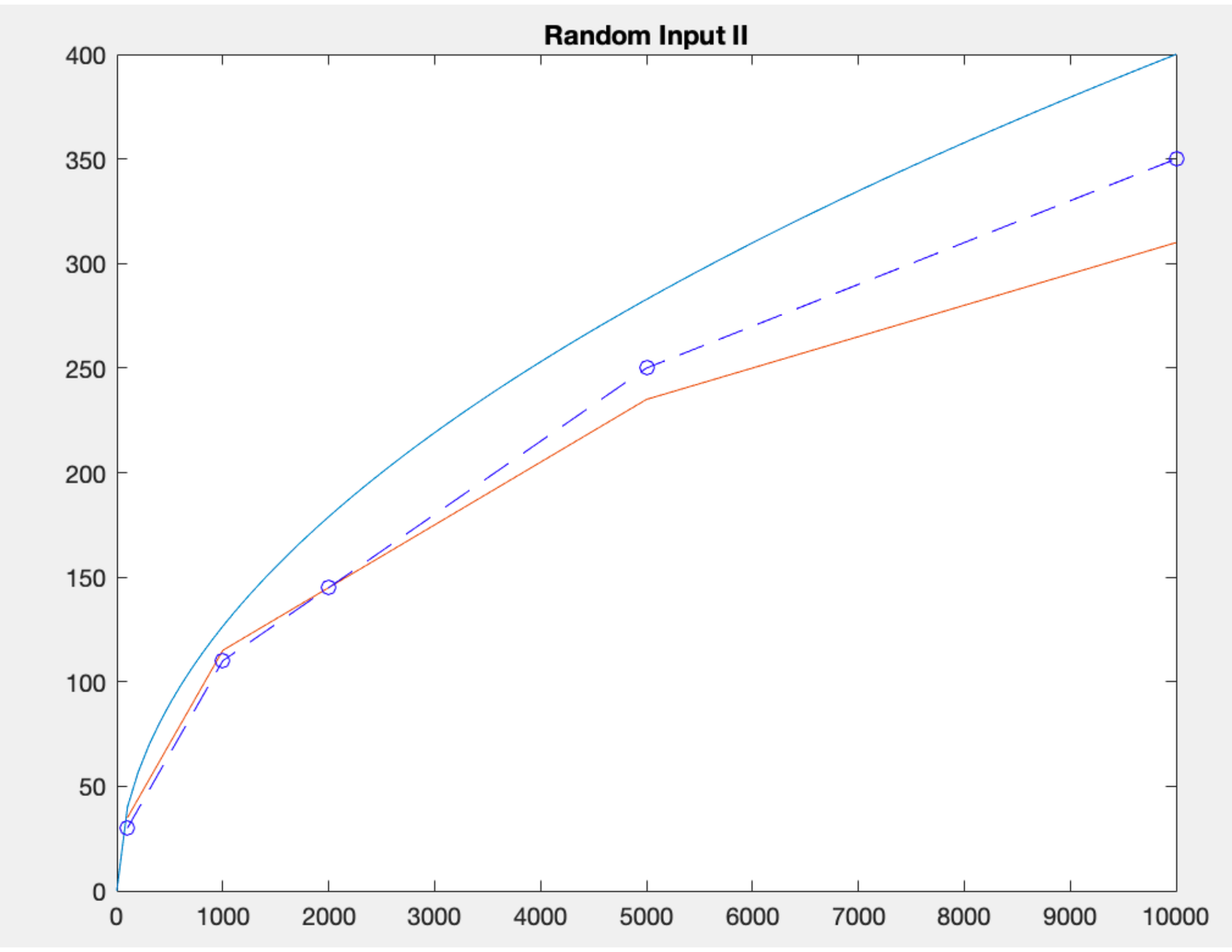} 
\centering
\label{figure 9}
\end{figure}

\begin{figure}[!ht]

 		\caption{Resource Consumption Rate of Algorithm \ref{alg:3} (Left) and Zoomed in View (Right) with Input II}
  \includegraphics[scale=0.3]{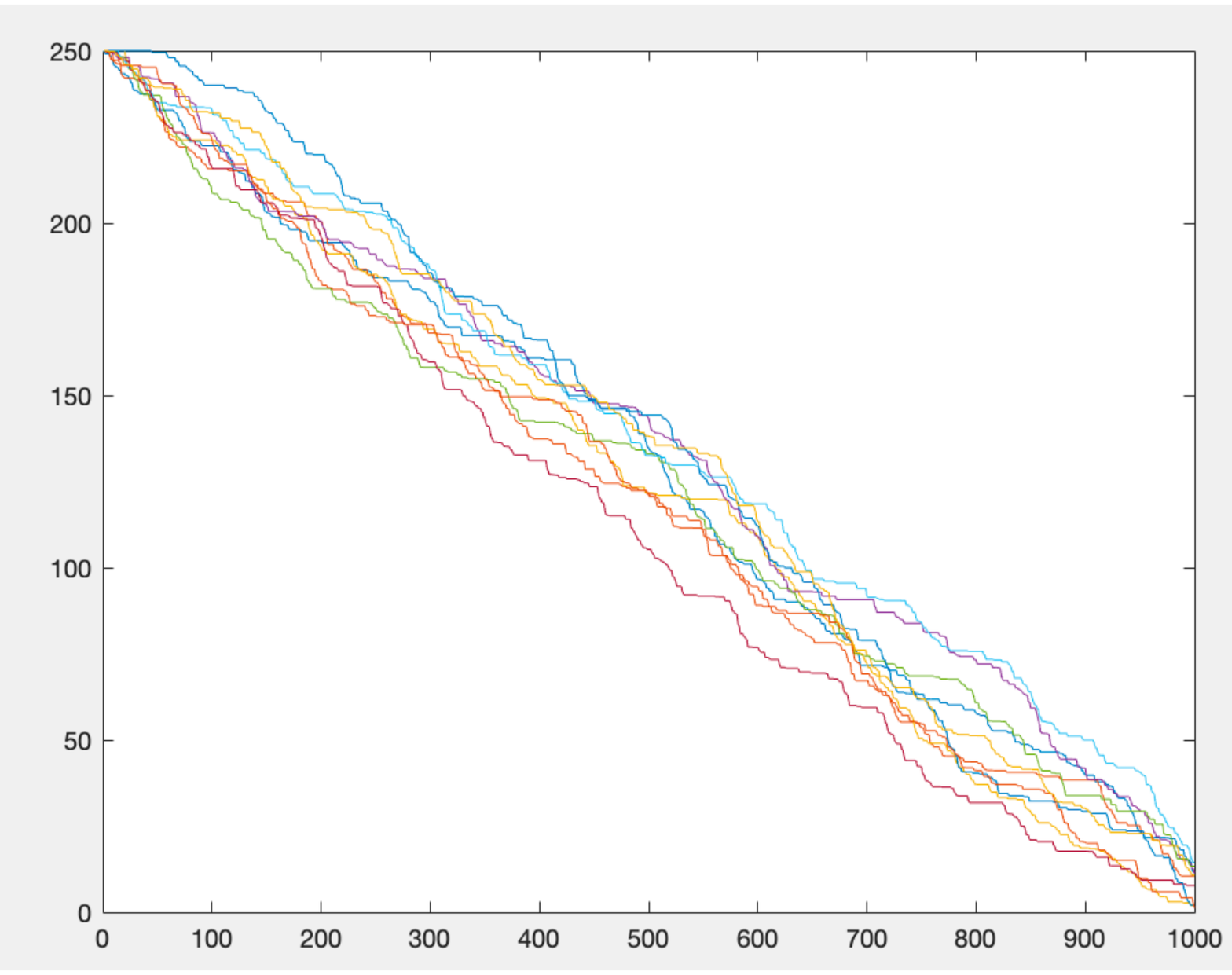} 
    \includegraphics[scale=0.3]{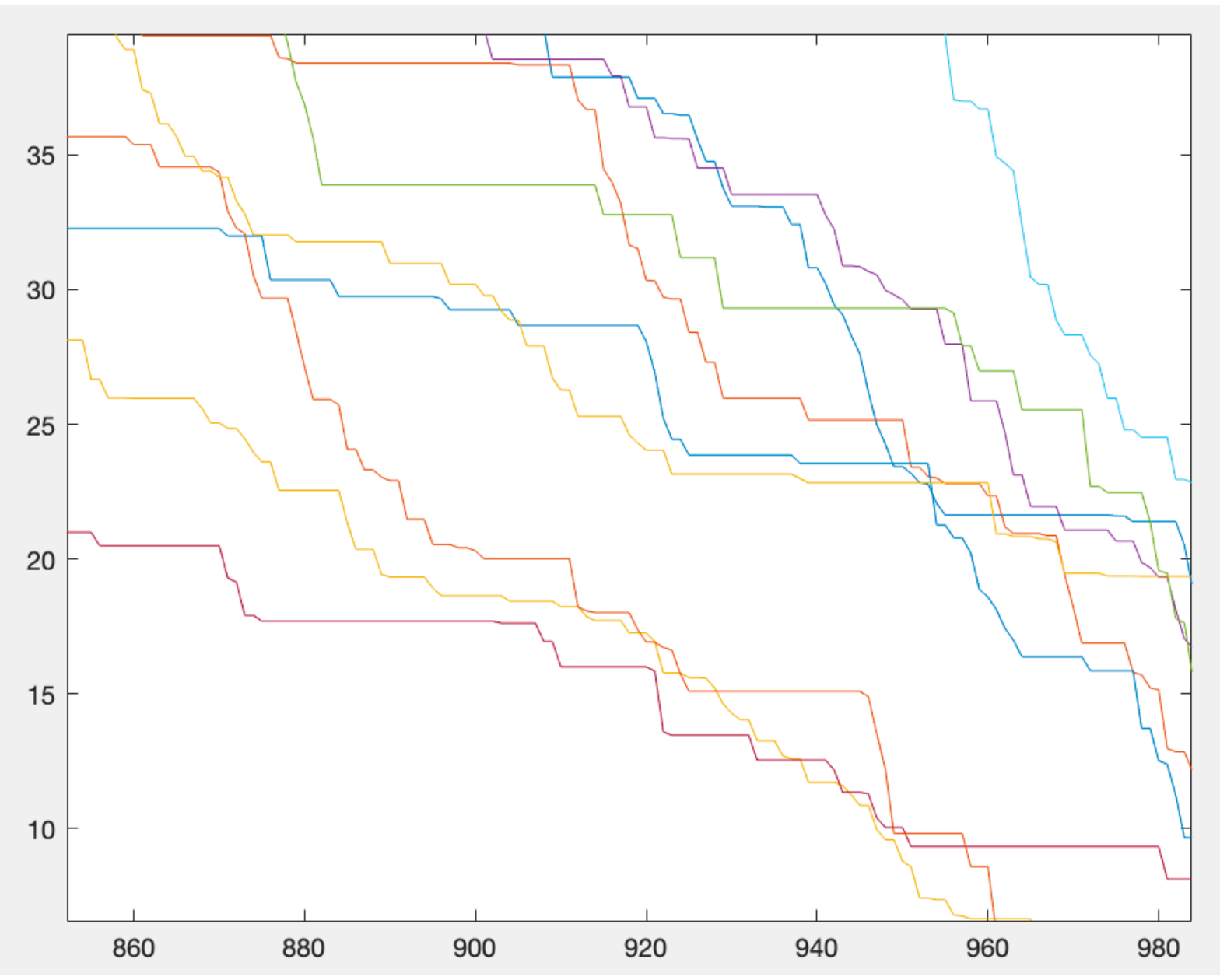} 
\centering
\label{figure 10}
\end{figure}

\begin{figure}[!ht]

 		\caption{Regret for Algorithm \ref{alg:3} (Dotted) compared with  Algorithm \ref{alg:2} (Red) with Input III}
  \includegraphics[scale=0.3]{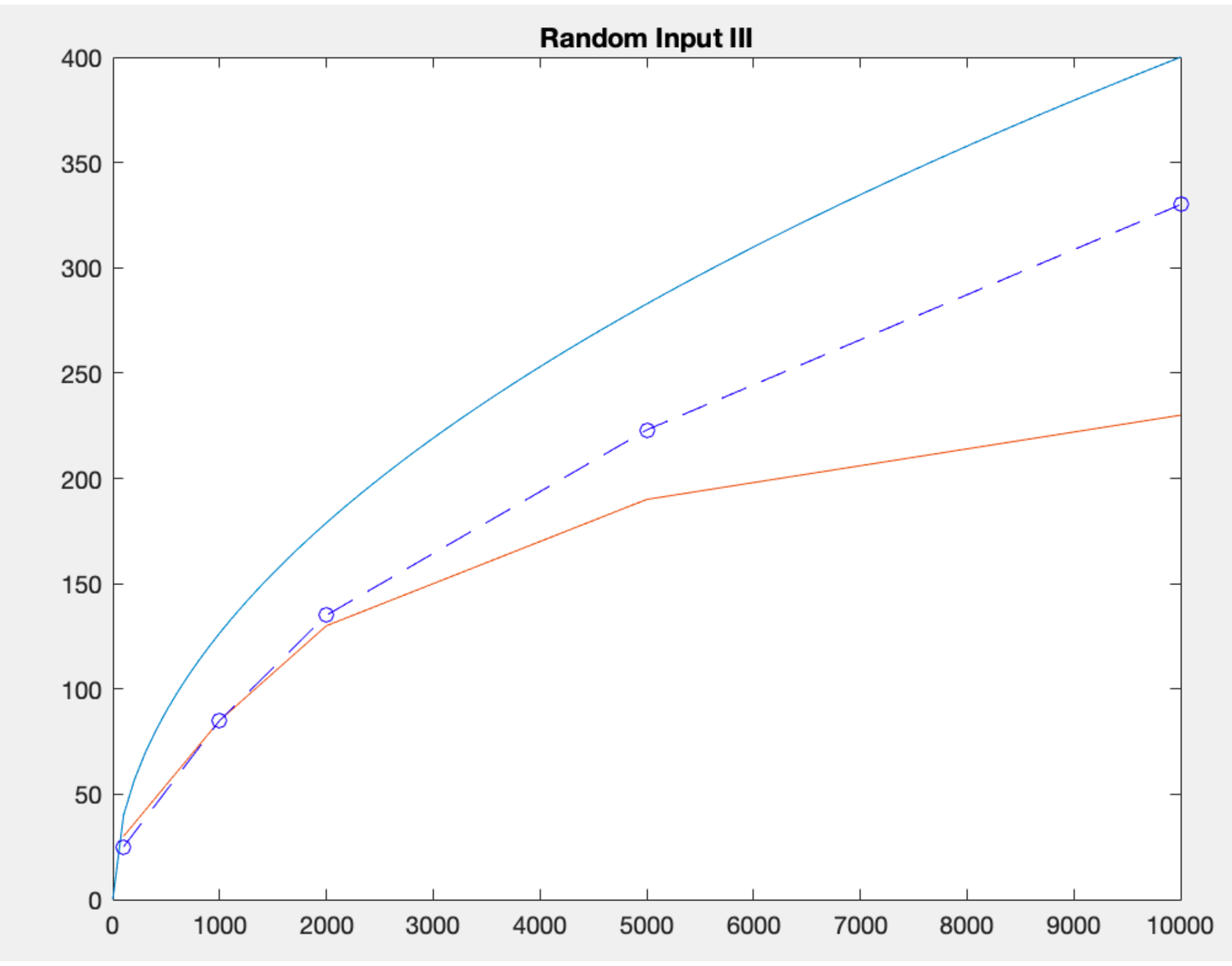}
\centering
\label{figure 11}
\end{figure}

\begin{figure}[!ht]
	\caption{Resource Consumption Rate of Algorithm \ref{alg:3} (Left) and Zoomed in View (Right) with Input III}
 		
  \includegraphics[scale=0.3]{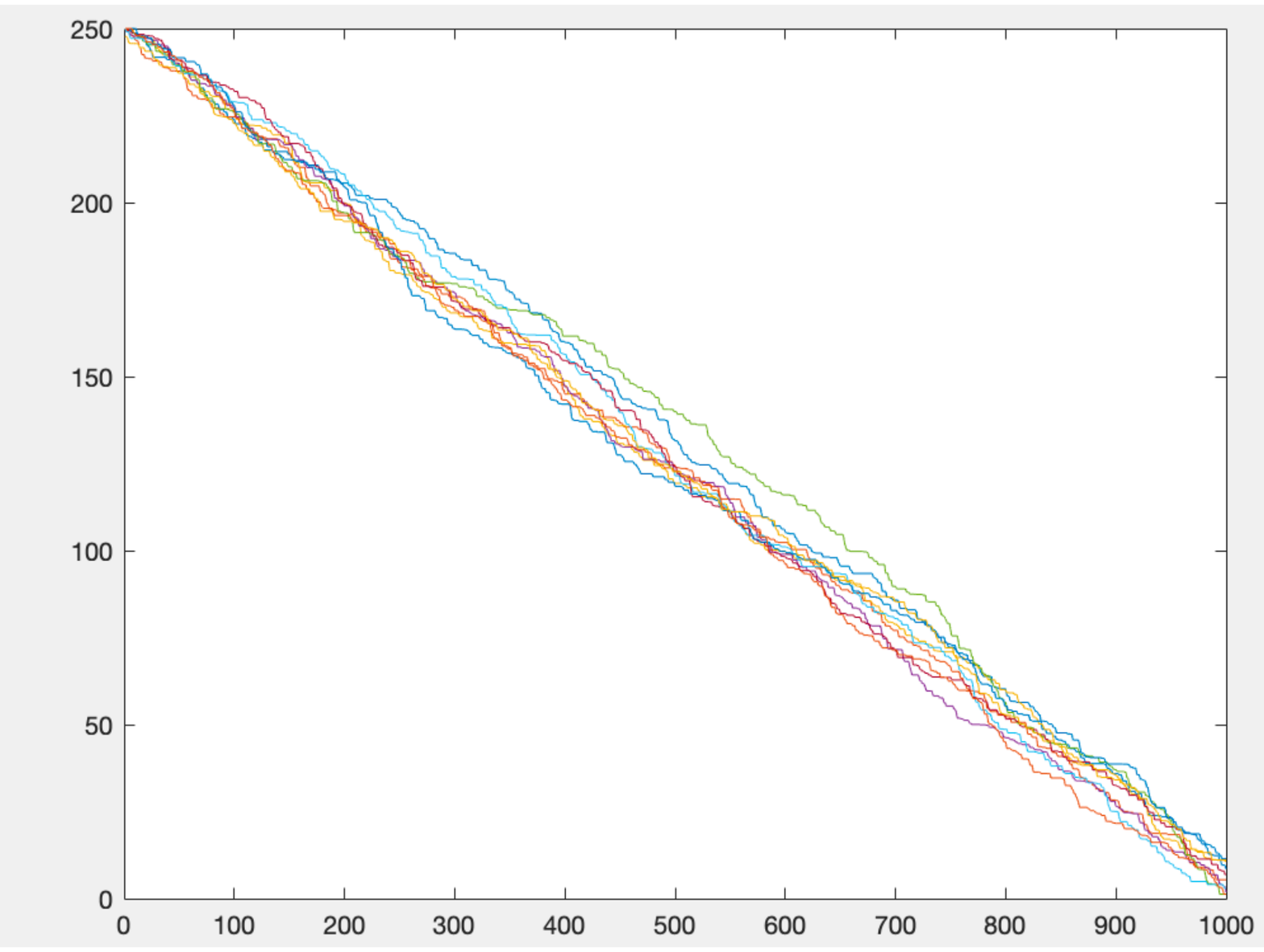} 
    \includegraphics[scale=0.3]{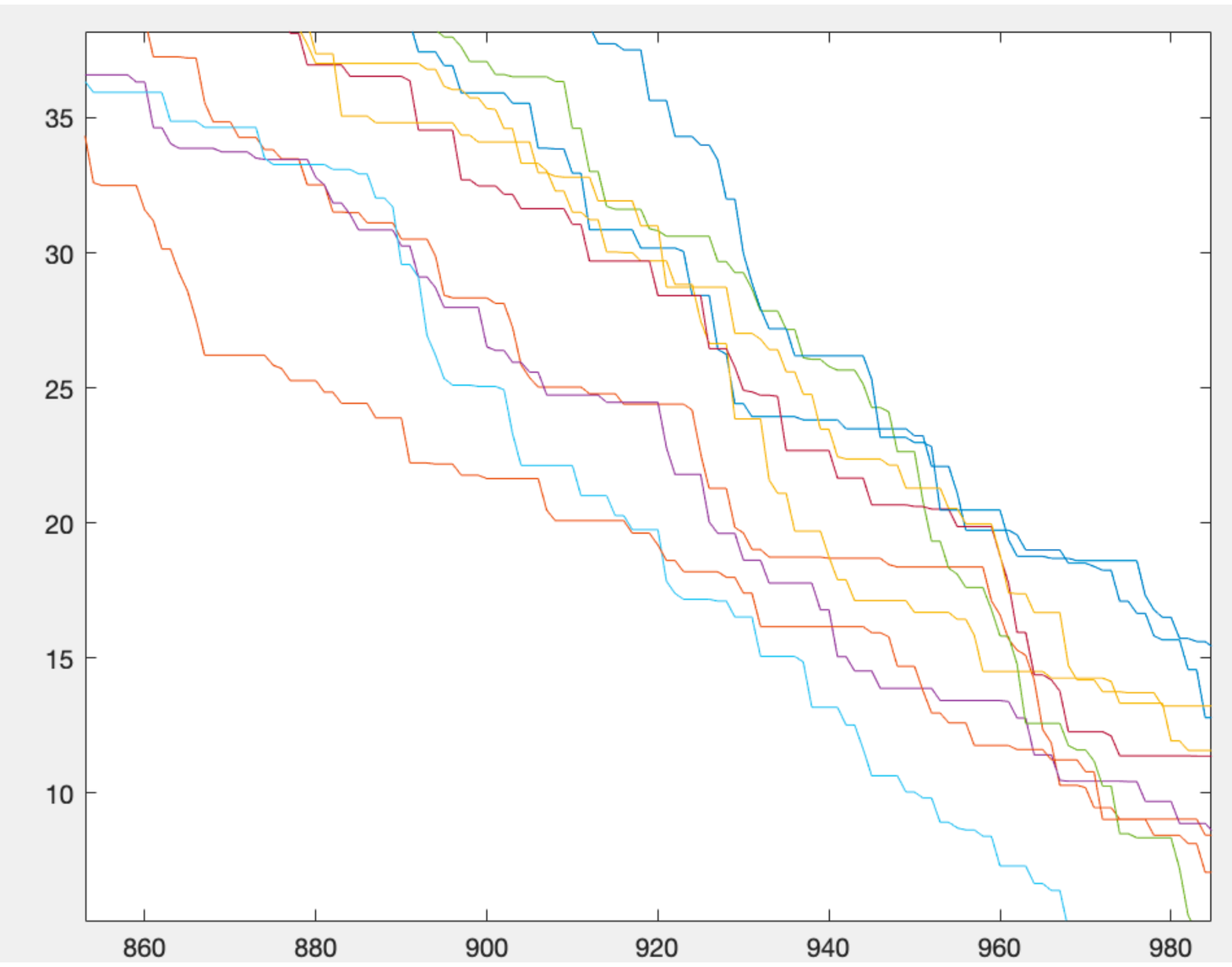}

\centering
\label{figure 12}
\end{figure}

We can observe that though the regrets may be improved when operation period is small, the regrets are actually greater when the period is long across all three random inputs in Figure \ref{figure 7},\ref{figure 9}, and \ref{figure 11} . This can be explained by the fact that when the period is small, a more conservative approach may be better, since the estimation is usually rough. However, when the period is long, there is no need for making special compensations for the estimation error, and a conservative approach is likely to cause long-term underperformance. If we observe the consumption table Figure \ref{figure 8}, \ref{figure 10},\ref{figure 12}, we see the consumption is indeed more conservative, but such conservation does not give the rise to overperformance. Hence, there is generally no need to add a shrinkage term in the algorithm, for the tradeoff of conservation is too high. \\

To solve this trade off, let us consider Action-History-Dependent Learning Algorithm from \cite{5}, which adjusts the optimal dual solution based on the previous actions. Therefore, such algorithm is an adaptive learning algorithm. The advantage for such algorithm, as we will see, compared to our previous algorithm that assumes a normalized consumption rate of $\frac{t}{n} b_{i}=t d_{i}$, is that it compensates the mistakes we made from approximation. For example, if we consume too much resource at first, then it will increase the dual price and slow down the consumption. Recall from theorem \ref{7.2} we know that the regret comes in three parts: the average error, the early depletion, and the wasted resources. This adaptive algorithms will significantly decrease the regret coming both from the early depletion and the wasted resources, for it adjusts its consumption based on real leftover resources instead of following the normalized rate. We present the algorithm as the following:
\begin{algorithm}[H]
	\caption{Action-History-Dependent Learning Algorithm}\label{alg:4}
	\begin{algorithmic}[1]
		\State Input: $n, d_{1}, \ldots, d_{m}$
		\State Initialize the constraint $b_{i 0}=n d_{i}$ for $i=1, \ldots, m$
		\State Initialize the dual price $\boldsymbol{p}_{1}=\mathbf{0}$.
		\For {$t=1, \ldots, n$ }
		\State $\quad$ Observe $\left(r_{t}, \boldsymbol{a}_{t}\right)$ and set
		$$
		x_{t}=\left\{\begin{array}{ll}
			1, & \text { if } r_{t}>\boldsymbol{a}_{t}^{\top} \boldsymbol{p}_{t} \\
			0, & \text { if } r_{t} \leq \boldsymbol{a}_{t}^{\top} \boldsymbol{p}_{t}
		\end{array}\right.
		$$
		\State If the constraints are not violated
		\State Update the constraint vector
		$$
		b_{i t}=b_{i, t-1}-a_{i t} x_{t} \text { for } i=1, \ldots, m
		$$
		\State $\quad$ Specify an optimization problem
		$$
		\begin{aligned}
			\max & \sum_{j=1}^{t} r_{j} x_{j} \\
			\text { s.t. } & \sum_{j=1}^{t} a_{i j} x_{j} \leq \frac{t b_{i t}}{n-t}, \quad i=1, \ldots, m \\
			& 0 \leq x_{j} \leq 1, \quad j=1, \ldots, t
		\end{aligned}
		$$
		\State $\quad$ If $t<n$, solve its dual problem and obtain the dual price $\boldsymbol{p}_{t+1}$
		$$
		\begin{array}{c}
			\boldsymbol{p}_{t+1}=\underset{\boldsymbol{p}}{\arg \min } \sum_{i=1}^{m} \frac{b_{i t} p_{i}}{n-t}+\frac{1}{t} \sum_{j=1}^{t}\left(r_{j}-\sum_{i=1}^{m} a_{i j} p_{i}\right)^{+} \\
			\text {s.t. } p_{i} \geq 0, \quad i=1, \ldots, m .
		\end{array}
		$$
		\EndFor
	\end{algorithmic}
\end{algorithm}
The performances are recorded below:
\begin{figure}[!ht]
 		\caption{Regret for Algorithm \ref{alg:4} (Dotted) compared with  Algorithm \ref{alg:2} (Red) with Input I}
  \includegraphics[scale=0.3]{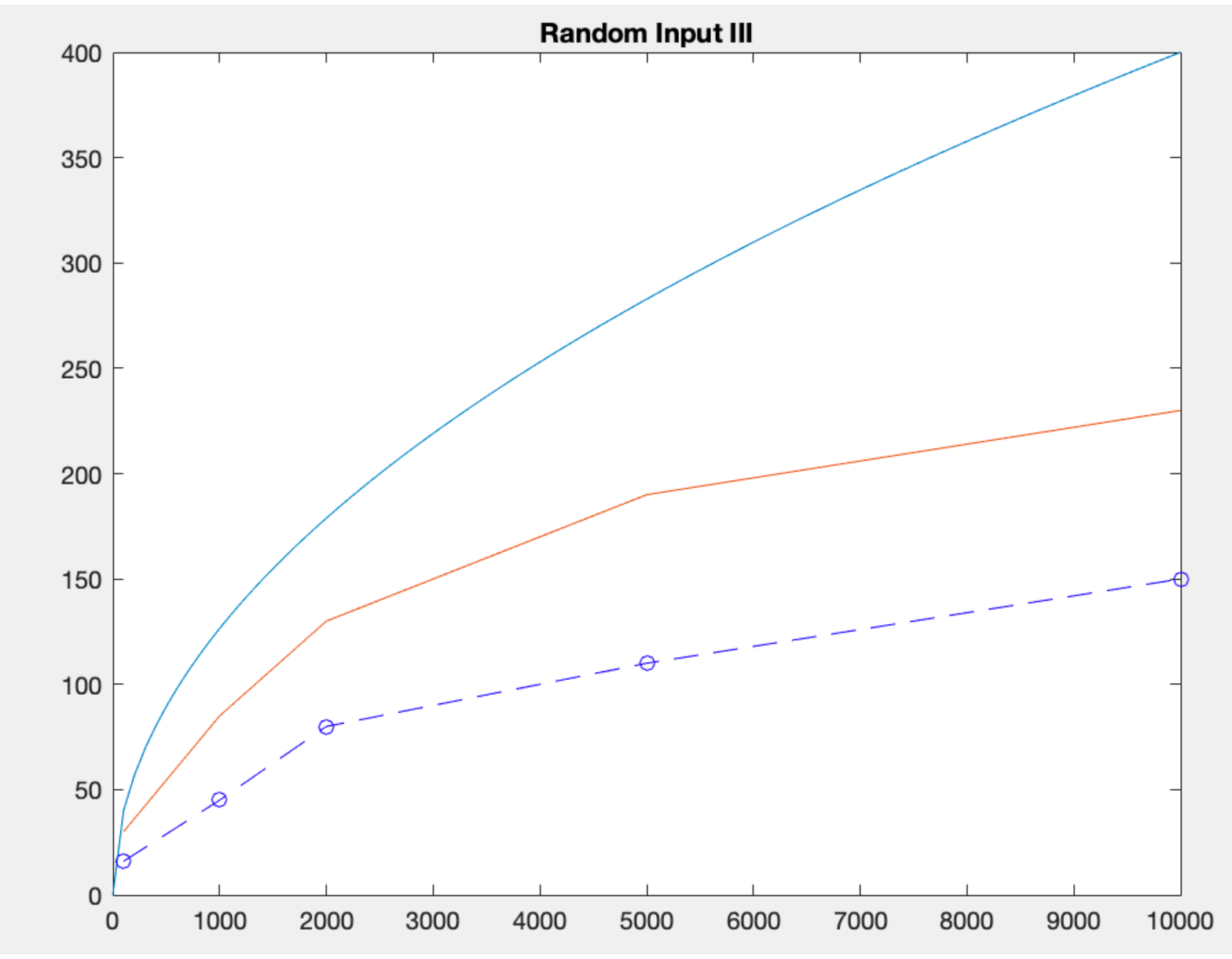}
\centering
\label{figure 13}
\end{figure}

\begin{figure}[!ht]

 		\caption{Resource Consumption Rate of Algorithm \ref{alg:4} (Left) and Zoomed in View (Right) with Input I}
  \includegraphics[scale=0.3]{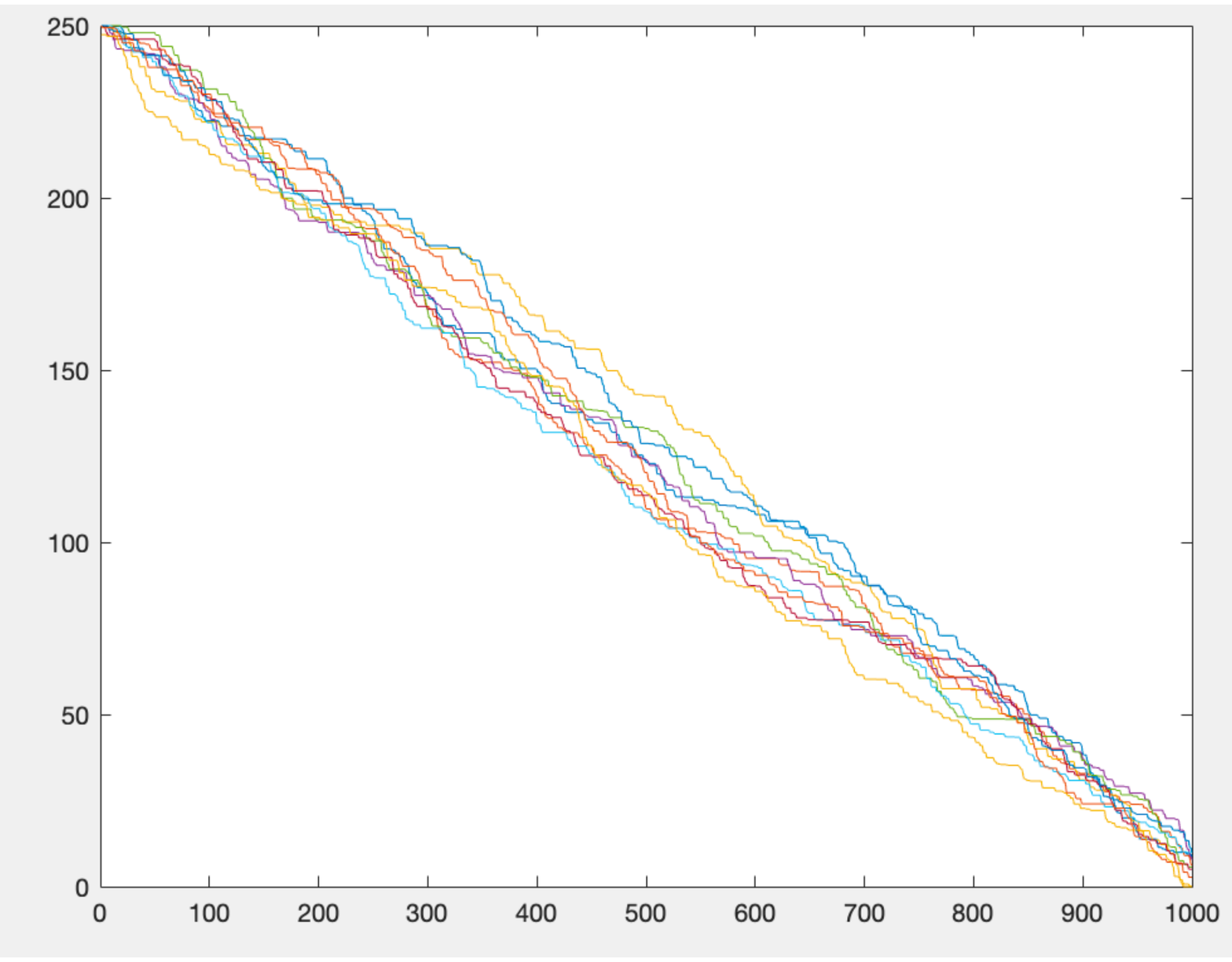}
    \includegraphics[scale=0.3]{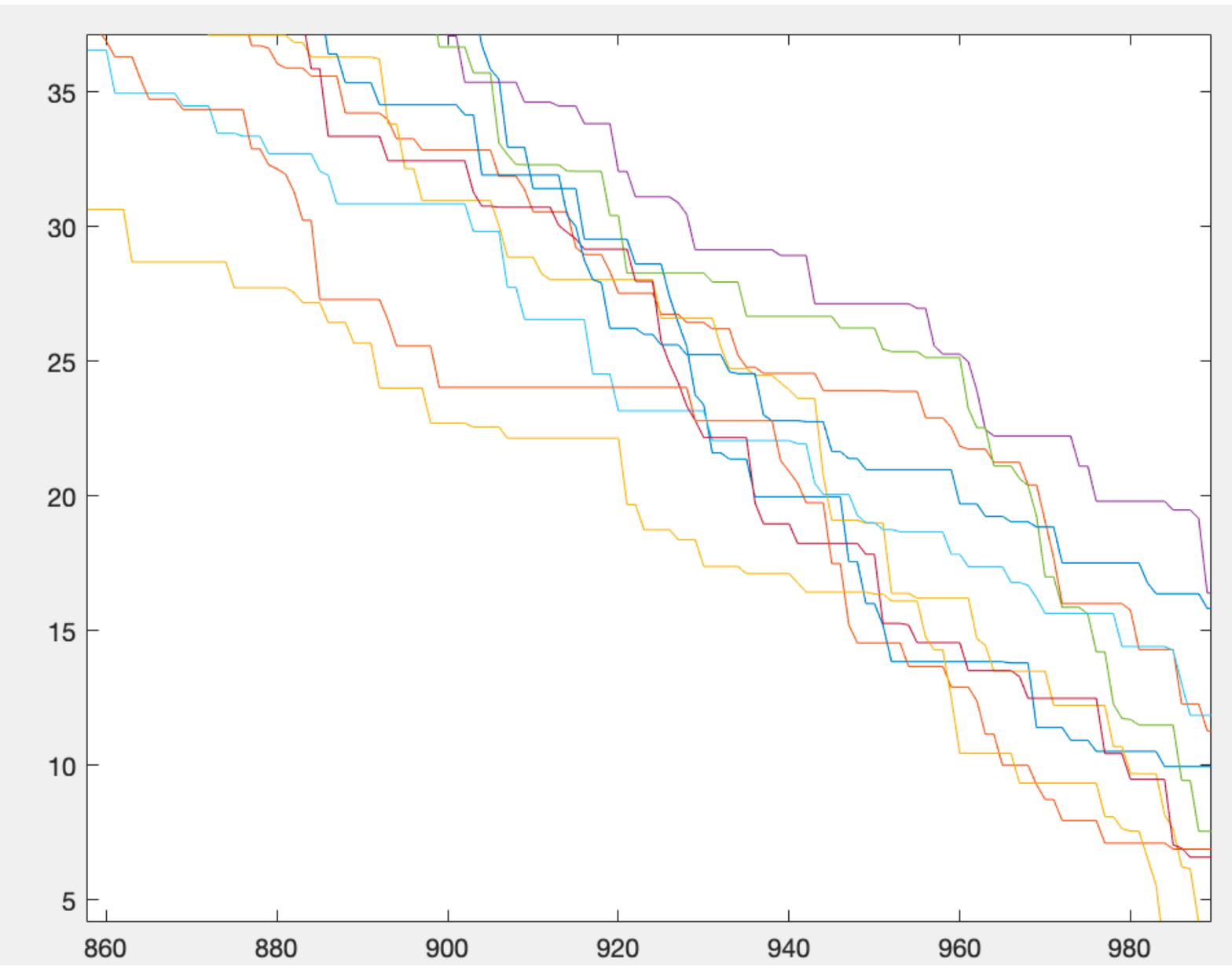} 
\centering
\label{figure 14}
\end{figure}

\begin{figure}[!ht]

 	\caption{Regret for Algorithm \ref{alg:4} (Dotted) compared with  Algorithm \ref{alg:2} (Red) with Input II}
  \includegraphics[scale=0.3]{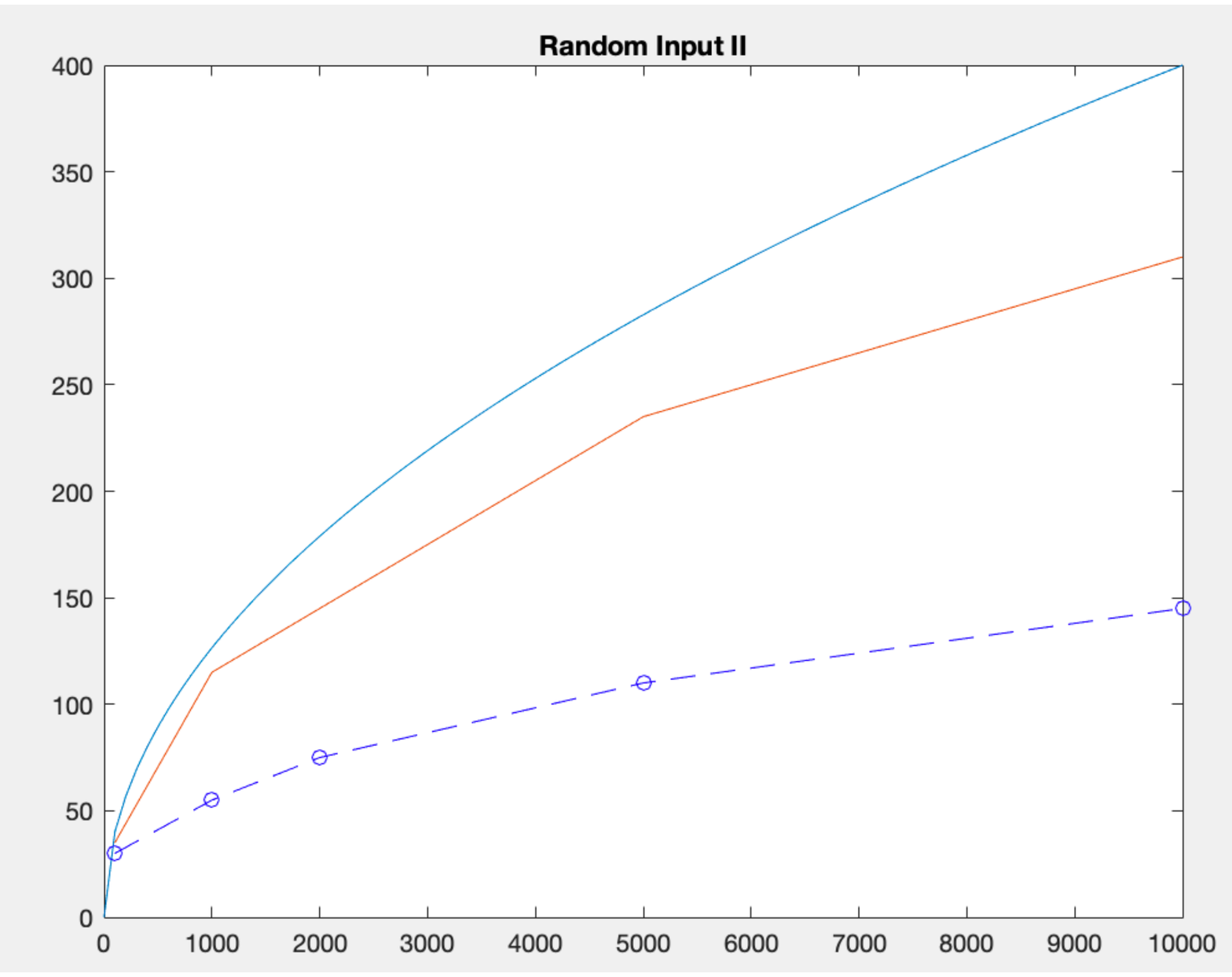} 
\centering
\label{figure 15}
\end{figure}

\begin{figure}[!ht]

 		\caption{Resource Consumption Rate of Algorithm \ref{alg:4} (Left) and Zoomed in View (Right) with Input II}
  \includegraphics[scale=0.3]{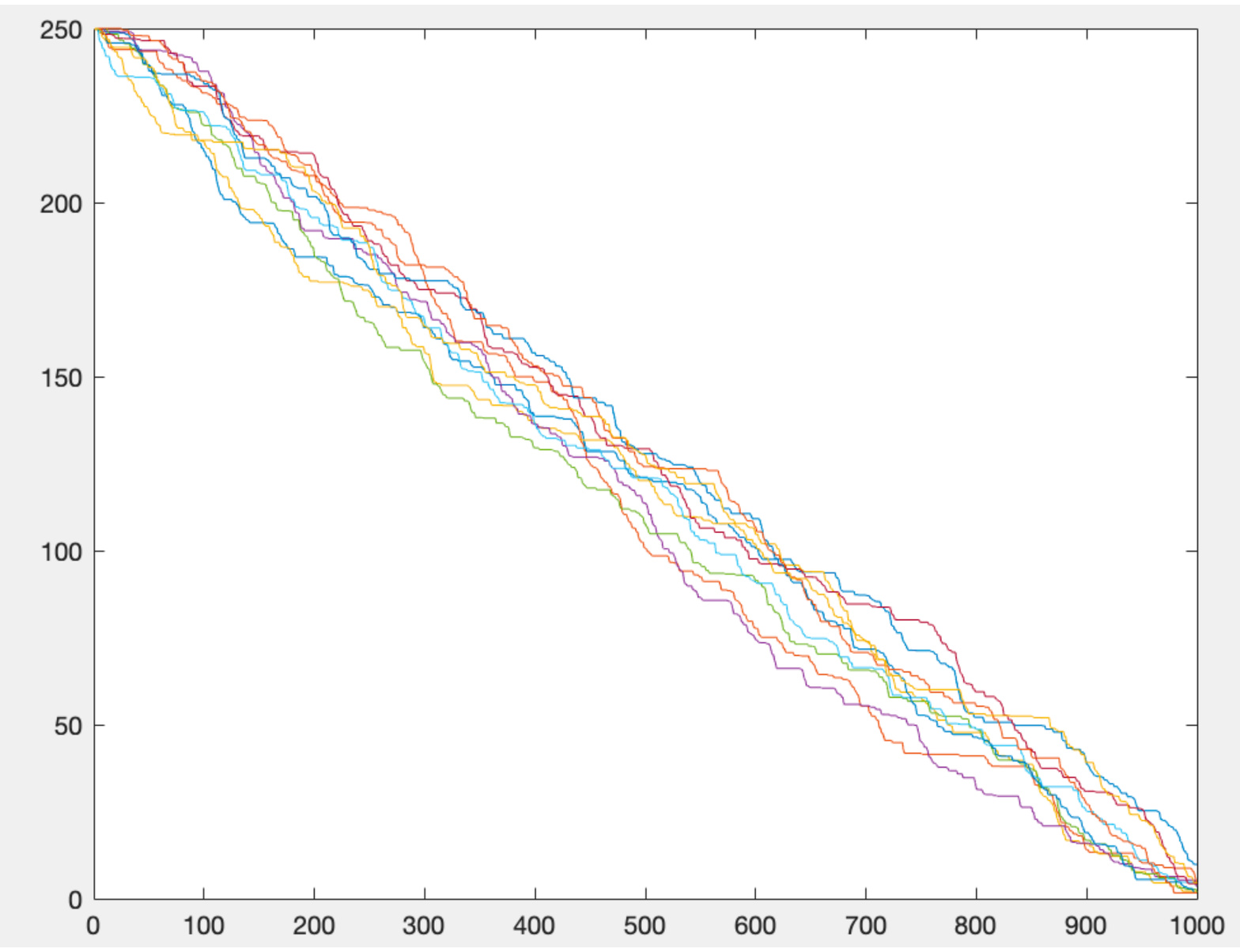} 
    \includegraphics[scale=0.3]{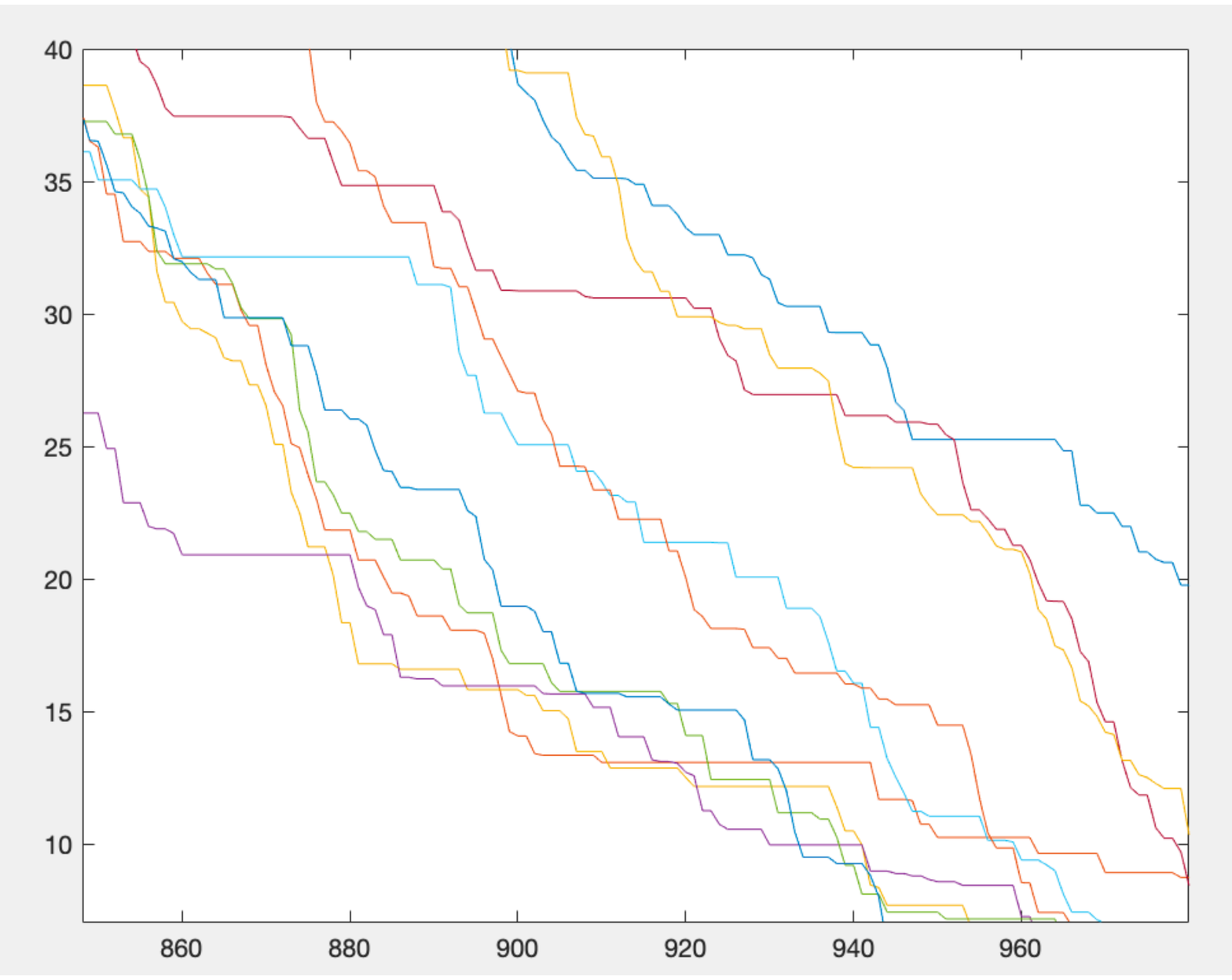} 
\centering
\label{figure 16}
\end{figure}

\begin{figure}[!ht]

 		\caption{Regret for Algorithm \ref{alg:4} (Dotted) compared with  Algorithm \ref{alg:2} (Red) with Input III}
  \includegraphics[scale=0.3]{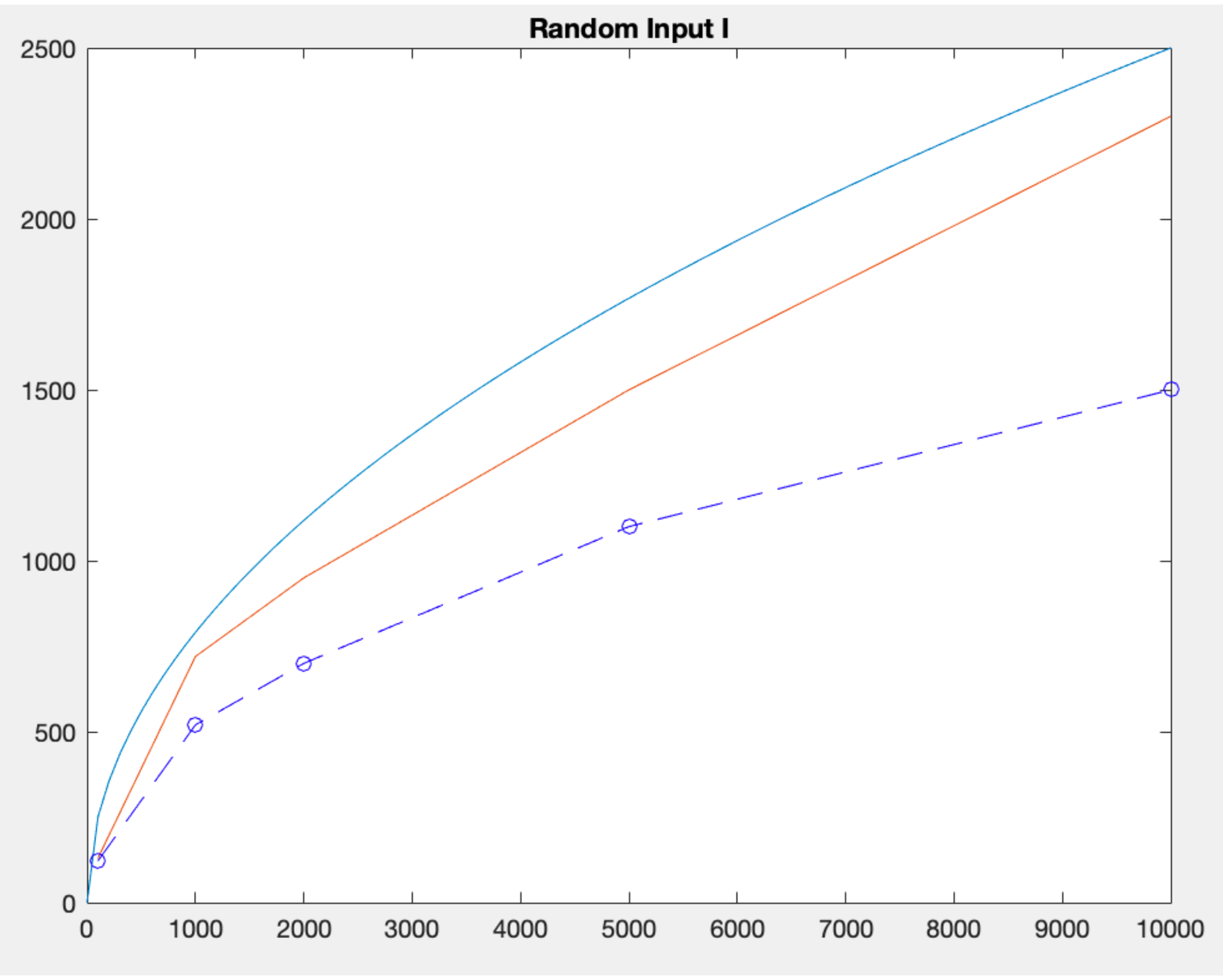}
\centering
\label{figure 17}
\end{figure}

\begin{figure}[!ht]
	\caption{Resource Consumption Rate of Algorithm \ref{alg:4} (Left) and Zoomed in View (Right) with Input III}
 		
  \includegraphics[scale=0.3]{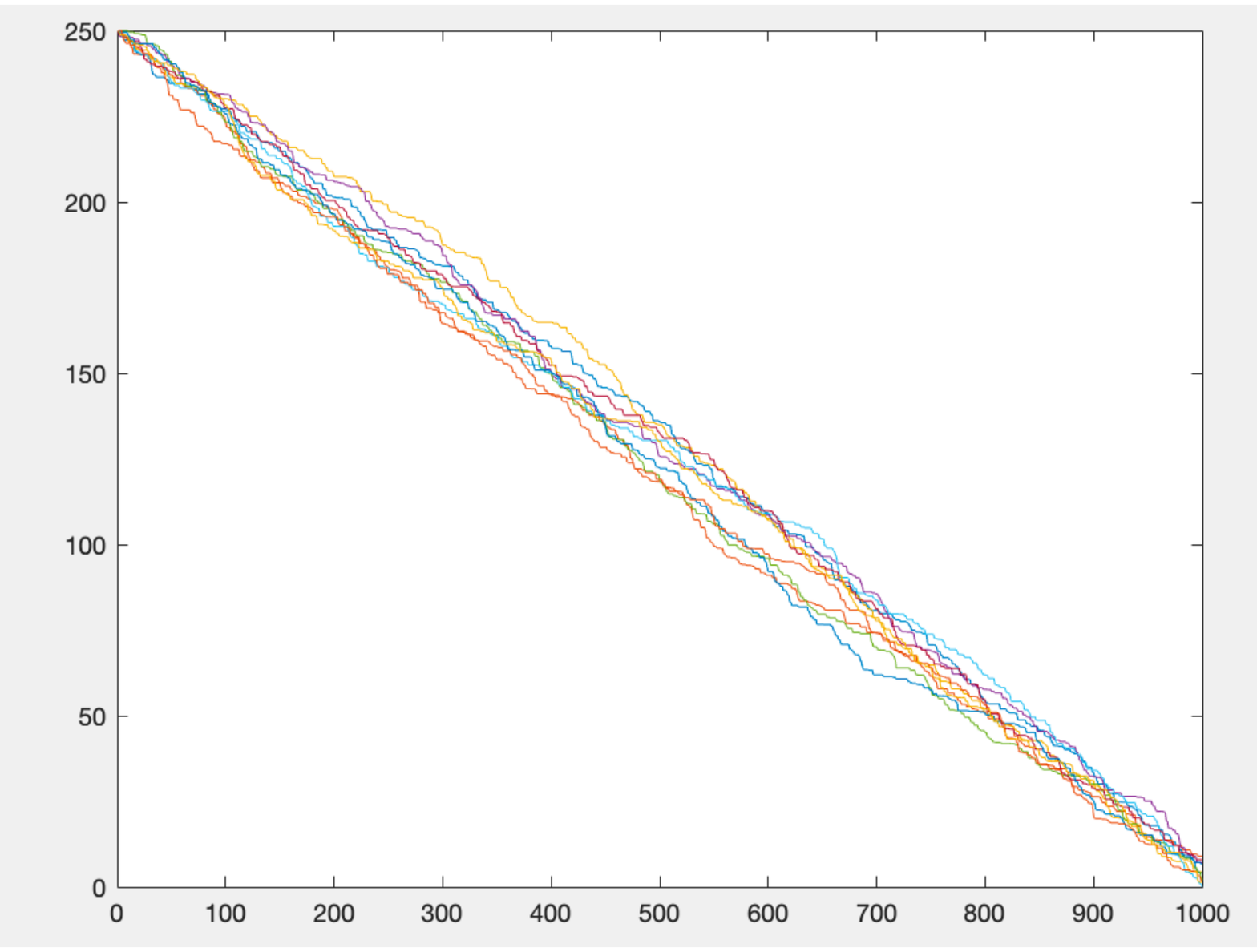} 
    \includegraphics[scale=0.3]{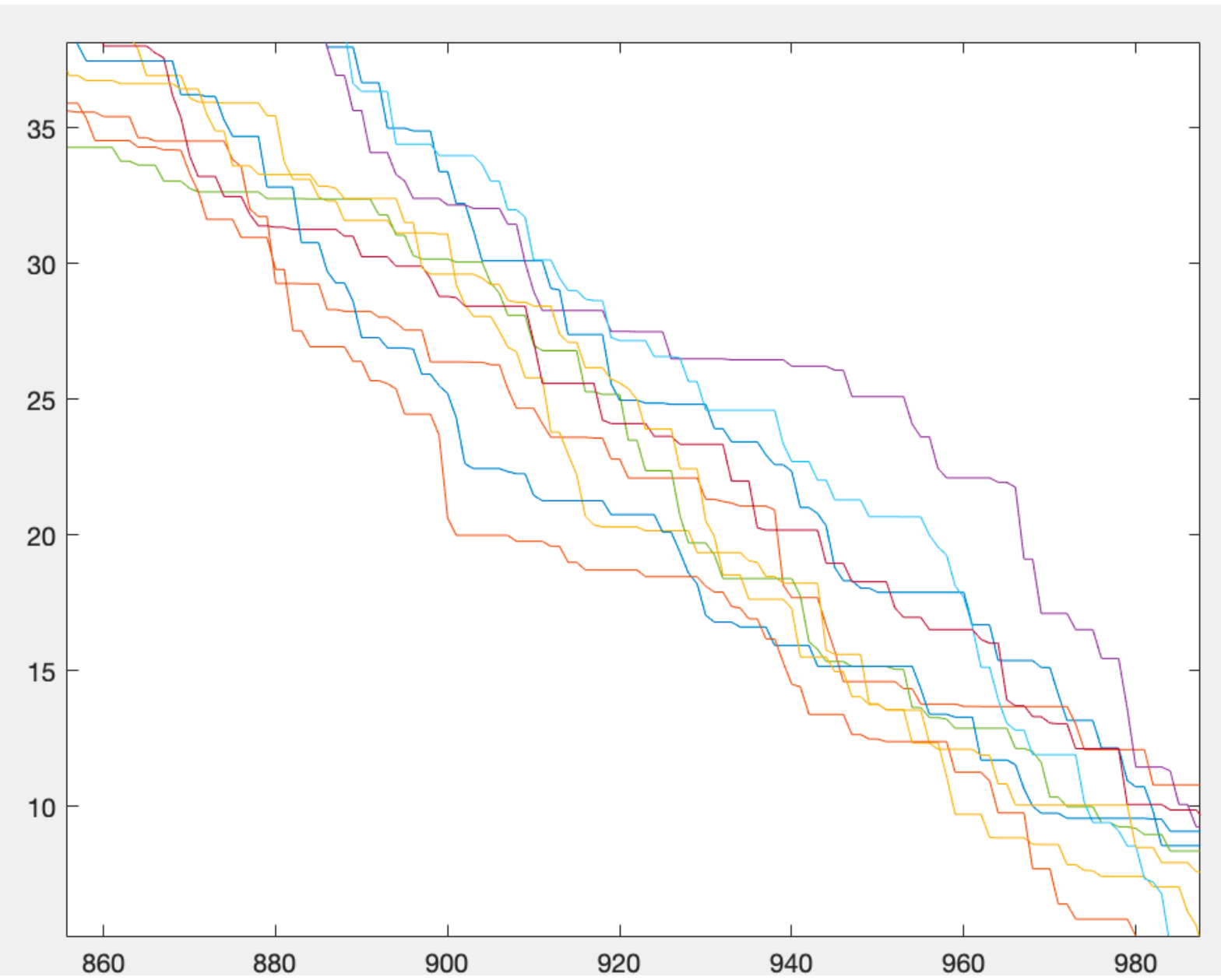}

\centering
\label{figure 18}
\end{figure}
As we can observe, both the regrets and the consumption smoothness are improved across three random inputs. The adaptive algorithm no longer aims for computing the real dual value across the entire horizon, for computing the dual for the past time periods is not relevant for making decisions in the future. By computing only the relevant dual based on the real leftover resources, this algorithm is more efficient by discarding irrelevant information. As a result, this algorithm no longer assumes linear consumption. When there is little resources left, the dual value is driven up which slows down the consumption rate; when storage is too large, the dual value is driven down to accept more orders. This adaptive feature allows the algorithm, in all three random inputs, to finish its resources almost exactly at the end time.

\clearpage
\section{Open Problems}

One of the essential questions to ask at this stage is whether we can design algorithms suitable for non-stationary price data, for example, trendy data. All of our algorithms and the algorithms raised by \cite{5} focus exclusively on stationary data. The ability to analyze non-stationary data is critical in the application. \cite{27} and \cite{28} demonstrate resource allocation with non-stationary data from online video-streaming and online time series data respectively. \cite{29} and \cite{30} propose and analyze algorithms that solve non-stationary linear programming problems on modern computing clusters. The promising step forward is to answer whether our algorithms can be adaptive for non-stationary data. 

To start with, we consider two types of trendy data: i) a weighted random walk and ii) a linear regression model with noise. We take the dimension of the products as $2$ and the capacities to be $0.25n$. So on average, the algorithm can accept a fourth of total orders. 
The details are given below:
\begin{tabularx}{1\textwidth} { 
		| >{\raggedright\arraybackslash}X 
		| >{\centering\arraybackslash}X 
		| >{\raggedleft\arraybackslash}X | }
	\hline
	Random Input IV  & $A_{ij} \sim \text{Uniform}(0.6, 1.4) $ & $r_i\sim r_{i-1}+0.2+\text{Uniform}(-0.2, 0.2)$\\
	\hline
	Random Input V & $A_{ij} \sim \text{Uniform}(0.6, 1.4) $ &  $r_i\sim 1+0.2i+\text{Uniform}(-0.2, 0.2)$\\
	\hline
\end{tabularx}
If we test the data using Algorithm \ref{alg:4}, we have the following regret:

\begin{figure}[!ht]

 		\caption{Regrets for Algorithm \ref{alg:4} with Random Input IV/V }
  \includegraphics[scale=0.3]{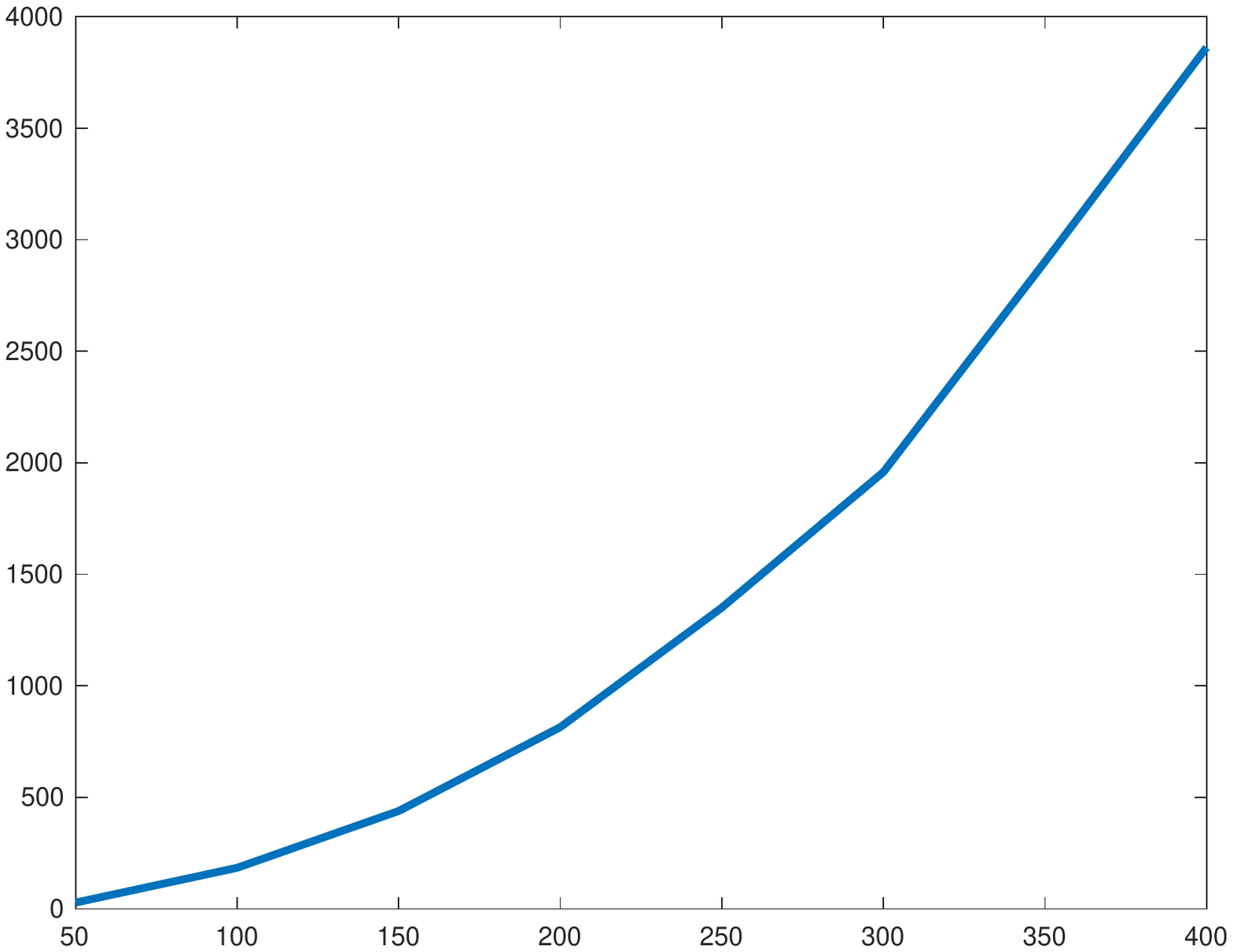}
    \includegraphics[scale=0.3]{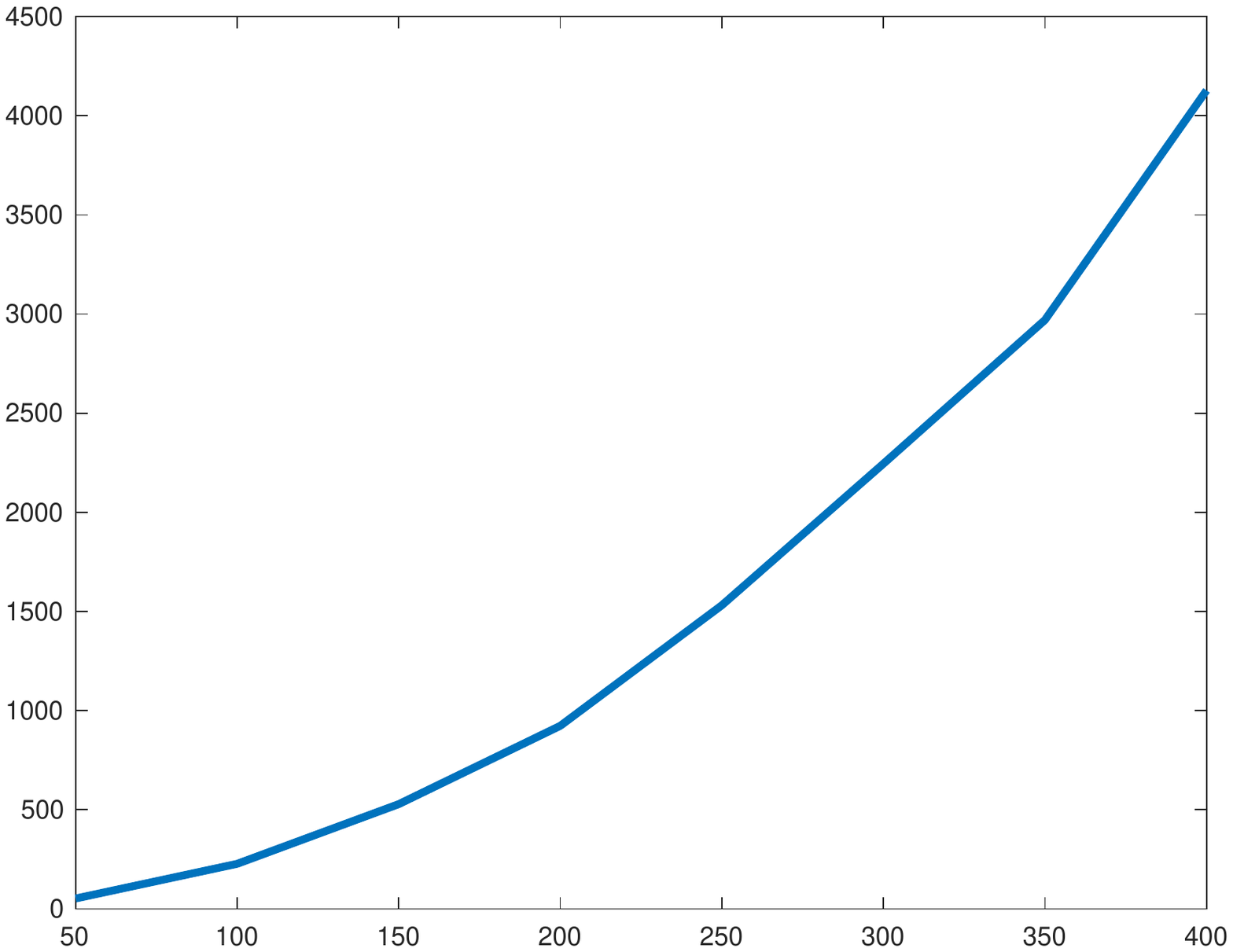} 
\centering
\label{figure 20}
\end{figure}
As we can observe, the regrets are super-linear, for the stationary dual algorithm can no longer cope with non-stationary data. In fact, this super-linear regret is a result of the misleading dual computed in early time, which, instead of gaining more information, provides additional noises. To handle trendy data, we need to force our algorithms to be trend-adaptive, namely to have the ability to predict the trend before computing for the dual. We can design the algorithm in the following way:

\begin{algorithm}
	\caption{Trend-Adaptive Action-History-Dependent Learning Algorithm}\label{alg:5}
	\begin{algorithmic}[1]
		\State Input: $d_{1}, \ldots, d_{m}$ where $d_{i}=b_{i} / n$
		\State Initialize: Find $\delta \in(1,2]$ and $L>0$ s.t. $\left\lfloor\delta^{L}\right\rfloor=n$.
		\State Let $t_{k}=\left\lfloor\delta^{k}\right\rfloor, k=1,2, \ldots, L-1$ and $t_{L}=n+1$
		\State Set $x_{1}=\ldots=x_{t_{1}}=0$
		\For{ $k=1,2, \ldots, L-1$ }
		\State Input $A_i,r_i$ for $i\leq t_k$.
		\State Predict the trend $\alpha_{t_k}$ with intercept $b_{t_k}$ using Least Square Method.
		\State Complete the data  $r_i=i\alpha_{t_k}+b_{t_k}$ for $n\geq i> t_k$.
	\State Complete the data  $A_i=\frac{1}{t_k}\sum_{j=1}^{t_k}A_j$ for $n\geq i> t_k$.
		
		\State Update the real-time constraints $b_{t_k}=b-Ax$.
		\State Specify an optimization problem with simulated data. 
		$$
		\begin{aligned}
			\max & \sum_{j=t_k}^{n} r_{j} x_{j} \\
			\text { s.t. } & \sum_{j=t_k}^{n} a_{i j} x_{j} \leq b_{t_k}, \quad i=1, \ldots, m \\
			& 0 \leq x_{j} \leq 1, \quad j=1, \ldots, t_{k}
		\end{aligned}
		$$
		\State  Solve its dual problem and obtain the optimal dual variable $\boldsymbol{p}_{k}^{*}$
		$$
		\begin{array}{c}
			\boldsymbol{p}_{k}^{*}=\underset{p}{\arg \min }\ b_{t_k}p+ \sum_{j=t_k}^{n}\left(r_{j}-\sum_{i=1}^{m} a_{i j} p_{i}\right)^{+} \\
			\text {s.t. } p_{i} \geq 0, \quad i=1, \ldots, m
		\end{array}
		$$
		\For{ $t=t_{k}+1, \ldots, t_{k+1}$} 
		\State If constraints permit, set
		$$
		x_{t}=\left\{\begin{array}{ll}
			1, & \text { if } r_{t}>\boldsymbol{a}_{t}^{\top} \boldsymbol{p}_{k}^{*} \\
			0, & \text { if } r_{t} \leq \boldsymbol{a}_{t}^{\top} \boldsymbol{p}_{k}^{*}
		\end{array}\right.
		$$
		\State  $\quad$ Otherwise, set $x_{t}=0$
		\State  $\quad$ If $t=n$, stop the whole procedure.
		\EndFor 
		\EndFor
	\end{algorithmic}
\end{algorithm}
\newpage
The performance is recorded below.
\begin{figure}[!ht]

 		\caption{Regrets for Algorithm \ref{alg:5}(Red) with Random Input IV/V compared to Algorithm \ref{alg:4}(Blue)}
  \includegraphics[scale=0.3]{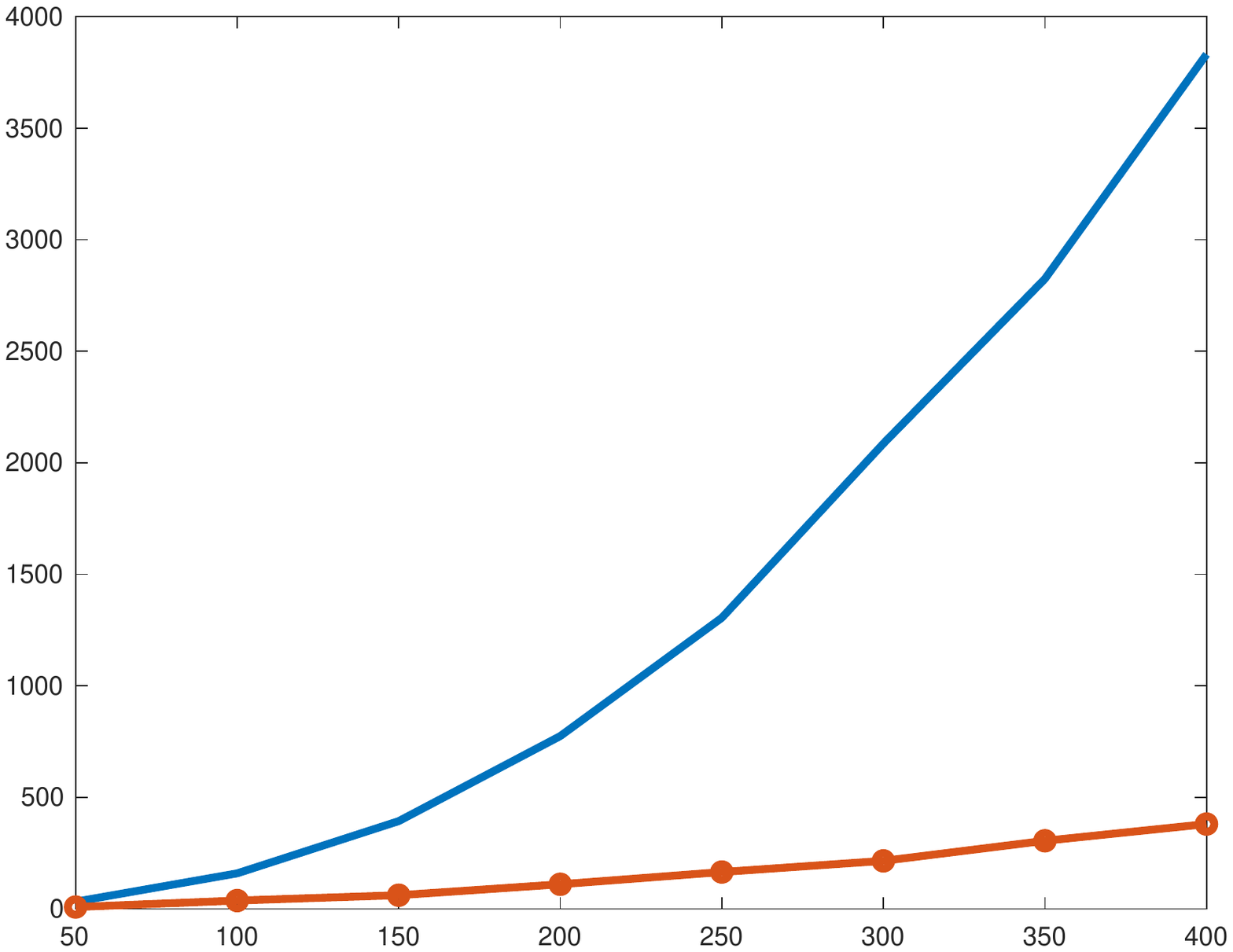}
    \includegraphics[scale=0.3]{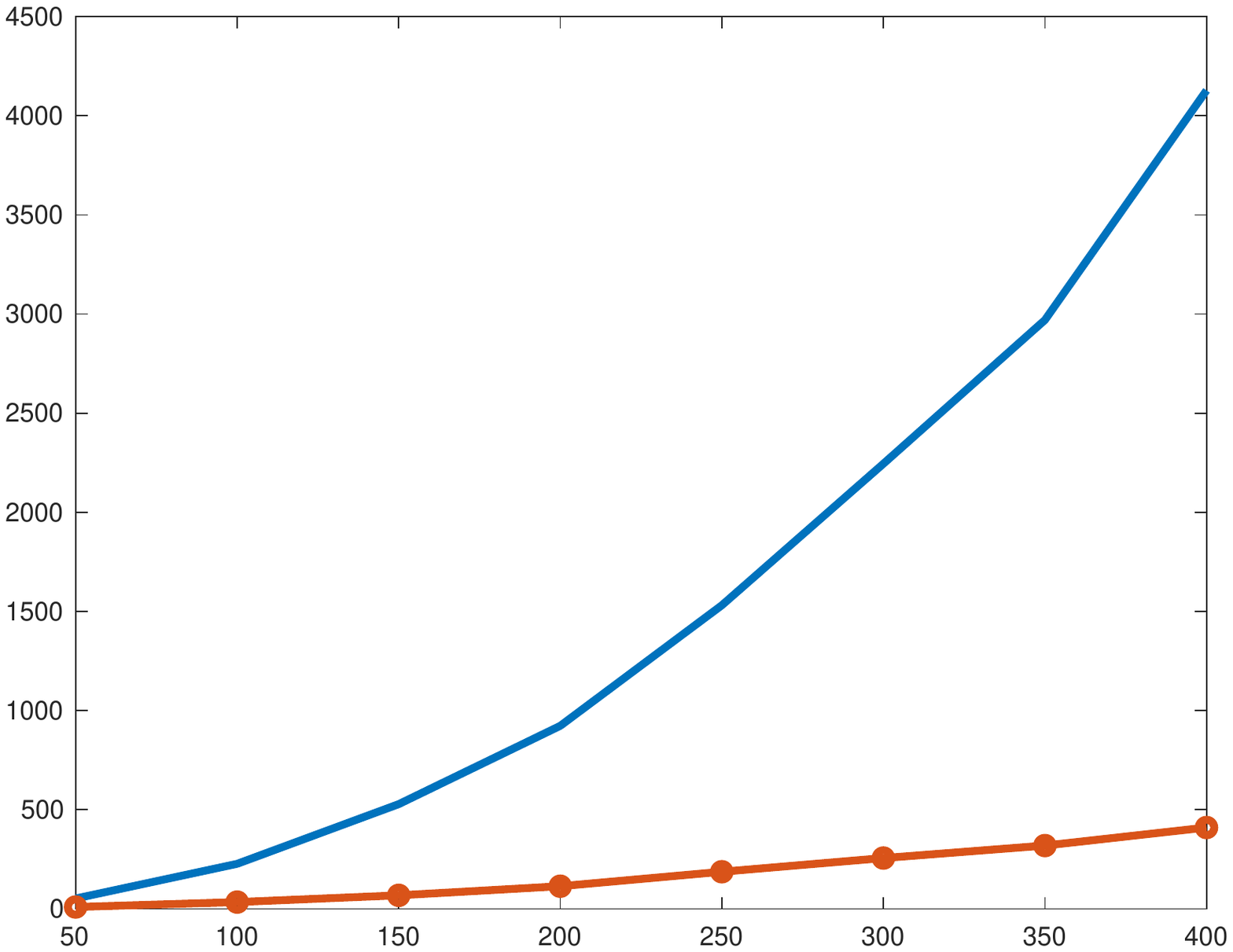} 
\centering
\label{figure 21}
\end{figure}
As we can observe, significant improvement is achieved using the new adaptive algorithm. We can take a closer look at the resource depletion rate:

\begin{figure}[!ht]

 		\caption{Resource Depletion Rates for Algorithm \ref{alg:5} with Random Input IV/V}
  \includegraphics[scale=0.3]{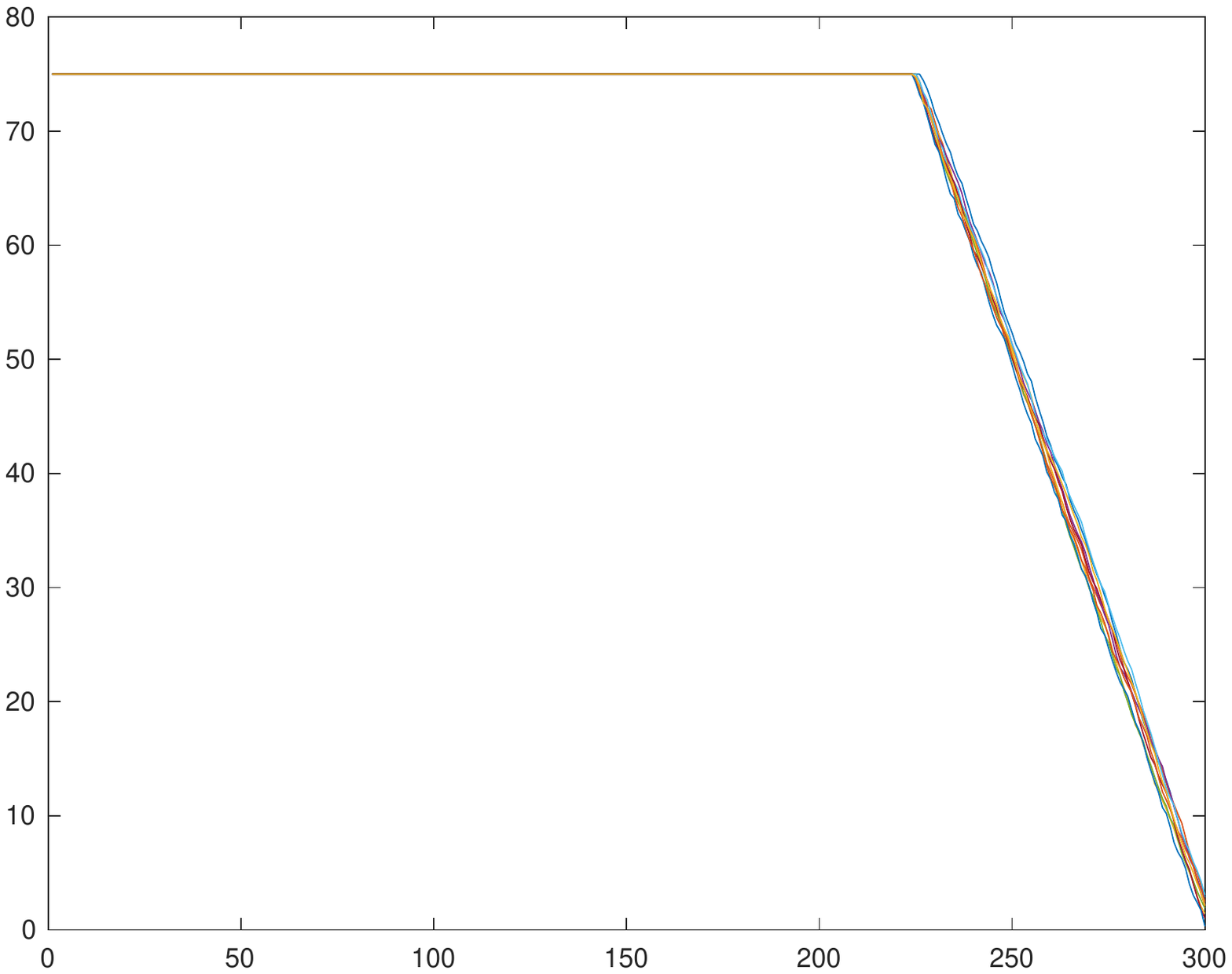}
    \includegraphics[scale=0.3]{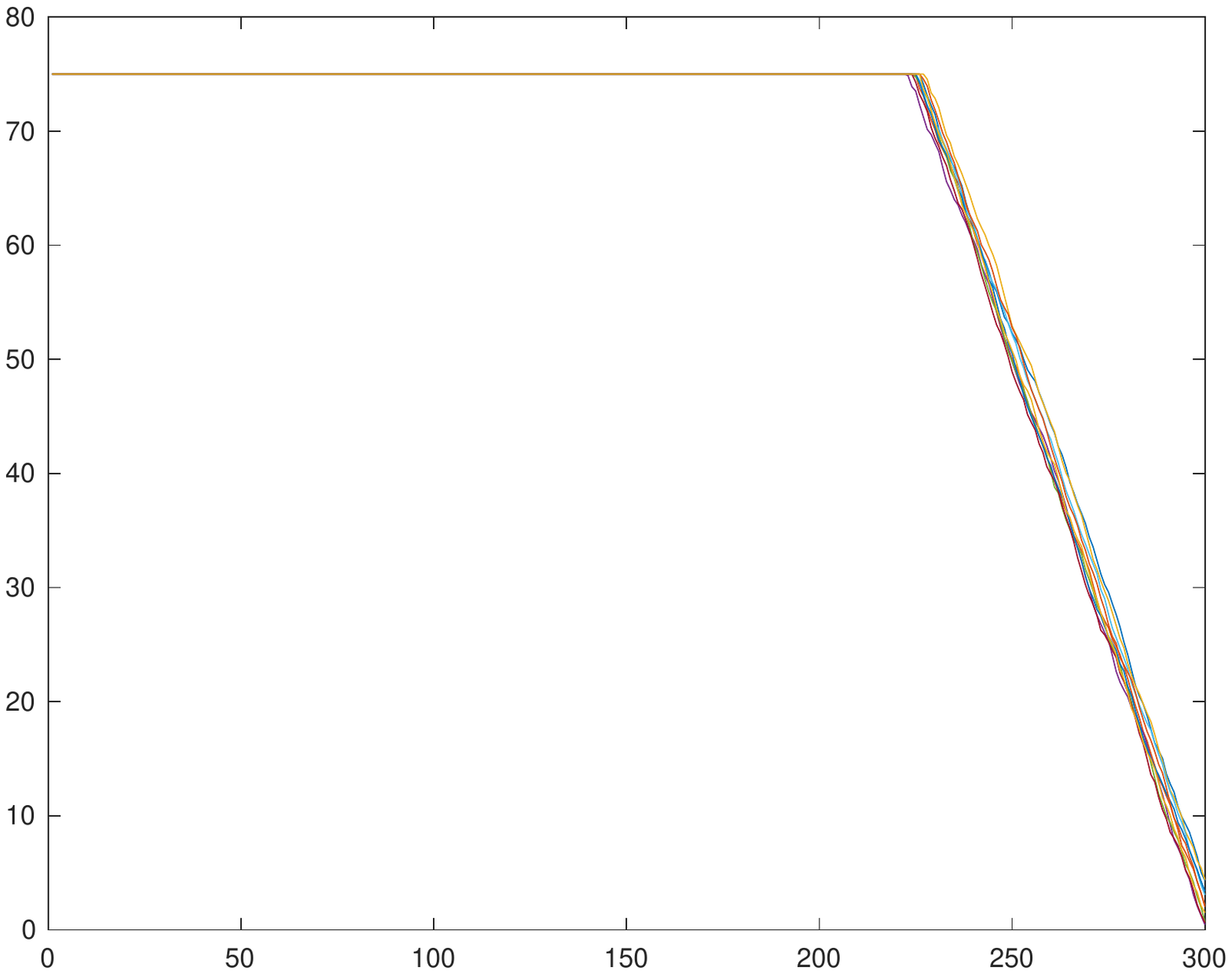} 
\centering
\label{figure 22}
\end{figure}
The depletion rates for Algorithm \ref{alg:5} are highly stable, accepting orders roughly near the end. The zoomed-in pictures are provided below.
\begin{figure}[!ht]

 		\caption{(Zoomed-in)Resource Depletion Rates for Algorithm \ref{alg:5} with Random Input IV/V}
  \includegraphics[scale=0.3]{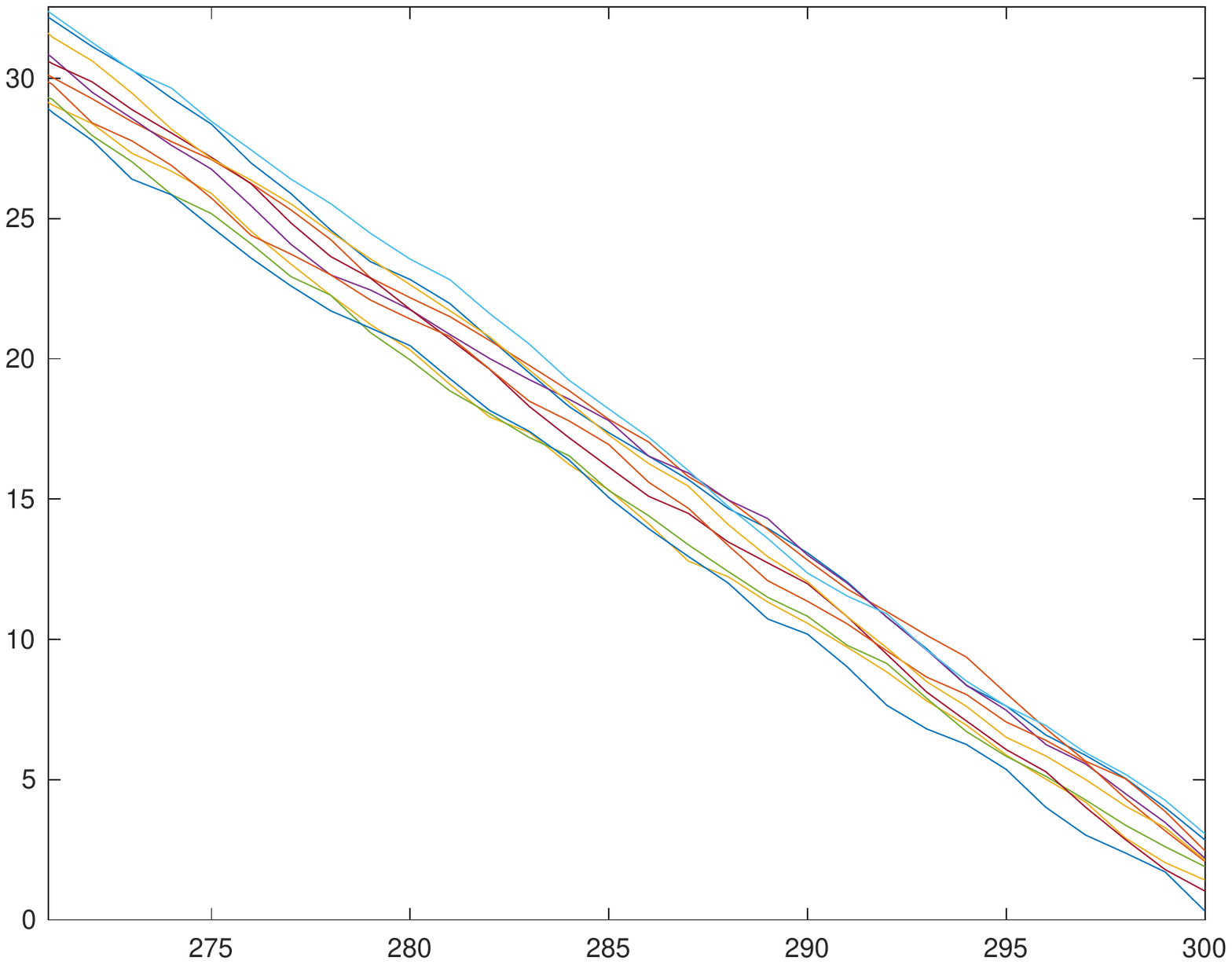}
    \includegraphics[scale=0.3]{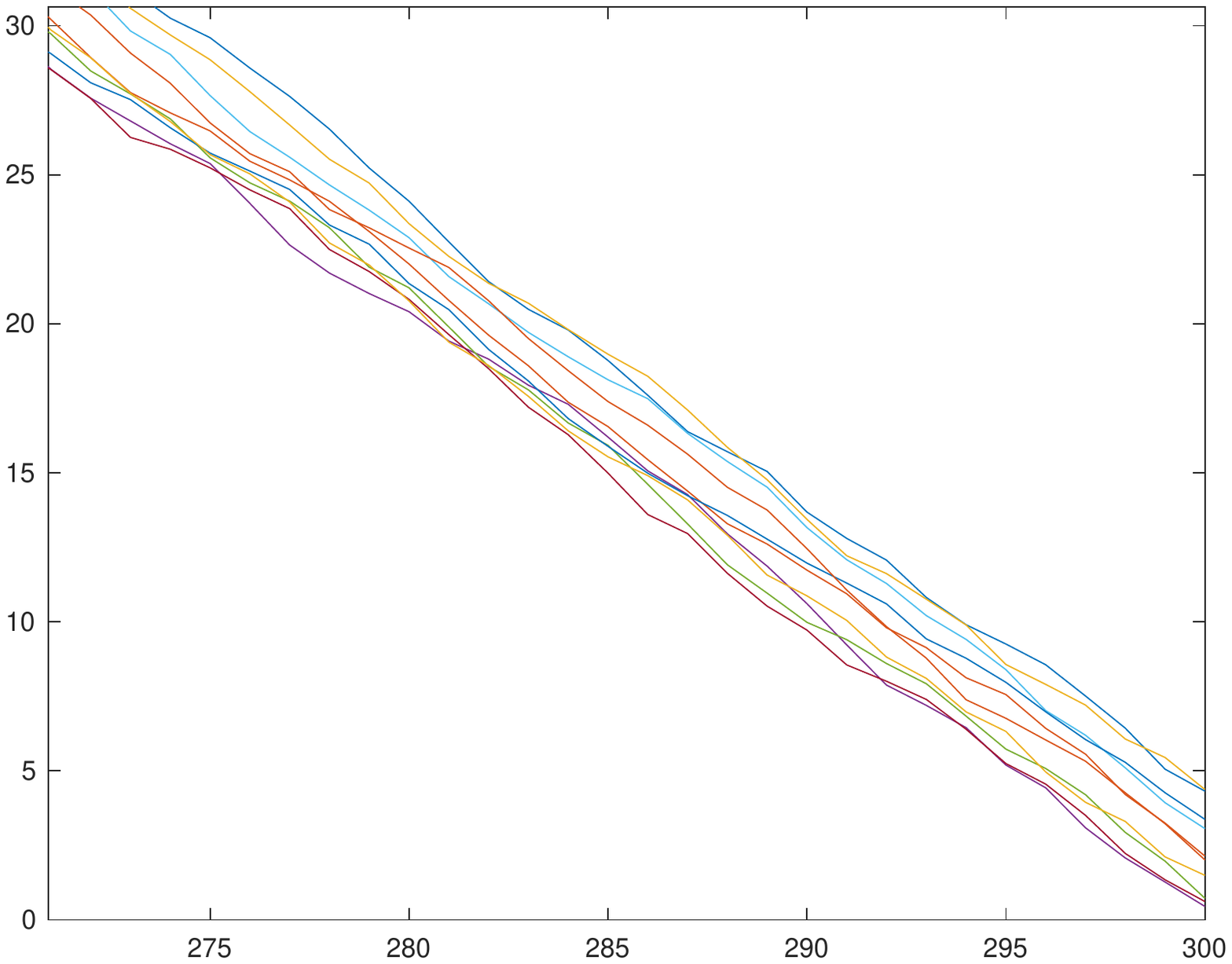} 
\centering
\label{figure 23}
\end{figure}
As we have discussed in the previous sections, the regret comes from three sources: the approximation of the dual optimal, the early depletion time, and the wasted resource. Since we have an increasing price, the cost of early depletion time is especially harmful, for the algorithm neglects the most profitable orders happening near the end. However, the exact regret coming from the early depletion time is unknown. We suspect that a slightly more conservative approach would be more helpful. For example, when computing for the dual $p_{t_k}$, we slightly increases it to be $p_{t_k}+\epsilon_{t_k}$ where $\epsilon_{t_k}$ vanishes quickly as $k\to L-1$. Surprised by those promising simulation results, we try to, in follow-up work, establish a formal statement on the regrets by verifying the following conjecture:

\begin{conjecture}
Suppose $a_t$ follows some i.i.d process and $r_i$ follows some linear regression model with white noise or weighted random walk model, then, under some suitable regularity conditions,
$$
\Delta_{n}\left(\pi_{5}\right) \leq O(n\log n)
$$
where $\pi_5$ is the online policy specified by Algorithm \ref{alg:5}. 
\end{conjecture}
We believe the conjecture at least holds for $O(n\sqrt{n})$, since most geometric algorithms have regret $O(\sqrt{n})$ and our additional price complexity shouldn't distort the regret by a factor of more than $O(n)$. Suppose this conjecture is proved to be true in either $O(\sqrt{n})$ or $O(n\log n)$, it would be a promising cornerstone in online linear programming, for it opens the possibility for non-stationary price data with quasi-linear regret, where the original algorithms exhibit $O(n^2)$ regrets. As a result, online linear programming algorithms can be used in a wide range of realistic settings unimaginable from the original i.i.d restrictions.

\clearpage
	\appendix
	\section{Appendices}
	\subsection{Proof of the Concentration for Regenerative Process}
	
	\begin{proof}
		First, we split the probability into three portions:
		$$
		\begin{aligned}
			&\mathbb{P}\left(\frac{1}{t}\left|\int_{0}^{t} f(X(s))ds-\alpha \right|>\epsilon\right) \leq \mathbb{P}\left(\frac{1}{t}\left|\int_{0}^{T_{0}}(f(X(s))-\alpha) d s\right|>\frac{\epsilon}{K}\right) \\
			&+\mathbb{P}\left(\frac{1}{t}\left|\sum_{k=1}^{N(t)} \int_{T_{k-1}}^{T_{k}}(f(X(s))-\alpha) d s\right|>\frac{\epsilon(K-2)}{K}\right) +\mathbb{P}\left(\frac{1}{t}\left|\int_{T_{N(t)}}^{t}(f(X(s))-\alpha) d s\right|>\frac{\epsilon}{K}\right) .
		\end{aligned}
		$$
		Since the integrals in the fist and last items are bounded almost surely, for sufficiently large $t$, the probability is zero. Since the first and the last term has integral bounded almost surely, as long as we know $MTK/\epsilon<t$  for any choice of large $K$, we can replace the middle portion by 
		$$
		\mathbb{P}\left(\frac{1}{t}\left|\sum_{k=1}^{N(t)} \int_{T_{k-1}}^{T_{k}}(f(X(s))-\alpha) d s\right|>\frac{K-2}{K}\epsilon\right). 
		$$
		Let us assume $KMT<t$, then the difficult part remains to bound the middle item. To simply the notation, we define the zero mean random variable
		$$
		\bar{Y}_{k}(f):=\int_{T_{k-1}}^{T_{k}}(f(X(s))-\alpha)ds. 
		$$ 
		Then, by the triangle inequality,
		$$
		\mathbb{P}\left(\frac{1}{t}\left|\sum_{k=1}^{N(t)} \int_{T_{k-1}}^{T_{k}}(f(X(s))-\alpha)ds \right|>\frac{\epsilon(K-2)}{K}\right)\\
		$$$$\leq 
		\mathbb{P}\left(\left|\frac{1}{t} \sum_{k=1}^{N(t)} \bar{Y}_{k}(f)-\frac{1}{t} \sum_{k \leq \lambda t} \bar{Y}_{k}(f)\right| \geq \frac{\epsilon(K-2)}{K}\right)+
		\mathbb{P}\left(\left|\frac{1}{t} \sum_{k \leq \lambda t} \bar{Y}_{k}(f)\right| \geq \frac{\epsilon(K-2)}{K}\right).
		$$
		Therefore, we represent the probability by an error bound and a standard i.i.d model. To bound the error, we have
		$$
		\begin{aligned}
			& \mathbb{P}\left(\left|\frac{1}{t} \sum_{k=1}^{N(t)} \bar{Y}_{k}(f)-\frac{1}{t} \sum_{k \leq \lambda t} \bar{Y}_{k}(f)\right| \geq \frac{\epsilon(K-2)}{K}\right) \\
			\leq & \mathbb{P}\left(\left|\frac{1}{t} \sum_{k=1}^{N(t)} \bar{Y}_{k}(f)-\frac{1}{t} \sum_{k \leq \lambda t} \bar{Y}_{k}(f)\right| \geq \frac{\epsilon(K-2)}{K} ;|N(t)-\lambda t| \leq \delta t\right) \\
			&+\mathbb{P}(|N(t)-\lambda t|>\delta t) .
		\end{aligned},
		$$
		where 
		$$
		\begin{aligned}
			& \mathbb{P}\left(\left|\frac{1}{t} \sum_{k=1}^{N(t)} \bar{Y}_{k}(f)-\frac{1}{t} \sum_{k \leq \lambda t} \bar{Y}_{k}(f)\right| \geq \frac{\epsilon(K-2)}{K} ;|N(t)-\lambda t| \leq \delta t\right) \\
			\leq & \mathbb{P}\left(\max _{\lambda t-\delta t \leq n \leq \lambda t+\delta t}\left|\frac{1}{t} \sum_{k \leq n} \bar{Y}_{k}(f)-\frac{1}{t} \sum_{k \leq \lambda t} \bar{Y}_{k}(f)\right| \geq \frac{\epsilon(K-2)}{K}\right) \\
			\leq & \mathbb{P}\left(\max _{1 \leq n \leq 2 \delta t}\left| \sum_{k \leq n} \bar{Y}_{k}(f)\right| \geq \frac{t\epsilon(K-2)}{K}\right) 
		\end{aligned}
		$$
		The last quantity $S_n=\sum_{k \leq n} \bar{Y}_{k}(f)$ is a martingale, hence $e^{S_n\theta}$, sub- martingale:
		$$\mathbb{P}\left(\max _{1 \leq n \leq 2 \delta t}\left| \sum_{k \leq n} \bar{Y}_{k}(f)\right| \geq \frac{t\epsilon(K-2)}{K}\right) =   \mathbb{P}\left(\max _{1 \leq n \leq 2 \delta t}\left| e^{S_n}\right| \geq e^{t\varepsilon(K-2)/K}\right)$$
		By Doob's maximum inequality for sub-Martingale, we know $$=   \mathbb{P}\left(\max _{1 \leq n \leq 2 \delta t}\left| e^{S_n}\right| \geq e^{t\varepsilon(K-2)/K}\right)\leq \frac{\mathbb{E}e^{S_{\floor{2\delta t}}}}{e^{\epsilon t}}\leq e^{(2\delta M-\epsilon(K-2)/K)t}$$
		Now, it remains is to bound $\mathbb{P}(|N(t)-\lambda t|>\delta t)$: 
		
		$$\mathbb{P}(|N(t)-\lambda t|>\delta t)\leq \mathbb{P}({N(t)}-\lambda t>\delta t)+\mathbb{P}({N(t)}-\lambda t <-\delta t)$$
		Therefore, by the definition of $N(t)$, 
		$$\mathbb{P}(|N(t)-\lambda t|>\delta t)\leq \mathbb{P}(\sum_{i=1}^{\floor{(\lambda+\delta)t}}\tau_i<t)+\mathbb{P}(\sum_{i=1}^{\floor{(\lambda-\delta)t}}\tau_i>t)$$
		which can be written as 
		$$= \mathbb{P}(\frac{1}{\floor{(\lambda+\delta)t}}\sum_{i=1}^{\floor{(\lambda+\delta)t}}\tau_i<\frac{t}{{\floor{(\lambda+\delta)t}}})+\mathbb{P}(\frac{1}{\floor{(\lambda-\delta)t}}\sum_{i=1}^{\floor{(\lambda-\delta)t}}\tau_i>\frac{t}{\floor{(\lambda-\delta)t}})$$
		which is bounded by 
		$$\leq \mathbb{P}(\frac{1}{\floor{(\lambda-\delta)t}}\sum_{i=1}^{\floor{(\lambda-\delta)t}}\tau_i<\frac{t}{{\floor{(\lambda+\delta)t}}})+\mathbb{P}(\frac{1}{\floor{(\lambda-\delta)t}}\sum_{i=1}^{\floor{(\lambda-\delta)t}}\tau_i>\frac{t}{\floor{(\lambda-\delta)t}})$$
		Since $\frac{1}{\lambda}-\frac{t}{{(\lambda+\delta)t}}\geq \frac{1}{\lambda}-\frac{t}{{\floor{(\lambda+\delta)t}}}>\frac{t}{{\floor{(\lambda-\delta)t}}}-\frac{1}{\lambda}$, we the two tail events are covered by the event that the sample average of $\floor{(\lambda-\delta)t}$ is $(\frac{\delta}{(\lambda-\delta)\lambda})$ away from the true mean $1/\lambda$. Therefore, 
		$$
		\mathbb{P}(|N(t)-\lambda t|>\delta t)\leq \mathbb{P}( |\frac{1}{\lfloor(\lambda-\delta) t\rfloor} \sum_{i=1}^{\lfloor(\lambda-\delta) t\rfloor} \tau_{i}-\frac{1}{\lambda}|>(\frac{\delta}{(\lambda-\delta)\lambda})).
		$$
		Now we apply the Hoeffding's inequality for bounded variable to obtain 
		$$
		\mathbb{P}(|N(t)-\lambda t|>\delta t)\leq 2 \exp \left(-\frac{ \delta^2t}{(\lambda-\delta)^2\lambda^2)T^2}\right).
		$$
		We are ready to combine everything together for the error bound:
		$$
		\epsilon(\delta):= 2 \exp \left(-\frac{ \delta^2t}{(\lambda-\delta)^2\lambda^2T^2}\right)+\exp\left({(2\delta M-\frac{\epsilon(K-2)}{K})t}\right)
		$$
		Hence, we know 
		$$
		\mathbb{P}\left(\frac{1}{t}\left|\sum_{k=1}^{N(t)} \int_{T_{k-1}}^{T_{k}}(f(X(s))-\alpha)ds \right|>\frac{\epsilon(K-2)}{K}\right)\\
		$$$$\leq 
		\epsilon(\delta)+
		\mathbb{P}\left(\left|\frac{1}{t} \sum_{k \leq \lambda t} \bar{Y}_{k}(f)\right| \geq \frac{\epsilon(K-2)}{K}\right),
		$$
		where the right hand side is a standard i.i.d sample average. We apply Hoeffding's inequality again to obtain
		$$
		\mathbb{P}\left(\left|\frac{1}{t} \sum_{k \leq \lambda t} \bar{Y}_{k}(f)\right| \geq \frac{\epsilon(K-2)}{K}\right)\leq 2\exp\left({-\frac{2\epsilon^2(K-2)^2}{K^2\lambda M^2T^2}t}\right)$$
		We know if $t>KTM/\epsilon$, then 
		$$
		\mathbb{P}\left(\frac{1}{t}\left| \int_{0}^{1}(f(X(s))-\alpha)ds \right|>\epsilon\right)\\
		$$$$\leq 2\exp\left({-\frac{2\epsilon^2(K-2)^2}{K^2\lambda M^2T^2}t}\right)+
		\epsilon(\delta)
		$$
		where 
		$$\epsilon(\delta)=2 \exp \left(-\frac{\delta^{2} t}{(\lambda-\delta)^{2} \lambda^{2} T^{2}}\right)+\exp \left(\left(2 \delta M-\frac{K-2}{K}\epsilon\right) t\right) $$\end{proof}

	\subsection{Proof of Proposition \ref{Q1}}:
	\begin{proof}
			For the first equality, it suffices to show 
			$$
			f(\boldsymbol{p})-f\left(\boldsymbol{p}^{*}\right)={\nabla f\left(\boldsymbol{p}^{*}\right)\left(\boldsymbol{p}-\boldsymbol{p}^{*}\right)}+{\lambda \mathbb{E}\sum_{i=0}^{\tau_1}\left[\int_{\boldsymbol{a}^{\top} \boldsymbol{p}}^{\boldsymbol{a}^{\top} \boldsymbol{p}^{*}}\left(I(r_i>v)-I\left(r_i>\boldsymbol{a}^{\top} \boldsymbol{p}^{*}\right)\right) d v\right]} .
			$$
			In particular, it suffices to show, 
			\begin{align*}
				& \sum_{i=1}^{m} d_{i} p_{i} -\lambda \sum_{j=0}^{\tau_1}\left(r_{i}-\sum_{i=1}^{m} a_{i} p_{i}\right)^{+}-\sum_{i=1}^{m} d_{i} p_{i}^*+\lambda \sum_{j=0}^{\tau_1}\left(r_{i}-\sum_{i=1}^{m} a_{i} p_{i}^*\right)^{+}\\
				& =\left(\left(d_{1}, \ldots, d_{m}\right)^{\top}-\lambda \sum_{j=0}^{\tau_1} \left(u_{1}, \ldots, u_{m}\right)^{\top} \cdot I\left(r_{j}>\sum_{i=1}^{m} a_{i} p_{i}^*\right)\right)^T(p-p^*)\\
				& +\int_{\boldsymbol{a}^{\top} \boldsymbol{p}}^{\boldsymbol{a}^{\top} \boldsymbol{p}^{*}}\left(I(r_i>v)-I\left(r_i>\boldsymbol{a}^{\top} \boldsymbol{p}^{*}\right)\right) d v, 
			\end{align*}
			which is true by Lemma \ref{A 1} in the appendix. The second equality is similar. 
		\end{proof}
	
	\subsection{Proof of Proposition \ref{Q2}}
		\begin{proof}
		The proof would be to translate the proof of Proposition 2 of \cite{5} in the language of the regenerative process. To show the upper bound:
		$$
		\begin{aligned}
			& \mathbb{E}\lambda \sum_{i=0}^{\tau_1}\left[\int_{\boldsymbol{a}^{\top} \boldsymbol{p}}^{\boldsymbol{a}^{\top} \boldsymbol{p}^{*}}\left(I(r_i>v)-I\left(r_i>\boldsymbol{a}^{\top} \boldsymbol{p}^{*}\right)\right) d v\right] \\
			=& \mathbb{E}\lambda \sum_{i=0}^{\tau_1} \left[\int_{\boldsymbol{a}^{\top} \boldsymbol{p}}^{\boldsymbol{a}^{\top} \boldsymbol{p}^{*}} \mathbb{P}(i>v \mid \boldsymbol{a})-\mathbb{P}\left(r_i>\boldsymbol{a}^{\top} \boldsymbol{p}^{*} \mid \boldsymbol{a}\right) d v\right]
		\end{aligned}
		$$
		By the Assumption 1* and 2*, 
		$$
		\begin{array}{l}
			\leq \mathbb{E}\lambda \sum_{i=0}^{\tau_1}\left[\int_{\boldsymbol{a}^{\top} \boldsymbol{p}}^{\boldsymbol{a}^{\top} \boldsymbol{p}^{*}} \mu\left(\boldsymbol{a}^{\top} \boldsymbol{p}^{*}-v\right) d v\right] \\
			=\frac{\mu}{2} \mathbb{E}\left[\left(\boldsymbol{a}^{\top} \boldsymbol{p}-\boldsymbol{a}^{\top} \boldsymbol{p}^{*}\right)^{2}\right] \\
			\leq \frac{\mu \bar{a}^{2}}{2}\left\|\boldsymbol{p}-\boldsymbol{p}^{*}\right\|_{2}^{2}
		\end{array}. 
		$$
		By symmetry we can equally show the lower bound. Therefore, 
		$$
		\frac{\lambda \lambda_{\min }}{2}\left\|\boldsymbol{p}-\boldsymbol{p}^{*}\right\|_{2}^{2} \leq f(\boldsymbol{p})-f\left(\boldsymbol{p}^{*}\right)-\nabla f\left(\boldsymbol{p}^{*}\right)\left(\boldsymbol{p}-\boldsymbol{p}^{*}\right) \leq \frac{\mu \bar{a}^{2}}{2}\left\|\boldsymbol{p}-\boldsymbol{p}^{*}\right\|_{2}^{2}.
		$$
		Moreover, by the optimality of $\boldsymbol{p}^*$, we can show $\nabla f(\boldsymbol{p}^*)\geq 0$, otherwise if we perturb any coordinate $\boldsymbol{p}_i'=\boldsymbol{p}_i^*-\frac{\nabla_i f(\boldsymbol{p}^*)}{\mu \bar{a}^2}$ while $\boldsymbol{p}_j'=\boldsymbol{p}_j^*$ for all $j\neq j$ we get from the upper bound that
		$$
		f(\boldsymbol{p}')-f\left(\boldsymbol{p}^{*}\right)-\boldsymbol{p}_i^*-\frac{(\nabla f(\boldsymbol{p}^*))^2}{\mu \bar{a}^2}\leq \frac{(\nabla f(\boldsymbol{p}^*))^2}{2\mu \bar{a}^2}
		$$
		which contradicts the optiamlity of $\boldsymbol{p}^{*}$. Similarly, we can check that $\boldsymbol{p}^{*}\dot \nabla f(\boldsymbol{p}^{*})=0$. So suppoes we have another optimal solution $\boldsymbol{p}$, then the lower bound 
		$$
		\frac{\lambda \lambda_{\min }}{2}\left\|\boldsymbol{p}-\boldsymbol{p}^{*}\right\|_{2}^{2} \leq f(\boldsymbol{p})-f\left(\boldsymbol{p}^{*}\right)-\nabla f\left(\boldsymbol{p}^{*}\right)\left(\boldsymbol{p}-\boldsymbol{p}^{*}\right)
		$$ 
		implies $\left\|\boldsymbol{p}-\boldsymbol{p}^{*}\right\|_{2}^{2}=0$ which estbalishes the uniquness. The i.i.d counterpart is illustrated as in Lemma \ref{A 2}. 
	\end{proof}
	 \subsection{Proof of Proposition \ref{Q3}}
	 	\begin{proof}
	 		Since we know $ \phi\left(p^{*}, u_{i}\right) $ is a regenerative process, we apply Regenerative Concentration \ref{First Main Result} to get that for each entry i,

	 	$$
	 \mathbb{P}\left(\left\|\frac{1}{n} \sum_{i=1}^{n} \phi\left(p^{*}, u_{i}\right)-\nabla f\left(\boldsymbol{p}^{*}\right)\right\|_{i} \geq \epsilon\right)
	  \leq 2 \exp \left(-\frac{ \epsilon^{2}(K-2)^{2}}{K^{2} 2\lambda \bar{a}^2 T^{2}} t\right)+\epsilon(\delta, K) $$
	 	where $K=n\epsilon/T(d_i+\bar{a})$, and 
	 	$$
	 	\epsilon(\delta, K)=2 \exp \left(-\frac{\delta^{2} t}{(\lambda-\delta)^{2} \lambda^{2} T^{2}}\right)+\exp \left(\left(2 \delta (d_i+\bar{a})-\frac{K-2}{K} \epsilon\right) t\right).
	 	$$
	 	Meanwhile, 
	 	$$
	 	\left\{\left\|\frac{1}{n} \sum_{i=1}^{n} \phi\left(p^{*}, u_{i}\right)-\nabla f\left(\boldsymbol{p}^{*}\right)\right\|_{2}\geq\epsilon \right\}
	 \subset \bigcup_{i=1}^{m}	\left\{\left\|\frac{1}{n} \sum_{i=1}^{n} \phi\left(p^{*}, u_{i}\right)-\nabla f\left(\boldsymbol{p}^{*}\right)\right\|_{i}\geq \frac{\epsilon}{\sqrt{m}}\right\}
	 	$$
	 	So apply Regenerative Heoffding  \ref{First Main Result} on each entry and take the union bound to get 
	 	$$
	 \mathbb{P}\left(\left\|\frac{1}{n} \sum_{i=1}^{n} \phi\left(p^{*},u_i\right)-\nabla f\left(\boldsymbol{p}^{*}\right)\right\|_{2} \geq \epsilon\right)\leq 2m \exp \left(-\frac{ \epsilon^{2}(K-2)^{2}}{K^{2} 2\lambda \bar{a}^2 T^{2}m} t\right)+m\tilde{\epsilon}(\delta, K) $$
	 where $K=n\epsilon/T(d_i+\bar{a})$ and 
	 $$
	 \epsilon(\delta, K)=2 \exp \left(-\frac{\delta^{2} t}{(\lambda-\delta)^{2} \lambda^{2} T^{2}}\right)+\exp \left(\left(2 \delta (d_i+\bar{a})-\frac{K-2}{K\sqrt{m}} \epsilon \right) t\right),
	 $$  where $u_i\sim \mathcal{P}_i$
	 \end{proof}

	\subsection{Proof of Proposition \ref{Q4}}
	
	The proof of this proposition has three steps, inspired by the strategy of \cite{5} Proposition 4. The first step is to show $
	\boldsymbol{M}_{n}=\frac{1}{n} \sum_{j=1}^{n} \boldsymbol{a}_{j} \boldsymbol{a}_{j}^{\top}
	$ concentrates around its $\boldsymbol{M}=\mathbb{E}\left[\boldsymbol{a}_{j} \boldsymbol{a}_{j}^{\top}\right]$. Therefore, the non-degeneracy condition is imposed on each sample average with high probability. Second, we want to partition the probability space $\bar{\Omega_{p}}$ into some finitely many sets $\Omega_{kl}$, on which we pick a representative $\boldsymbol{p}_{kl}$ to show $$
	\sum_{j=1}^{n} \int_{\boldsymbol{a}_{i}^{\top} \boldsymbol{p}_{k l}}^{\boldsymbol{a}_{j}^{\top} \boldsymbol{p}^{*}}\left(I\left(r_{j}>v\right)-I\left(r_{j}>\boldsymbol{a}_{j}^{\top} \boldsymbol{p}^{*}\right)\right) d v
	$$ concentrates on its mean. Then, we show that for any $\boldsymbol{p}$, it is close to some $\boldsymbol{p}_{kl}$ such that the difference:
	$$
	\begin{aligned}
		& {\sum_{j=1}^{n} \int_{\boldsymbol{a}_{j}^{\top} \boldsymbol{p}_{k l}}^{\boldsymbol{a}_{j}^{\top} \boldsymbol{p}^{*}}\left(I\left(r_{j}>v\right)-I\left(r_{j}>\boldsymbol{a}_{j}^{\top} \boldsymbol{p}^{*}\right)\right) d v}\\
		&-{\sum_{j=1}^{n} \int_{\boldsymbol{a}_{j}^{\top} \boldsymbol{p}}^{\boldsymbol{a}_{j}^{\top} \boldsymbol{p}^{*}}\left(I\left(r_{j}>v\right)-I\left(r_{j}>\boldsymbol{a}_{j}^{\top} \boldsymbol{p}^{*}\right)\right) d v} \\
		=& \sum_{j=1}^{n} \int_{\boldsymbol{a}_{j}^{\top} \boldsymbol{p}_{k l}}^{\boldsymbol{a}_{j}^{\top} \boldsymbol{p}}\left(I\left(r_{j}>v\right)-I\left(r_{j}>\boldsymbol{a}_{j}^{\top} \boldsymbol{p}^{*}\right)\right) d v.
	\end{aligned}
	$$
	is small. Therefore, we will derive some concentration bound for any $\boldsymbol{p} \in \Omega_{p}$. Finally, in step three, we combine those observations to prove the proposition. \\
	
	The first step is shown by \ref{A 5}.\\
	
	To show the second step, we need to consider a partition $\ref{P 1}$. Then, let $\boldsymbol{p}_{k l}$ be the center of the cube $\Omega_{k l}$. If we denote the event that the difference for each representative $\boldsymbol{p}_{kl}$ deviates from the mean:
	$$
	\begin{aligned}
		\mathcal{E}_{k l, 1}=&\left\{\frac{1}{n} \sum_{j=1}^{n} \int_{\boldsymbol{a}_{j}^{\top} \boldsymbol{p}_{k l}}^{\boldsymbol{a}_{j}^{\top} \boldsymbol{p}^{*}}\left(I\left(r_{j}>v\right)-I\left(r_{j}>\boldsymbol{a}_{j}^{\top} \boldsymbol{p}^{*}\right)\right) d v \leq\right.\\
		&\left.\frac{1}{n} \sum_{j=1}^{n} \mathbb{E}\left[\int_{\boldsymbol{a}_{j}^{\top} \boldsymbol{p}_{k l}}^{\boldsymbol{a}_{j}^{\top} \boldsymbol{p}^{*}}\left(I\left(r_{j}>v\right)-I\left(r_{j}>\boldsymbol{a}_{j}^{\top} \boldsymbol{p}^{*}\right)\right) d v \mid \boldsymbol{a}_{1}, \ldots, \boldsymbol{a}_{n}\right]-\epsilon \bar{a}\left\|\boldsymbol{p}^{*}-\overline{\boldsymbol{p}}_{k l}\right\|_{2}\right\},
	\end{aligned}
	$$
	then 
	\begin{proposition}
		For $t>T M K / \epsilon$
		$$
		\mathbb{P}\left(\mathcal{E}_{k l, 1} \mid \boldsymbol{a}_{1}, \ldots, \boldsymbol{a}_{n}\right) \leq \exp \left(-\frac{n \epsilon^{2}}{2 \lambda T^{2}}\right)+\epsilon\left(\delta, K\right) $$ where $$
		\epsilon\left(\delta, K\right)=2 \exp \left(-\frac{\delta^{2} t}{(\lambda-\delta)^{2} \lambda^{2} T^{2}}\right)+\exp \left(\left(2 \delta-\frac{K-2}{K} \epsilon\right) t\right)
		$$ 
	\end{proposition}
	\begin{proof}
		There are two things to check before applying the concentration bound on the regenerative process. First, we need to show it is indeed a regenerative process within the summation. Second, we want to show each term is bounded almost surely. Indeed, Since $r_j$ is regenerative and $a_j$ i.i.d, we have 
		$$
		\left|\int_{\boldsymbol{a}_{j}^{\top} \boldsymbol{p}_{k l}}^{\boldsymbol{a}_{j}^{\top} \boldsymbol{p}^{*}}\left(I\left(r_{j}>v\right)-I\left(r_{j}>\boldsymbol{a}_{j}^{\top} \boldsymbol{p}^{*}\right)\right) d v\right|
		$$
		is a non-delayed regenerative process with the same period $\tau_1$.Now, by Assumption 1*, we obtain 
		$$
		\left|\int_{\boldsymbol{a}_{j}^{\top} \boldsymbol{p}_{k l}}^{\boldsymbol{a}_{j}^{\top} \boldsymbol{p}^{*}}\left(I\left(r_{j}>v\right)-I\left(r_{j}>\boldsymbol{a}_{j}^{\top} \boldsymbol{p}^{*}\right)\right) d v\right| \leq\left|\boldsymbol{a}_{j}^{\top} \boldsymbol{p}_{k l}-\boldsymbol{a}_{j}^{\top} \boldsymbol{p}^{*}\right| \leq \bar{a}\left\|\boldsymbol{p}^{*}-\overline{\boldsymbol{p}}_{k l}\right\|_{2}. 
		$$
		So by \ref{First Main Result}, for $t>TMK/\epsilon$
		$$
		\mathbb{P}\left(\mathcal{E}_{k l, 1} \mid \boldsymbol{a}_{1}, \ldots, \boldsymbol{a}_{n}\right) \leq \exp \left(-\frac{n  \epsilon^{2}}{2\lambda T^2 }\right)+\epsilon(\delta,K)
		$$
		
		$$
		\epsilon(\delta, K)=2 \exp \left(-\frac{\delta^{2} t}{(\lambda-\delta)^{2} \lambda^{2} T^{2}}\right)+\exp \left(\left(2 \delta 
		-\frac{K-2}{K} \epsilon\right) t\right)
		$$
		Since the conditional probability is independent of $ \boldsymbol{a}_{1}, \ldots, \boldsymbol{a}_{n}$, the same result holds for the unconditional version. Note that the error term does not depend on the choice of the cube in such set-up.
	\end{proof}
	Now, we provide a similar concentration analysis on the maximum probability distance between the points in a cube and its representative. If we define $$
	\Gamma_{k l}\left(r_{j}, \boldsymbol{a}_{j}\right)=\max _{\boldsymbol{p} \in \Omega_{k l}} \int_{\boldsymbol{a}_{j}^{\top} \boldsymbol{p}}^{\boldsymbol{a}_{j}^{\top} \boldsymbol{p}_{k l}} I\left(r_{j}>v\right)-I\left(r_{j}>\boldsymbol{a}_{j}^{\top} \boldsymbol{p}^{*}\right) d v,
	$$
	then if we show the following event is small, we accomplish the goal of bounding the distance between the points and their representative:
	$$
	\mathcal{E}_{k l, 2}=\left\{\left|\frac{1}{n} \sum_{j=1}^{n} \Gamma_{k l}\left(r_{j}, \boldsymbol{a}_{j}\right)-\frac{1}{n} \sum_{j=1}^{n} \mathbb{E}\left[\Gamma_{k l}\left(r_{j}, \boldsymbol{a}_{j}\right) \mid \boldsymbol{a}_{1}, \ldots, \boldsymbol{a}_{n}\right]\right| \geq 2 \epsilon \bar{a} \max _{\boldsymbol{p} \in \Omega_{k l}}\left\|\boldsymbol{p}-\boldsymbol{p}_{k l}\right\|_{2}\right\}.
	$$
	
	\begin{proposition}
		Using the notation from above, we have
		
		$$
		\mathbb{P}\left(\mathcal{E}_{k l,2}\right) \leq \exp \left(-\frac{n \epsilon^{2}}{2 \lambda T^{2}}\right)+\tilde{\epsilon}\left(\delta, K\right)
		$$
		where
		$$
		\tilde{\epsilon}\left(\delta, K\right)=2 \exp \left(-\frac{\delta^{2} t}{(\lambda-\delta)^{2} \lambda^{2} T^{2}}\right)+\exp \left(\left(2 \delta -\frac{K-2}{K} \epsilon\right) t\right)
		$$\label{4.10}
	\end{proposition}
	First, by \ref{A 4}, we know that 
	$$
	\left|\Gamma_{k l}\left(r_{j}, \boldsymbol{a}_{j}\right)\right| \leq \max _{\boldsymbol{p} \in \Omega_{k l}}\left|\boldsymbol{a}_{j}^{\top} \boldsymbol{p}-\boldsymbol{a}_{j}^{\top} \boldsymbol{p}_{k l}\right| \leq \bar{a} \max _{\boldsymbol{p} \in \Omega_{k l}}\left\|\boldsymbol{p}-\boldsymbol{p}_{k l}\right\|_{2}.
	$$
	Hence, we can derive the following proof using this upper bound:
	\begin{proof}
		Since the function $\Gamma$ is a bounded regenerative process with the same random periods as $r_j$, we can apply \ref{First Main Result} to get 
		$$
		\mathbb{P} \left(\mathcal{E}_{k l, 2}\mid \boldsymbol{a}_{1}, \ldots, \boldsymbol{a}_{n}\right) \leq
		\exp \left(-\frac{n \epsilon^2}{2 \lambda T^2}\right)+\tilde{\epsilon} (\delta, K) 
		$$
		where
		$$\tilde{\epsilon}(\delta, K)=2 \exp \left(-\frac{\delta^{2} t}{(\lambda-\delta)^{2} \lambda^{2} T^{2}}\right)+\exp \left(\left(2 \delta -\frac{K-2}{K} \epsilon\right) t\right)
		$$
		where the upper bounds are derived from the inequalities \ref{A 8} Since the result is independent of $\boldsymbol{a}_{1}, \ldots, \boldsymbol{a}_{n}$, the same holds for the unconditional version. Again, the error is actually independent of the partition. 
	\end{proof}
	The reason we want to compute $\mathcal{E}_{kl,1}$ and $\mathcal{E}_{kl,2}$ are the following: the probability in Proposition \ref{4.8}, defined as $K$, is equivalent to 
	$$
	K=1-\mathbb{P}\left(\bigcap_{k=1}^{N} \bigcap_{l=1}^{l_{k}}\left(\mathcal{E}_{k l, 1}^{c} \bigcap \mathcal{E}_{k l, 2}^{c}\right) \bigcap \mathcal{E}_{0}\right)
	$$
	by \ref{A 9}. So we are ready to prove Proposition 4.8:
	\begin{proof} 
		It suffices to bound the following event
		$$
		\begin{aligned}
			K:=1-\mathbb{P}\left(\bigcap_{k=1}^{N} \bigcap_{l=1}^{l_{k}}\left(\mathcal{E}_{k l, 1}^{c} \bigcap \mathcal{E}_{k l, 2}^{c}\right) \bigcap \mathcal{E}_{0}\right) &=\mathbb{P}\left(\bigcup_{k=1}^{N} \bigcup_{l=1}^{l_{k}}\left(\mathcal{E}_{k l, 1} \bigcup \mathcal{E}_{k l, 2}\right) \bigcup \mathcal{E}_{0}^{c}\right) \\
			& \leq \mathbb{P}\left(\mathcal{E}_{0}^{c}\right)+\sum_{k=1}^{N} \sum_{l=1}^{l_{k}}\left(\mathbb{P}\left(\mathcal{E}_{k l, 1}\right)+\mathbb{P}\left(\mathcal{E}_{k l, 2}\right)\right).
		\end{aligned}
		$$
		Now by Proposition 4.9 and Proposition 4.10, we obtain 
		$$K\leq m \exp \left(-\frac{n \lambda_{\min }^{2}}{4 \bar{a}^{2}}\right)+2 \exp \left(-\frac{n \epsilon^{2}}{2\lambda T^2}\right) \cdot(2 N)^{m}+\sum_{k=1}^{N} \sum_{l=1}^{l_{k}}(\epsilon\left(\delta, K\right)+\tilde{\epsilon}\left(\delta, K\right)) $$
		$$\leq m \exp \left(-\frac{n \lambda_{\min }^{2}}{4 \bar{a}^{2}}\right)+2 \exp \left(-\frac{n \epsilon^{2}}{2\lambda T^2} +\right) \cdot(2 N)^{m}$$$$+(2N)^m\cdot \left (2 \exp \left(-\frac{\delta^{2} t}{(\lambda-\delta)^{2} \lambda^{2} T^{2}}\right)+\exp \left(\left(2 \delta -\frac{K-2}{K} \epsilon\right) t\right)\right) .$$
		This would completes the proof.
	\end{proof}

	\subsection{Proof of Regenerative Dual Convergence \ref{Q5}}
	\begin{proof}
		First, we consider the first order approximation:
		$$
		\mathcal{E}_{1}=\left\{\left\|\frac{1}{n} \sum_{j=1}^{n} \phi\left(\boldsymbol{p}^{*}, \boldsymbol{u}_{j},\boldsymbol{v}_{j},\tau\right)-\nabla f\left(\boldsymbol{p}^{*}\right)\right\|_{2} \leq \epsilon\right\}
		$$
		From proposition \ref{4.7}, we know 
		$$
		\mathbb{P}\left(\mathcal{E}_{1}^{c}\right) \leq  2 m \exp \left(-\frac{\epsilon^{2}(K-2)^{2}}{K^{2} 2 \lambda \bar{a}^{2} T^{2} m} t\right)+m \tilde{\epsilon}(\delta, K)
		$$
		where $K=n \epsilon / T\left(d_{i}+\bar{a}\right)$ and 
		$$
		\epsilon(\delta, K)=2 \exp \left(-\frac{\delta^{2} t}{(\lambda-\delta)^{2} \lambda^{2} T^{2}}\right)+\exp \left(\left(2 \delta\left(d_{i}+\bar{a}\right)-\frac{K-2}{K \sqrt{m}} \epsilon\right) t\right)
		$$
		Let us also consider the second order approximation: 
		$$
		\mathcal{E}_{2}=\left\{\frac{1}{n} \sum_{j=1}^{n} \int_{\boldsymbol{a}_{j}^{\top} \boldsymbol{p}}^{\boldsymbol{a}_{j}^{\top} \boldsymbol{p}^{*}}\left(I\left(r_{j}>v\right)-I\left(r_{j}>\boldsymbol{a}_{j}^{\top} \boldsymbol{p}^{*}\right)\right) d v \geq-\epsilon^{2}-2 \epsilon \bar{a}\left\|\boldsymbol{p}^{*}-\boldsymbol{p}\right\|_{2}+\frac{\lambda \lambda_{\min }}{32}\left\|\boldsymbol{p}^{*}-\boldsymbol{p}\right\|_{2}^{2}\right\}
		$$
		From proposition \ref{4.8}, we have 
		$$
		\mathbb{P}\left(\mathcal{E}_{2}^{c}\right)\leq m \exp \left(-\frac{n \lambda_{\min }^{2}}{4 \bar{a}^{2}}\right)+2 \exp \left(-\frac{n \epsilon^{2}}{2 \lambda T^{2}}+\right) \cdot(2 N)^{m}
		$$
		$$+(2 N)^{m} \cdot\left(2 \exp \left(-\frac{\delta^{2} t}{(\lambda-\delta)^{2} \lambda^{2} T^{2}}\right)+\exp \left(\left(2 \delta -\frac{K-2}{K} \epsilon\right) t\right)\right.$$
		We know that on the event of $\mathcal{E}_{1} \cap \mathcal{E}_{2}$, we have 
		$$
		\mathbb{P}\left(\frac{\left\|\boldsymbol{p}_{n}^{*}-\boldsymbol{p}^{*}\right\|_{2}^{2}}{\kappa^{2}}\leq \epsilon^{2}\right)
		$$
		where $$
		\kappa=\frac{2 \bar{a}+1+\sqrt{(2 \bar{a}+1)^{2}+\frac{\lambda \lambda_{\min }}{8}}}{\lambda \lambda_{\min } / 16}
		$$
		Therefore, if we let $\epsilon'=\epsilon^2$
		$$
		\frac{1}{\kappa^{2}} \mathbb{E}\left\|\boldsymbol{p}_{n}^{*}-\boldsymbol{p}^{*}\right\|_{2}^{2}=\int_{0}^{\frac{\bar{r}^{2}}{d}^{2}} \mathbb{P}\left(\frac{\left\|\boldsymbol{p}_{n}^{*}-\boldsymbol{p}^{*}\right\|_{2}^{2}}{\kappa^{2}}>\epsilon^{\prime}\right) \mathrm{d} \epsilon^{\prime}
		$$
		$$\leq 
		\int_{0}^{\frac{\bar{r}^{2}}{d^{2}}} m \exp \left(-\frac{n \lambda_{\min }^{2}}{4 \bar{a}^{2}}\right)+2 \exp \left(-\frac{n \epsilon'}{2 \lambda T^{2}}+\right) \cdot(2 N)^{m}+2 m \exp \left(-\frac{\epsilon'(K-2)^{2}}{K^{2} 2 \lambda \bar{a}^{2} T^{2} m} \right)
		$$

		$$+m\left(2 \exp \left(-\frac{\delta^{2} t}{(\lambda-\delta)^{2} \lambda^{2} T^{2}}\right)+\exp \left(\left(2 \delta\left(d_{i}+\bar{a}\right)-\frac{K-2}{K \sqrt{m}} \sqrt{\epsilon'}\right) n\right)\right)$$
		
		$$+(2 N)^{m} \cdot2 \exp \left(-\frac{\delta'^{2} t}{(\lambda-\delta')^{2} \lambda^{2} T^{2}}\right)+\exp \left(\left(2 \delta' -\frac{K-2}{K} \sqrt{\epsilon'}\right) n\right) \wedge 1 \mathrm{~d} {\epsilon^{\prime}}
		$$
		By \ref{A 10}\ref{A 11}\ref{B 1}\ref{B 2}\ref{B 3}\ref{B 4} we have, by choosing $\delta=\frac{(K-2)\sqrt{\epsilon'}}{4(d_i+\bar{a})K(\sqrt{M})}$ and $\delta'=\frac{(K-2){\epsilon'}}{4(\bar{a}K)}$, we can find some constant $\left(c+c^{\prime}+c^{\prime \prime}+c''''+c'''''+c''''''\right) $ which only depends on $\bar{a},\lambda,T,d_i$, i.e the information about the upper bound of the data, 
		$$
		\mathbb{E}\left\|\boldsymbol{p}_{n}^{*}-\boldsymbol{p}^{*}\right\|_{2}^{2} \leq \kappa^{2}\left(c+c^{\prime}+c^{\prime \prime}+c''''+c'''''+c''''''\right) \frac{m \log m \log \log n}{n}.
		$$
		Therefore, we have, under the assumption 1*,2* and 3*, for some large constant $C$:
		$$
		\mathbb{E}\left\|\boldsymbol{p}_{n}^{*}-\boldsymbol{p}^{*}\right\|_{2} \leq C \sqrt{\frac{m \log m \log \log n}{n}}.
		$$
		Hence, we have proved the second main result of this paper. 
	\end{proof}
\subsection{Proof of Proposition \ref{R1}}
	\begin{proof}
	For the first inequality, we have, using $N(n)$ to denote the number of complete periods up to time $n$:
	$$
	\begin{aligned}
		\mathbb{E} R_{n}^{*} &=\mathbb{E}\left[\sum_{j=1}^{n} r_{j} x_{j}^{*}\right] \\
		&=\mathbb{E}\left[n \boldsymbol{d}^{\top} \boldsymbol{p}_{n}^{*}+\sum_{j=1}^{n}\left(r_{j}-\boldsymbol{a}_{j}^{\top} \boldsymbol{p}_{n}^{*}\right)^{+}\right] \quad \text { (From the strong duality) } \\
		&\leq \mathbb{E}\left[n \boldsymbol{d}^{\top} \boldsymbol{p}^{*}+\sum_{j=1}^{n}\left(r_{j}-\boldsymbol{a}_{j}^{\top} \boldsymbol{p}^{*}\right)^{+}\right] \quad \text { (From the optimality of } \boldsymbol{p}_{n}^{*}) \\
		&\leq \sum_{i=1}^{n} g_i\left(\boldsymbol{p}^{*}\right).
	\end{aligned}
	$$
	For the second inequality, it suffices to check
	$$
	\begin{aligned}
		&\mathbb{E}\left[r_i I\left(r_i>\boldsymbol{a}_i^{\top} \boldsymbol{p}^{*}\right)+\left(\boldsymbol{d}-\boldsymbol{a}_i I\left(r_i>\boldsymbol{a}_i^{\top} \boldsymbol{p}^{*}\right)\right)^{\top} \boldsymbol{p}^{*}\right]\\ &-\mathbb{E}\left[r_i I\left(r_i>\boldsymbol{a}_i^{\top} \boldsymbol{p}\right)+\left(\boldsymbol{d}-\boldsymbol{a}_i I\left(r_i>\boldsymbol{a}_i^{\top} \boldsymbol{p}\right)\right)^{\top} \boldsymbol{p}^{*}\right] \\
		=&\mathbb{E}\left[\left(r_i-\boldsymbol{a}_i^{\top} \boldsymbol{p}^{*}\right)\left(I\left(r_i>\boldsymbol{a}_i^{\top} \boldsymbol{p}^{*}\right)-I\left(r_i>\boldsymbol{a}_i^{\top} \boldsymbol{p}\right)\right)\right] \\
		=&\mathbb{E}\left[\left(\boldsymbol{a}_i^{\top} \boldsymbol{p}^{*}-r_i\right) I\left(\boldsymbol{a}_i^{\top} \boldsymbol{p}^{*} \geq r>\boldsymbol{a}^{\top} \boldsymbol{p}\right)\right]+\mathbb{E}\left[\left(r_i-\boldsymbol{a}_i^{\top} \boldsymbol{p}^{*}\right) I\left(\boldsymbol{a}_i^{\top} \boldsymbol{p}^{*}<r \leq \boldsymbol{a}_i^{\top} \boldsymbol{p}\right)\right] \\&\geq 0.
	\end{aligned}
	$$
	Hence, naturally we have the maximum occurs at the optimal value $\boldsymbol{p}^*$. Finally, we can show that the last quantity of the above difference can be bounded by, using the same proof as from \cite{5} Lemma 3: 
	$$
	\begin{aligned}
		&g_i\left(\boldsymbol{p}^{*}\right)-g_i(\boldsymbol{p})\\
	&	\leq \mathbb{E}\left[\left(\boldsymbol{a}_i^{\top} \boldsymbol{p}^{*}-r_i\right) I\left(\boldsymbol{a}_i^{\top} \boldsymbol{p}^{*} \geq r>\boldsymbol{a}^{\top} \boldsymbol{p}\right)\right]+\mathbb{E}\left[\left(r_i-\boldsymbol{a}_i^{\top} \boldsymbol{p}^{*}\right) I\left(\boldsymbol{a}_i^{\top} \boldsymbol{p}^{*}<r \leq \boldsymbol{a}_i^{\top} \boldsymbol{p}\right)\right]\\
		& \leq \mu \bar{a}^{2}\left\|\boldsymbol{p}^{*}-\boldsymbol{p}\right\|_{2}^{2}.
		\end{aligned}
	$$

\end{proof}
\subsection{Proof of Theorem \ref{R2}}
	\begin{proof}
	First, we split the accumulated rewards into three categories:
	$$
	\mathbb{E} R_{n}(\boldsymbol{\pi})$$$$=\mathbb{E}\left[\sum_{t=1}^{\tau_{\bar{a}}}\left(r_{t} x_{t}+\boldsymbol{d}^{\top} \boldsymbol{p}^{*}-\boldsymbol{a}_{t}^{\top} \boldsymbol{p}^{*} x_{t}\right)\right]+\mathbb{E}\left[\sum_{t=\tau_{\bar{a}}+1}^{n}\left(r_{t} x_{t}+\boldsymbol{d}^{\top} \boldsymbol{p}^{*}-\boldsymbol{a}_{t}^{\top} \boldsymbol{p}^{*} x_{t}\right)\right]-\mathbb{E}\left[\boldsymbol{b}_{n}^{\top} \boldsymbol{p}^{*}\right],
	$$
	where we here use the same derivation as in \cite{5}. Let us analyze the first portion: exchanging the summation and integration, and applying the tower property:
	
	$$\begin{aligned}
		&		\mathbb{E}\left[\sum_{t=1}^{\tau_{\bar{a}}}\left(r_{t} x_{t}+\boldsymbol{d}^{\top} \boldsymbol{p}^{*}-\boldsymbol{a}_{j}^{\top} \boldsymbol{p}^{*} x_{t}\right)\right]\\
		&=\sum_{t=1}^{n} \mathbb{E}\left[\mathbb{E}\left[\left(r_{t} x_{t}+\boldsymbol{d}^{\top} \boldsymbol{p}^{*}-\boldsymbol{a}_{t}^{\top} \boldsymbol{p}^{*} x_{j}\right) I\left(\tau_{\tilde{a}} \geq t\right) \mid \boldsymbol{b}_{t-1}, \mathcal{H}_{t-1}\right]\right]\\
		&=\mathbb{E}\sum_{t=1}^{\tau_{\tilde{a}}} \left[g_t\left(\boldsymbol{p}_{t}\right)\right]
	\end{aligned}$$
	
	For the second term, we use the same bound as in \cite{5} for it is independent of the regenerative process:
	$$
	\begin{aligned}
		\mathbb{E}\left[\sum_{t=\tau_{\bar{a}}+1}^{n}\left(r_{t} x_{t}+\boldsymbol{d}^{\top} \boldsymbol{p}^{*}-\boldsymbol{a}_{t}^{\top} \boldsymbol{p}^{*} x_{t}\right)\right] & \geq \mathbb{E}\left[\sum_{t=\tau_{\bar{a}}+1}^{n}\left(r_{t} x_{t}-\boldsymbol{a}_{t}^{\top} \boldsymbol{p}^{*} x_{t}\right)\right] \\
		& \geq-\mathbb{E}\left[n-\tau_{\bar{a}}\right] \cdot\left(\bar{r}+\frac{\bar{r} \bar{a}}{\underline{d}}\right).
	\end{aligned}
	$$
	If we combine those results, we get 
	$$
	\mathbb{E} R_{n}(\boldsymbol{\pi}) \geq \mathbb{E}\left[\sum_{j=1}^{\tau_{\bar{a}}} g\left(\boldsymbol{p}_{j}\right)\right]-\mathbb{E}\left[n-\tau_{\bar{a}}\right] \cdot\left(\bar{r}+\frac{\bar{r} \bar{a}}{\underline{d}}\right)-\mathbb{E}\left[\frac{\bar{r}}{\bar{d}} \cdot \sum_{i \in I_{B}} b_{i n}\right].
	$$
	Hence, we have 
	$$
	\mathbb{E} R_{n}(\boldsymbol{\pi}) \geq \mathbb{E}\left[\sum_{j=1}^{\tau_{\bar{a}}} g\left(\boldsymbol{p}_{j}\right)\right]-\mathbb{E}\left[n-\tau_{\bar{a}}\right] \cdot\left(\bar{r}+\frac{\bar{r} \bar{a}}{\underline{d}}\right)-\mathbb{E}\left[\frac{\bar{r}}{\bar{d}} \cdot \sum_{i \in I_{B}} b_{i n}\right].
	$$
	Finally, we take the difference 
	$$
	\mathbb{E} R_{n}^{*}-\mathbb{E} R_{n}(\boldsymbol{\pi}) \leq \mathbb{E}\left[\sum_{j=1}^{\tau_{\bar{a}}} \mu \bar{a}^{2}\left\|\boldsymbol{p}_{j}-\boldsymbol{p}^{*}\right\|_{2}^{2}\right]+\mathbb{E}\left[n-\tau_{\bar{a}}\right] \cdot\left(\bar{r}+\frac{\bar{r} \bar{a}}{\underline{d}}\right)+\mathbb{E}\left[\frac{\bar{r}}{\bar{d}} \cdot \sum_{i \in I_{B}} b_{i n}\right]
	$$
	By choosing $$
	K=\max \left\{\mu \bar{a}^{2}, \bar{r}+\frac{\bar{r} \bar{a}}{\underline{d}}, \frac{\bar{r}}{\underline{d}}\right\}
	$$ the proof is complete. 
\end{proof}

\subsection{Proof of Theorem \ref{Al1}}
	\begin{proof}
	Since from \ref{6.4} with the fact that $\boldsymbol{p}_{t}=\boldsymbol{p}^{*}$,
	$$
	\mathbb{E}\left[\sum_{t=1}^{n}\left\|\boldsymbol{p}_{t}-\boldsymbol{p}^{*}\right\|_{2}^{2}\right]=0, 
	$$
	the order of regret is bounded by $$
	\left(n-\tau_{\bar{a}}\right)+\sum_{i \in I_{B}} b_{i n}.
	$$
	Let us define 
	$$
	\tau_{\bar{a}}^{i}=\min \{n\} \cup\left\{t \geq i: \sum_{j=1}^{t} a_{i j} I\left(r_{j}>\boldsymbol{a}_{j}^{\top} \boldsymbol{p}^{*}\right)>n d_{i}-\bar{a}\right\},
	$$
	which denotes the stopping time where for ith resource we can no longer accept a large order. Moreover, since we have a fixed dual optimal $\boldsymbol{p}^{*}$, conditioned on the event that we have at least one complete regenerative cycle, the consumption rate cannot be super-linear
	$$
	\mathbb{E}\left[\sum_{j=1}^{t} a_{i j} I\left(r_{j}>\boldsymbol{a}_{j}^{\top} \boldsymbol{p}^{*}\right)\mid t>\tau_1 \right] \leq 2t d_i.
	$$
	Similarly, we have 
	$$
	\operatorname{Var}\left[\sum_{j=1}^{t} a_{i j} I\left(r_{j}>\boldsymbol{a}_{j}^{\top} \boldsymbol{p}^{*}\right)\mid t>\tau_1\right] \leq \bar{a}^2T^2 t.
	$$
	In the rest of the proof, we assume it is always the case that $t>\tau_1$ almost surely, for otherwise the period is too small to be conclusive. Then, with those conditions, we can use directly the results from \cite{5}: for some large $K$,
	$$
	\left(n-\tau_{\bar{a}}\right)+\sum_{i \in I_{B}} b_{i n}\leq Km\sqrt{n}.
	$$
	and 
	$$
	\mathbb{E}\left[b_{i n}\right] \leq \bar{a} T^{2} \sqrt{n}.
	$$
	After combining the results, we get the inequality. 
	
	The only part where the we have a difference is that it is not true $d_{i}=\mathbb{E}\left[a_{i j} I\left(r_{j}>\boldsymbol{a}_{j}^{\top} \boldsymbol{p}^{*}\right)\right]$ for a binding constraint. So in the proof 
	$$
	\mathbb{E}\left[b_{i n}\right]\leq \sqrt{\mathbb{E}\left[\left|\sum_{j=1}^{n}\left(d_{i}-a_{i j} I\left(r_{j}>\boldsymbol{a}_{j}^{\top} \boldsymbol{p}^{*}\right)\right)\right|^{2}\right]}
	$$
	the right-hand term may not be bounded by the simple sample variance. However, when we have at least one complete full period as we have assumed, then $d_{i}=\mathbb{E}\left[a_{i j} I\left(r_{j}>\boldsymbol{a}_{j}^{\top} \boldsymbol{p}^{*}\right)\right]+o(n)$. Also, since $\mathbb{E}\left[\sum_{j=1}^{t} a_{i j} I\left(r_{j}>\boldsymbol{a}_{j}^{\top} \boldsymbol{p}^{*}\right)\right]\sim O(t)$, the final term becomes 
	$$\sqrt{\mathbb{E}\left[\left|\sum_{j=1}^{n}\left(d_{i}-a_{i j} I\left(r_{j}>\boldsymbol{a}_{j}^{\top} \boldsymbol{p}^{*}\right)\right)\right|^{2}\right]}\leq 
	\sqrt{\operatorname{Var}\left[\sum_{j=1}^{n} a_{i j} I\left(r_{j}>\boldsymbol{a}_{j}^{\top} \boldsymbol{p}^{*}\right)\right]+O(t)}\sim O(\sqrt{t}).
	$$
	After taking care of this small difference, we can conclude the proof.
\end{proof}

\subsection{Proof of Theorem \ref{7.2}}
	\begin{proof}
	Since we know the order of the regret is bounded by the order of the following items from \ref{6.4}:
	$$
	\mathbb{E}\left[\sum_{t=1}^{\tau_{\bar{a}}}\left\|\boldsymbol{p}_{t}-\boldsymbol{p}^{*}\right\|_{2}^{2}+\left(n-\tau_{\bar{a}}\right)+\sum_{i \in I_{B}} b_{i n}\right],
	$$
	it suffices to show that each item is $O(\sqrt{n} \log n)$. Let us first denote the stopping when certain resource i runs out:
	$$
	\tau_{\bar{a}}^{i}=\min \{n\} \cup\left\{t \geq 1: \sum_{j=1}^{t} a_{i j} I\left(r_{j}>\boldsymbol{a}_{j}^{\top} \boldsymbol{p}_{j}\right)>n d_{i}-\bar{a}\right\}.
	$$
	From our Regenerative Dual Convergence Theorem, we have for some large constant $C$,
	$$
	\mathbb{E}\left\|\boldsymbol{p}_{t_{k}}-\boldsymbol{p}^{*}\right\|_{2}^{2} \leq \frac{C m}{t_{k}} \log \log t_{k}.
	$$
	Since we have 
	$$
	\mathbb{E}\left[\sum_{j=1}^{\tau_{\bar{a}}}\left\|\boldsymbol{p}_{j}-\boldsymbol{p}^{*}\right\|_{2}^{2}\right] \leq \sum_{j=1}^{n} \mathbb{E}\left\|\boldsymbol{p}_{j}-\boldsymbol{p}^{*}\right\|_{2}^{2}
	$$
	and there are at most order of $\log n$ of such updating interval $t_k$, for we are using the geometric update frequency; then the order of this item is $O(\log n \log \log n)$. Hence, the order is lower than $O(\sqrt{n} \log n)$. Now we move to bound the next item, the idle time of the algorithm when the resource is depleted. Let us observe that 
	$\tau_{\bar{a}}=\min _{i} \tau_{\bar{a}}^{i}$, so to bound this item it suffices to show for each i $\tau_{\bar{a}}^{i}$ is $O(\sqrt{n}\log n)$. By definition,
	$$
	\left\{\tau_{\bar{a}}^{i} \leq t\right\}=\left\{\sum_{j=1}^{t^{\prime}} a_{i j} I\left(r_{j}>\boldsymbol{a}_{j}^{\top} \boldsymbol{p}_{j}\right) \geq n d_{i}-\bar{a} \text { for some } 1 \leq t^{\prime} \leq t\right\}
	$$
	Similar to the previous theorem, let us compute the expectation and variance of the consumption rate: for $t>T$. When we have the optimal dual $p^*$ we have 
	$$
	\mathbb{E}\left[\sum_{j=1}^{t} a_{i j} I\left(r_{j}>\boldsymbol{a}_{j}^{\top} \boldsymbol{p}^{*}\right)\right] \leq t d_{i}+Td_i\leq 2td_i
	$$
	We denote the difference between the consumption under the optimal dual and under the approximated dual as 
	$$g_{j}(\boldsymbol{p})=\mathbb{E}\left[a_{i j} I\left(r_{j}>\boldsymbol{a}_{j}^{\top} \boldsymbol{p}\right)\right].$$
	By assumption \ref{ass 2*} b) we know this item is bounded by the approximation error:
	$$g_j(\boldsymbol{p})\leq \bar{a}^{2} \mu\left\|\boldsymbol{p}-\boldsymbol{p}^{*}\right\|_{2}.$$
	Hence, by Dual Convergence Theorem 
	$$
	\begin{aligned}
		\mathbb{E}\left[\sum_{j=1}^{t} a_{i j} I\left(r_{j}>\boldsymbol{a}_{j}^{\top} \boldsymbol{p}_{j}\right)\right] 
		& \leq \sum_{j=1}^{t}\left(\mathbb{E}\left[a_{i j} I\left(r_{j}>\boldsymbol{a}_{j}^{\top} \boldsymbol{p}_{j}\right)\right]-\mathbb{E}\left[a_{i j} I\left(r_{j}>\boldsymbol{a}_{j}^{\top} \boldsymbol{p}^{*}\right)\right]\right)+t d_{i} \\
		& \leq \bar{a}^{2} \mu \sum_{k=1}^{t_{k+1}} \sum_{j=t_{k}+1} \frac{C \sqrt{m}}{\sqrt{t_{k}}} \sqrt{\log \log t_{k}} I(j \leq t)+t d_{i} \\
		& \leq 5 C \bar{a}^{2} \mu \sqrt{m} \sqrt{t} \sqrt{\log \log t}+t d_{i}
	\end{aligned}
	$$
	For the variance, let us consider the following decomposition: 
	$$
	\begin{aligned}
		\operatorname{Var}\left[\sum_{j=1}^{t} a_{i j} I\left(r_{j}>\boldsymbol{a}_{j}^{\top} \boldsymbol{p}_{j}\right)\right]=& \mathbb{E}\left[\sum_{j=1}^{t} a_{i j} I\left(r_{j}>\boldsymbol{a}_{j}^{\top} \boldsymbol{p}_{j}\right)-\sum_{j=1}^{t} \mathbb{E}\left[a_{i j} I\left(r_{j}>\boldsymbol{a}_{j}^{\top} \boldsymbol{p}_{j}\right) \mid \boldsymbol{p}_{j}\right]\right]^{2} \\
		&+\operatorname{Var}\left[\sum_{j=1}^{t} \mathbb{E}\left[a_{i j} I\left(r_{j}>\boldsymbol{a}_{j}^{\top} \boldsymbol{p}_{j}\right) \mid \boldsymbol{p}_{j}\right]\right]
	\end{aligned}
	$$
	Let us consider the first term, and rewrite it as
	$$\mathbb{E}\left[\sum_{k=1}^{N(t)}\sum_{j=t(k)+1}^{t(k+1)} a_{i j} I\left(r_{j}>\boldsymbol{a}_{j}^{\top} \boldsymbol{p}_{j}\right)-\sum_{k=1}^{N(t)}\sum_{j=t(k)+1}^{t(k+1)}  \mathbb{E}\left[a_{i j} I\left(r_{j}>\boldsymbol{a}_{j}^{\top} \boldsymbol{p}_{j}\right) \mid \boldsymbol{p}_{j}\right]\right]^{2}+\epsilon$$
	where $t(i)$ is the time index of the ith regeneration and $\epsilon$ is the leftover term in the form
	$$\epsilon=\mathbb{E}\left[\sum_{j=t(N(t))+1}^{t}  a_{i j} I\left(r_{j}>\boldsymbol{a}_{j}^{\top} \boldsymbol{p}_{j}\right)-\sum_{j=t(N(t))+1}^{t}  \mathbb{E}\left[a_{i j} I\left(r_{j}>\boldsymbol{a}_{j}^{\top} \boldsymbol{p}_{j}\right) \mid \boldsymbol{p}_{j}\right]\right]^{2}<K$$
	where $K$ is a large constant. In particular, each complete cycle is a martingale difference:
	$$
	\left[\sum_{k=1}^{N(t)}\sum_{j=t(k)+1}^{t(k+1)} a_{i j} I\left(r_{j}>\boldsymbol{a}_{j}^{\top} \boldsymbol{p}_{j}\right)-\sum_{k=1}^{N(t)}\sum_{j=t(k)+1}^{t(k+1)}  \mathbb{E}\left[a_{i j} I\left(r_{j}>\boldsymbol{a}_{j}^{\top} \boldsymbol{p}_{j}\right) \mid \boldsymbol{p}_{j}\right]\right]$$
	has mean zero conditioned on the $\left\{\mathcal{H}_{t}\right\}_{t=1}^{n}$ with finite moment. In addition, the $\epsilon$ is by definition independent of the previous cycles. Hence, we have, 
	$$
	\begin{aligned}
		& \mathbb{E}\left[\sum_{j=1}^{t} a_{i j} I\left(r_{j}>\boldsymbol{a}_{j}^{\top} \boldsymbol{p}_{j}\right)-\sum_{j=1}^{t} \mathbb{E}\left[a_{i j} I\left(r_{j}>\boldsymbol{a}_{j}^{\top} \boldsymbol{p}_{j}\right) \mid \boldsymbol{p}_{j}\right]\right]^{2} \\
		=& \sum_{j=1}^{t} \mathbb{E}\left[a_{i j} I\left(r_{j}>\boldsymbol{a}_{j}^{\top} \boldsymbol{p}_{j}\right)-\mathbb{E}\left[a_{i j} I\left(r_{j}>\boldsymbol{a}_{j}^{\top} \boldsymbol{p}_{j}\right) \mid \boldsymbol{p}_{j}\right]\right]^{2}+K^2 \leq K_a t
	\end{aligned}
	$$
	where we take $t$ to be sufficiently large and $K_a$ is a large constant. For the second term, we can derive a bound using the regenerative dual convergence theorem as 
	$$
	\operatorname{Var}\left[\sum_{j=1}^{t} \mathbb{E}\left[a_{i j} I\left(r_{j}>\boldsymbol{a}_{j}^{\top} \boldsymbol{p}_{j}\right) \mid \boldsymbol{p}_{j}\right]\right]\leq C'  t \log t \log \log t
	$$
	for some large constant $C'$. So if we combine everything, we would have that 
	$$
	\operatorname{Var}\left[\sum_{j=1}^{t} a_{i j} I\left(r_{j}>\boldsymbol{a}_{j}^{\top} \boldsymbol{p}_{j}\right)\right] \leq C'' t \log t \log \log t
	$$
	for some large constant $C''$. Hence, this would allow us to derive the following inequality,
	\begin{equation}
		\begin{aligned}
			&\mathbb{P}\left(\sum_{j=1}^{t^{\prime}} a_{i j} I\left(r_{j}>\boldsymbol{a}_{j}^{\top} \boldsymbol{p}_{j}\right) \geq n d_{i}-\bar{a} \text { for some } 1 \leq t^{\prime} \leq t\right) \\
			&=\mathbb{P}\left(\sum_{j=1}^{t^{\prime}} a_{i j} I\left(r_{j}>\boldsymbol{a}_{j}^{\top} \boldsymbol{p}_{j}\right)-\mathbb{E}\left[\sum_{j=1}^{t^{\prime}} a_{i j} I\left(r_{j}>\boldsymbol{a}_{j}^{\top} \boldsymbol{p}_{j}\right)\right] \geq n d_{i}-\bar{a}-\mathbb{E}\left[\sum_{j=1}^{t^{\prime}} a_{i j} I\left(r_{j}>\boldsymbol{a}_{j}^{\top} \boldsymbol{p}_{j}\right)\right]\right. \\
			&\text { for some } \left.1 \leq t^{\prime} \leq t\right) \\
			&\leq \mathbb{P}\left(\sum_{j=1}^{t^{\prime}} a_{i j} I\left(r_{j}>\boldsymbol{a}_{j}^{\top} \boldsymbol{p}_{j}\right)-\mathbb{E}\left[\sum_{j=1}^{t^{\prime}} a_{i j} I\left(r_{j}>\boldsymbol{a}_{j}^{\top} \boldsymbol{p}_{j}\right)\right] \geq(n-t) d_{i}-C'' \sqrt{t} \sqrt{\log \log t}\right. \\
			&\text { for some } \left.1 \leq t^{\prime} \leq t\right)\\
			&\leq \mathbb{P}\left(\sum_{j=1}^{t^{\prime}} a_{i j} I\left(r_{j}>\boldsymbol{a}_{j}^{\top} \boldsymbol{p}_{j}\right)-\mathbb{E}\left[\sum_{j=1}^{t^{\prime}} a_{i j} I\left(r_{j}>\boldsymbol{a}_{j}^{\top} \boldsymbol{p}_{j}\right)\right] \geq(n-t) d_{i}-C'' \sqrt{t} \sqrt{\log \log t}\right. \\
			&\text { for some } \left.1 \leq t^{\prime} \leq t(N(t)+1)\right)	
		\end{aligned}
	\end{equation}
	Now, we can construct a martingale as 
	$$
	M_{k}=\sum_{j=1}^{t(k)} a_{i j} I\left(r_{j}>\boldsymbol{a}_{j}^{\top} \boldsymbol{p}_{j}\right)-\mathbb{E}\left[\sum_{j=1}^{t(k)} a_{i j} I\left(r_{j}>\boldsymbol{a}_{j}^{\top} \boldsymbol{p}_{j}\right)\right]
	$$
	which is the sum of $k$ regenerative cycles. Hence, using Doob's maximum inequality, we have 7.1 is bounded by 
	$$
	\leq \frac{\operatorname{Var}\left[\sum_{j=1}^{t} a_{i j} I\left(r_{j}>\boldsymbol{a}_{j}^{\top} \boldsymbol{p}_{j}\right)\right]}{\left((n-t) d_{i}-\bar{a}-c'' \sqrt{t} \sqrt{\log \log t}\right)^{2}}
	$$
	$$\leq \frac{Ct\log t\log \log t}{\left((n-t) d_{i}-\bar{a}-c'' \sqrt{t} \sqrt{\log \log t}\right)^{2}}.$$
	As a result, 
	$$
	\begin{aligned}
		\mathbb{E}\left[n-\tau_{\bar{a}}^{i}\right] & \leq \sum_{t=1}^{n} \mathbb{P}\left(\tau_{\bar{a}}^{i} \leq t\right) \\
		&=\sum_{t=1}^{n} \mathbb{P}\left(\sum_{j=1}^{t^{\prime}} a_{i j} I\left(r_{j}>\boldsymbol{a}_{j}^{\top} \boldsymbol{p}_{j}\right) \geq n d_{i}-\bar{a} \text { for some } 1 \leq t^{\prime} \leq t\right) \\
		& \leq n-n_{0}+\sum_{t=1}^{n_{0}}\frac{Ct\log t\log \log t}{\left((n-t) d_{i}-\bar{a}-c'' \sqrt{t} \sqrt{\log \log t}\right)^{2}}\\
		& \leq C^{*}\sqrt{n} \log n
	\end{aligned}
	$$
	Hence, we have this term has order of $O(\sqrt{n}\log n)$. Finally, we can bound the last term by 
	$$
	b_{i n}=\left(n d_{i}-\sum_{j=1}^{n} a_{i j} I\left(r_{j}>\boldsymbol{a}_{j}^{\top} \boldsymbol{p}_{j}\right)\right)^{+} \leq\left|n d_{i}-\sum_{j=1}^{n} a_{i j} I\left(r_{j}>\boldsymbol{a}_{j}^{\top} \boldsymbol{p}_{j}\right)\right|,
	$$
	which is bounded, by Jensen's inequality to $f(x)=x^2$ and by Cauchy-Schwartz
	$$
	=\sqrt{\left(\mathbb{E}\left[n d_{i}-\sum_{j=1}^{n} a_{i j} I\left(r_{j}>\boldsymbol{a}_{j}^{\top} \boldsymbol{p}_{j}\right)\right]\right)^{2}+\operatorname{Var}\left[n d_{i}-\sum_{j=1}^{n} a_{i j} I\left(r_{j}>\boldsymbol{a}_{j}^{\top} \boldsymbol{p}_{j}\right)\right]}.
	$$
	From our previous computation, this has order $O(\sqrt{n\log \log n})$. Hence, if we sum all three error terms 
	$$
	\mathbb{E}\left[\sum_{t=1}^{\tau_{\bar{a}}}\left\|\boldsymbol{p}_{t}-\boldsymbol{p}^{*}\right\|_{2}^{2}+\left(n-\tau_{\bar{a}}\right)+\sum_{i \in I_{B}} b_{i n}\right]\sim O(\sqrt{n}\log n)
	$$
	which completes the proof.
\end{proof}
	\subsection{First Appendix}
	This section provides relevant materials from the paper \cite{5}. All lemmas are listed in the context of the section Traditional Online Linear Programming. 
	\begin{lemma}
		For any $\boldsymbol{p} \geq \mathbf{0}$, we have the following identity,
		$$
		f(\boldsymbol{p})-f\left(\boldsymbol{p}^{*}\right)=\underbrace{\nabla f\left(\boldsymbol{p}^{*}\right)\left(\boldsymbol{p}-\boldsymbol{p}^{*}\right)}_{\text {First-order }}+\underbrace{\mathbb{E}\left[\int_{\boldsymbol{a}^{\top} \boldsymbol{p}}^{\boldsymbol{a}^{\top} \boldsymbol{p}^{*}}\left(I(r>v)-I\left(r>\boldsymbol{a}^{\top} \boldsymbol{p}^{*}\right)\right) d v\right]}_{\text {Second-order }} .
		$$
		where the expectation is taken with respect to $(r, \boldsymbol{a}) \sim \mathcal{P}$.\label{A 1}
	\end{lemma}
	We have applied this result in Proposition 4.4 because the item-wise equality of the last equation of 4.4 is equivalent to the item-wise equality of this lemma. 
	\begin{lemma}
		Suppose we have 
		$$
		\frac{\lambda \lambda_{\min }}{2}\left\|\boldsymbol{p}-\boldsymbol{p}^{*}\right\|_{2}^{2} \leq f(\boldsymbol{p})-f\left(\boldsymbol{p}^{*}\right)-\nabla f\left(\boldsymbol{p}^{*}\right)\left(\boldsymbol{p}-\boldsymbol{p}^{*}\right) \leq \frac{\mu \bar{a}^{2}}{2}\left\|\boldsymbol{p}-\boldsymbol{p}^{*}\right\|_{2}^{2},
		$$
		The optimal solution is unique. \label{A 2}
	\end{lemma}
	This lemma is a part of Proposition 2 of \cite{5}. Though it is a standard argument through the sub-gradient and the optimality condition, in the appendix of \cite{5} there is a self-contained proof. 
	\begin{lemma}
		For any $\boldsymbol{p} \in \mathbb{R}^{m}$, we have the following identity,
		$$
		f_{n}(\boldsymbol{p})-f_{n}\left(\boldsymbol{p}^{*}\right)=\frac{1}{n} \sum_{j=1}^{n} \phi\left(\boldsymbol{p}^{*}, \boldsymbol{u}_{j}\right)^{\top}\left(\boldsymbol{p}-\boldsymbol{p}^{*}\right)+\frac{1}{n} \sum_{j=1}^{n} \int_{\boldsymbol{a}_{j}^{\top} \boldsymbol{p}}^{\boldsymbol{a}_{j}^{\top} \boldsymbol{p}^{*}}\left(I\left(r_{j}>v\right)-I\left(r_{j}>\boldsymbol{a}_{j}^{\top} \boldsymbol{p}^{*}\right)\right) d v
		$$
	\end{lemma}
	We have applied this lemma in proposition 4.6 for the same reason as we applied lemma \ref{A 1} to proposition 4.4. \\

	\begin{lemma}
		If we define $\boldsymbol{M}_{n} :=\frac{1}{n} \sum_{j=1}^{n} \boldsymbol{a}_{j} \boldsymbol{a}_{j}^{\top}$
		and $\boldsymbol{M} =\mathbb{E}\left[\boldsymbol{a}_{j} \boldsymbol{a}_{j}^{\top}\right]$, then the event $\mathcal{E}_{0}=\left\{\lambda_{\min }\left(\boldsymbol{M}_{n}\right) \leq \frac{\lambda_{\min }}{2}\right\}$ has the probability 
		$$
		\mathbb{P}\left(\mathcal{E}_{0}\right)\leq m \cdot \exp \left(\frac{-n \lambda_{\min }^{2}}{4{\bar{a}}^{2}}\right)
		$$ \label{A 5}
	\end{lemma}
	This lemma is the step 1 of Proposition 4 in \cite{5}.\\
	
	\begin{lemma}
		We consider the following partition : the partition, whose motivation can be found at \cite{5}, on $\bar{\Omega}:=\left\{\boldsymbol{p} \in \mathbb{R}^{m} \mid\left\|\boldsymbol{p}-\boldsymbol{p}^{*}\right\|_{\infty} \leq \frac{\bar{r}}{\underline{d}}\right\}$ is the following
		$$
		\bar{\Omega}=\bigcup_{k=1}^{N} \bigcup_{l=1}^{l_{k}} \Omega_{k l} .
		$$
		Each $\Omega_{k l}$ is obtained from the following procedure: let $
		\bar{\Omega}_{k}=\left\{\boldsymbol{p} \in \mathbb{R}^{m} \mid\left\|\boldsymbol{p}-\boldsymbol{p}^{*}\right\|_{\infty} \leq q^{k} \frac{\bar{r}}{\underline{d}}\right\}
		$ be the large rectangles, and  $\bar{\Omega}_{k}, k\in \{1,2,\dots,N\}$, where $N$ is determined later, forms the onion-like overlapping rectangles. Then, we partition out the overlapping portion and define the $k$-th layer as $
		\bar{\Omega}_{k-1} \backslash \bar{\Omega}_{k}
		$. On each layer, we further partition the set into disjoint cubes $\left\{\Omega_{k l}\right\}_{l=1}^{l_{k}}$ with edges of length $(1-q) q^{k-1} \frac{\bar{r}}{d}$ for $k=1, \ldots, N-1$ and $l=1, \ldots, l_{k}$. The center cube is  $\bar{\Omega}_{N}=\Omega_{N 1}$ with edge of $q^{N} \frac{\bar{r}}{\underline{d}}$ and $l_{N}=1$. For convenience, we can adjust the value of $q$ to allow for integer number of cubes. Let $\boldsymbol{p}_{k l}$ be the center of the cube $\Omega_{k l}, \underline{\boldsymbol{p}}_{k l}$ and $\overline{\boldsymbol{p}}_{k l}$ be the points in $\Omega_{k l}$ that are closest and furthest from $\boldsymbol{p}^{*}$, respectively. That is,
		$$
		\begin{aligned}
			\underline{\boldsymbol{p}}_{k l} &=\underset{\boldsymbol{p} \in \Omega_{k l}}{\arg \min }\left\|\boldsymbol{p}-\boldsymbol{p}^{*}\right\|_{2}, \\
			\overline{\boldsymbol{p}}_{k l} &=\underset{\boldsymbol{p} \in \Omega_{k l}}{\arg \max }\left\|\boldsymbol{p}-\boldsymbol{p}^{*}\right\|_{2} .
		\end{aligned}
		$$ \label{P 1}
	\end{lemma}
	
	There are several facts about this partition, as shown in the proof of proposition 4 of \cite{5}:  \begin{lemma}
		By the construction of the partition above, we have 
		$$
		\left\|\boldsymbol{p}^{*}-\overline{\boldsymbol{p}}_{k l}\right\|_{2} \leq\left(1+\frac{\sqrt{m}(1-q)}{q}\right)\left\|\boldsymbol{p}^{*}-\boldsymbol{p}_{k l}\right\|_{2}
		$$
		and 
		$$
		\max _{\boldsymbol{p} \in \Omega_{k l}}\left\|\boldsymbol{p}-\boldsymbol{p}_{k l}\right\|_{2}\leq \frac{\sqrt{m}(1-q)}{q}\left\|\boldsymbol{p}^{*}-\overline{\boldsymbol{p}}_{k l}\right\|_{2}.
		$$\label{A 6}
	\end{lemma}
	\begin{lemma}
		With the same event $\mathcal{E}_{0}$ so defined above, we have 
		$$
		\mathbb{E}\left[\frac{1}{n} \sum_{j=1}^{n} \int_{\boldsymbol{a}_{j}^{\top} \boldsymbol{p}_{k l}}^{\boldsymbol{a}_{j}^{\top} \boldsymbol{p}^{*}}\left(I\left(r_{j}>v\right)-I\left(r_{j}>\boldsymbol{a}_{j}^{\top} \boldsymbol{p}^{*}\right)\right) d v \mid \boldsymbol{a}_{1}, \ldots, \boldsymbol{a}_{n}, \mathcal{E}_{0}\right]\geq \frac{\lambda \lambda_{\min }}{4}\left\|\boldsymbol{p}^{*}-\boldsymbol{p}_{k l}\right\|_{2}^{2}
		$$\label{A 7}
	\end{lemma}
	
	\begin{lemma}
		\begin{equation}
			\mathbb{E}\left[\Gamma_{k l}\left(r_{j}, \boldsymbol{a}_{j}\right) \mid \boldsymbol{a}_{1}, \ldots, \boldsymbol{a}_{n}\right]\leq \mu \bar{a}^{2}\left\|\boldsymbol{p}^{*}-\overline{\boldsymbol{p}}_{k l}\right\|_{2} \max _{\boldsymbol{p} \in \Omega_{k l}}\left\|\boldsymbol{p}-\boldsymbol{p}_{k l}\right\|_{2}
		\end{equation}
		
		\begin{equation}
			\left|\Gamma_{k l}\left(r_{j}, \boldsymbol{a}_{j}\right)\right| \leq \max _{\boldsymbol{p} \in \Omega_{k l}}\left|\boldsymbol{a}_{j}^{\top} \boldsymbol{p}-\boldsymbol{a}_{j}^{\top} \boldsymbol{p}_{k l}\right| \leq \bar{a} \max _{\boldsymbol{p} \in \Omega_{k l}}\left\|\boldsymbol{p}-\boldsymbol{p}_{k l}\right\|_{2}.
		\end{equation}
		\label{A 8}
	\end{lemma}
	The inequalities are the inequalities on page 37 of \cite{5}.
	
	\begin{lemma}
		On event $\cap_{k=1}^{N} \cap_{l=1}^{l_{k}}\left(\mathcal{E}_{k l, 1}^{c} \cap \mathcal{E}_{k l, 2}^{c}\right) \cap \mathcal{E}_{0}$, we have 
		$$
		\frac{1}{n} \sum_{j=1}^{n} \int_{\boldsymbol{a}_{j}^{\top} \boldsymbol{p}}^{\boldsymbol{a}_{j}^{\top} \boldsymbol{p}^{*}}\left(I\left(r_{j}>v\right)-I\left(r_{j}>\boldsymbol{a}_{j}^{\top} \boldsymbol{p}^{*}\right)\right) d v \geq-\epsilon^{2}-2 \epsilon \bar{a}\left\|\boldsymbol{p}^{*}-\boldsymbol{p}\right\|_{2}+\frac{\lambda \lambda_{\min }}{32}\left\|\boldsymbol{p}^{*}-\boldsymbol{p}\right\|_{2}^{2}
		$$
		for all $\boldsymbol{p} \in \bar{\Omega}$. Here $$
		N=\left\lfloor\log _{q}\left(\frac{\underline{d} \epsilon^{2}}{\bar{a} \bar{r} \sqrt{m}}\right)\right\rfloor+1,
		$$ and $$
		q=\max \left\{\frac{1}{1+\frac{1}{\sqrt{m}}}, \frac{1}{1+\frac{1}{\sqrt{m}}\left(\frac{\lambda \lambda_{\min }}{8 \mu \bar{a}^{2}}\right)^{\frac{1}{3}}}\right\}.
		$$\label{A 9}
	\end{lemma}
	
	\begin{lemma}
		The inequality
		$$
		\int_{0}^{\infty}(\exp (\log m-x \log m)) \wedge 1 \mathrm{~d} x \leq 2
		$$
		holds for all $m \geq 2$. \label{A 10}
	\end{lemma}
	\begin{lemma}
		The inequality
		$$
		\int_{0}^{\infty} 1 \wedge\left(\exp \left(-x m \log m \log \log n+m \log \left(\sqrt{m} \log \left(\frac{n}{x}\right)\right)\right)\right) \mathrm{d} x \leq 2
		$$
		holds for all $n \geq \max \{m, 3\}$ and $m \geq 2$ \label{A 11}
	\end{lemma}

	\begin{lemma}
		We have
		$$
		\left|\Gamma_{k l}\left(r_{j}, \boldsymbol{a}_{j}\right)\right| \leq \max _{\boldsymbol{p} \in \Omega_{k l}}\left|\boldsymbol{a}_{j}^{\top} \boldsymbol{p}-\boldsymbol{a}_{j}^{\top} \boldsymbol{p}_{k l}\right| \leq \bar{a} \max _{\boldsymbol{p} \in \Omega_{k l}}\left\|\boldsymbol{p}-\boldsymbol{p}_{k l}\right\|_{2}
		$$\label{A 4}
	\end{lemma}
	This lemma provides the necessary upper bound for Proposition 4.10 \ref{4.10}.

	\subsection{Second Appendix}
	This section provides additional explanations for the proofs in the main section. 
	\begin{proposition}
		We have, for some finite $M$, $$\int_{0}^{M}
		m\exp \left(\left(-\frac{1}{ \sqrt{m}} \sqrt{\epsilon^{\prime}}\right) t\right)\wedge 1 d\epsilon'\leq 30\frac{\sqrt{m}\log m}{n} 
		$$
		\label{B 1}
	\end{proposition}
	
	\begin{proof}
		By the change of variable with $\sqrt{\epsilon'}=\frac{\sqrt{m}\log m\sqrt{\epsilon}}{n}$, we have 
		$$\begin{aligned} &\int_{0}^{M} m\exp \left(\left(-\frac{1}{ \sqrt{m}} \sqrt{\epsilon^{\prime}}\right) t\right) \wedge 1d\epsilon^{\prime}\\
			&\leq \frac{\sqrt{m}\log m}{n}\left(1+\int_{1}^{\infty} \exp( -\sqrt{\epsilon}\log m+\log m)\right)d\epsilon\\
			&\leq  \frac{\sqrt{m}\log m}{n}\left(1+\int_{1}^{\infty} \exp( -\sqrt{\epsilon}\log 2+\log 2)\right)d\epsilon\\
			&\leq  30\frac{\sqrt{m}\log m}{n}
		\end{aligned}$$
		where the second to the last inequality is true because over the support of $(1,\infty)$, $\exp( -\sqrt{\epsilon}\log 2+\log 2)$ is strictly decreasing in $m$. 
		
	\end{proof}
	Similarly, since for $\epsilon\geq 1$, we have $\epsilon>{\sqrt{\epsilon}}$
	\begin{proposition}
		We have, $$\int_{0}^{M}
		m\exp \left(\left(-\frac{1}{ \sqrt{m}} {\epsilon^{\prime}}\right) t\right)\wedge 1 d\epsilon'\leq 30\frac{\sqrt{m}\log m}{n} 
		$$
		\label{B 3}
	\end{proposition}

	\begin{proposition}
		We have for some finite $C$,  $$\int_{0}^{M}
		\exp \left(- \sqrt{\epsilon^{\prime}} n\right)\cdot (2N)^m\wedge 1 d\epsilon'\leq C\cdot \frac{m\log m\log \log n}{n}
		$$\label{B 2}
	\end{proposition}
	\begin{proof}
		We know from the definition of $N$, there exists some finite $c_0$ such that
		$$
		2 N \leq c_{0} \sqrt{m} \log \left(\frac{\sqrt{m}}{\epsilon^{\prime}}\right)
		$$
		So it suffices to show there exists some finite $C'$ such that 
		$$\int_{0}^{M}
		\exp \left(- \sqrt{\epsilon^{\prime}} n\right)\cdot \left(\sqrt{m} \log \left(\frac{\sqrt{m}}{\epsilon^{\prime}}\right)\right)^m\wedge 1 d\epsilon'\leq C'\cdot \frac{m\log m\log \log n}{n}.
		$$
		Indeed, by a change of variable with $\sqrt{\epsilon'}=\frac{m\log m\log \log n}{n}\sqrt{\epsilon}$, we have
		$$\begin{aligned}
			& \int_{0}^{M}\exp \left(- \sqrt{\epsilon^{\prime}} n\right)\cdot \left(\sqrt{m} \log \left(\frac{\sqrt{m}}{\epsilon^{\prime}}\right)\right)^m\wedge 1 d\epsilon'\\
			&\leq \frac{m\log m\log \log n}{n}\int_{0}^{\infty} 1 \wedge \exp(-\sqrt{\epsilon}m\log m\log\log n+m\log \left(\sqrt{m}\log\left(\frac{n\sqrt{m}}{m(\log\log m)\sqrt{\epsilon}}\right) \right))d\epsilon\\
			&\leq \frac{m\log m\log \log n}{n}\int_{0}^{\infty} 1 \wedge \exp(-\sqrt{\epsilon}m\log m\log\log n+m\log \left(\sqrt{m}\log\left(\frac{n}{\sqrt{\epsilon}}\right) \right))d\epsilon\\
			&\leq \frac{m\log m\log \log n}{n}\left[1+\int_{1}^{\infty}  \exp(-\sqrt{\epsilon}m\log m\log\log n+m\log \left(\sqrt{m}\log\left(\frac{n}{\sqrt{\epsilon}}\right) \right))d\epsilon\right]\\
			&\leq \frac{m\log m\log \log n}{n}\left[1+\int_{1}^{\infty}  \exp(-\sqrt{\epsilon}m\log m\log\log n+m\log \left(\sqrt{m}\log n \right))d\epsilon\right]\\
			&\leq \frac{m\log m\log \log n}{n}\left[1+\int_{1}^{\infty}  \exp(-\sqrt{\epsilon}2\log 2\log\log 2+2\log \left(\sqrt{2}\log 3 \right))d\epsilon\right]\\
			&\leq 30\cdot \frac{m\log m\log \log n}{n}\end{aligned}
	$$
	\end{proof}
	And similarly, 
	\begin{proposition}
		We have, for some finite $C$, $$\int_{0}^{M}
		\exp \left(-{\epsilon^{\prime}} n\right)\cdot (2N)^m\wedge 1 d\epsilon'\leq C\cdot \frac{m\log m\log \log n}{n}
		$$\label{B 4}
	\end{proposition}

	\section*{Acknowledgement}
I am grateful for the continuous supports from Yinyu Ye, without whom this paper cannot be made possible. 
		\bibliographystyle{apalike}
	
	\bibliography{OLPcite}

\end{document}